\documentclass[psamsfonts]{amsart}
  
\usepackage{amssymb,amsfonts}
\usepackage[margin=1.5in]{geometry}
\usepackage{mathtools}
\usepackage{color}
\usepackage[color,matrix,all,arc]{xy}
\usepackage{enumerate}
\usepackage{mathrsfs}
\usepackage[mathscr]{euscript}
\usepackage{graphicx}
\usepackage{float}
\usepackage{url} 
\usepackage{bbm}
\usepackage{soul}
\usepackage{xcolor}
\definecolor{allrefcolors}{rgb}{.05,.45,.6}
\usepackage[pagebackref,linktocpage=true,colorlinks=true,allcolors=allrefcolors,bookmarksopen,bookmarksdepth=3]{hyperref}
\usepackage[noabbrev]{cleveref}
\usepackage{mathtools}
\usepackage{dsfont}
\usepackage{bbm}

\usepackage{tikz}
\usepackage{tqft}
\usepackage{pgfplots}
\usepackage[caption=false]{subfig}

\usepackage{enumitem}
\setenumerate{label=(\roman*),topsep=1pt,itemsep=1pt,partopsep=0pt,parsep=0pt}

\theoremstyle{theorem}
\newtheorem{thm}{Theorem}[section]
\newtheorem{prop}[thm]{Proposition}
\newtheorem{lem}[thm]{Lemma}
\newtheorem{cor}[thm]{Corollary}

\theoremstyle{definition}
\newtheorem{defn}[thm]{Definition}
\newtheorem{con}[thm]{Construction}
\newtheorem{exmp}[thm]{Example}
\newtheorem{notn}[thm]{Notation}

\theoremstyle{remark}
\newtheorem{rem}[thm]{Remark}

\crefname{thm}{Theorem}{Theorems}
\crefname{prop}{Proposition}{Propositions}
\crefname{lem}{Lemma}{Lemmas}
\crefname{cor}{Corollary}{Corollaries}

\crefname{defn}{Definition}{Definitions}
\crefname{con}{Construction}{Constructions}
\crefname{exmp}{Example}{Examples}
\crefname{notn}{Notation}{Notations}
\crefname{asmpt}{Assumption}{Assumptions}

\crefname{rem}{Remark}{Remarks}
\crefname{warn}{Warning}{Warnings}

\crefname{section}{Section}{Sections}
\crefname{subsection}{Subsection}{Subsections}

\crefname{sec}{Section}{Sections}
\crefname{subsec}{Subsection}{Subsections}
\crefname{eqn}{Equation}{Equations}
\crefname{part}{Part}{Parts}

\newcommand{\IC}{\mathbb{C}}
\newcommand{\ID}{\mathbb{D}}

\newcommand{\IH}{\mathbb{H}}

\newcommand{\IK}{\mathbb{K}}

\newcommand{\IN}{\mathbb{N}}

\newcommand{\IP}{\mathbb{P}}
\newcommand{\IQ}{\mathbb{Q}}
\newcommand{\IR}{\mathbb{R}}

\newcommand{\IZ}{\mathbb{Z}}

\newcommand{\Ione}{\mathbbm{1}}

\newcommand{\sA}{\mathcal{A}}
\newcommand{\sB}{\mathcal{B}}
\newcommand{\sC}{\mathcal{C}}
\newcommand{\sD}{\mathcal{D}}

\newcommand{\sF}{\mathcal{F}}

\newcommand{\sH}{\mathcal{H}}
\newcommand{\sI}{\mathcal{I}}
\newcommand{\sJ}{\mathcal{J}}
\newcommand{\sK}{\mathcal{K}}

\newcommand{\sM}{\mathcal{M}}

\newcommand{\sO}{\mathcal{O}}

\newcommand{\sR}{\mathcal{R}}

\newcommand{\fo}{\mathfrak{o}}
\newcommand{\fc}{\mathfrak{c}}

\newcommand{\gap}{\hspace{15pt}}

\newcommand{\wt}[1]{\widetilde{#1}}

\DeclareMathOperator{\End}{End}

\DeclareMathOperator{\colim}{colim}

\DeclareMathOperator{\Hor}{Hor}
\DeclareMathOperator{\Ver}{Ver}

\DeclareMathOperator{\ind}{ind}

\DeclareMathOperator{\Ch}{Ch}

\DeclareMathOperator{\op}{op}

\DeclareMathOperator{\val}{val}

\DeclareMathOperator{\Tel}{Tel}

\DeclareMathOperator{\intrr}{int}

\DeclareMathOperator{\dist}{dist}
\DeclareMathOperator{\dg}{dg}

\begin{document}

\title{Sections and unirulings of families over $\IP^1$}

\author{Alex Pieloch}

\address{Department of Mathematics\\ Columbia University\\ 2990 Broadway \\
New York, NY 10027}
\email{pieloch@math.columbia.edu}

\thanks{The author was supported by the National Science Foundation Graduate Student Fellowship
Program through NSF grant DGE 16-44869.  The author was also partially supported by NSF grants DMS-1609148, and DMS1564172, and by the Simons Foundation through its ``Homological Mirror Symmetry'' Collaboration grant.}

\begin{abstract}
We consider morphisms $\pi: X \to \IP^1$ of smooth projective varieties over $\IC$.  We show that if $\pi$ has at most one singular fibre, then $X$ is uniruled and $\pi$ admits sections.  We reach the same conclusions, but with genus zero multisections instead of sections, if $\pi$ has at most two singular fibres, and the first Chern class of $X$ is supported in a single fibre of $\pi$.

To achieve these result, we use action completed symplectic cohomology groups associated to compact subsets of convex symplectic domains.  These groups are defined using Pardon's virtual fundamental chains package for Hamiltonian Floer cohomology.  In the above setting, we show that the vanishing of these groups implies the existence of unirulings and (multi)sections.
\end{abstract}

\maketitle

\tableofcontents

\section{Introduction}\label{sec:Introduction}

\subsection{Statement of Results}\label{subsec:StatementOfResults}

In this paper, we study complex projective varieties.  When we refer to a variety, we will always mean a variety over $\IC$.  The primary goal of this paper is to prove the following results.

\begin{thm}\label{thm:ActualMainTheorem}
If $\pi: \overline{M} \to \IP^1$ is a morphism of smooth projective varieties that is smooth\footnote{Here we mean \emph{smooth} in the algebro-geometric sense, that is, $\pi$ is a holomorphic submersion.} away from $\infty$, then $\overline{M}$ is uniruled and admits sections.
\end{thm}

\begin{thm}\label{thm:ActualMainTheoremCY}
If $\pi: \overline{M} \to \IP^1$ is a morphism of smooth projective varieties that is smooth away from $0$ and $\infty$ and the first Chern class of $\overline{M} \smallsetminus \pi^{-1}(\infty)$ vanishes, then $\overline{M}$ is uniruled and admits genus zero multisections.
\end{thm}

These results hold in the category of complex algebraic varieties.  Except for some select cases (discussed below), there are no algebro-geometric proofs of these results.  Our proofs are symplectic in nature.

\begin{rem}
The motivation for our results comes from Hodge theory.  By the work of Griffiths \cite{Griffiths_PeriodMapHolomorphic}, given a projective morphism of smooth varieties $\pi: X \to S$, one obtains a holomorphic map $\Phi: \widetilde{S^*} \to D,$ where $D$ is a classifying space of polarized Hodge structures, which carries a complex structure and is referred to as a period domain, and $\widetilde{S^*}$ is the universal cover of the complement in $S$ of the singular values of $\pi$.  Roughly, $\Phi$ sends a point $\widetilde{s}$ to the marked polarized Hodge structure of the fibre $X_{\widetilde{s}}$.  If $X_{\widetilde{s}}$ is a smooth genus $g$ curve, then $D$ is the Seigel upper half-space $\IH_g$, which has negative holomorphic sectional curvature.  So in this case, when $S^*$ is a curve of non-negative genus, $\Phi$ is constant and, the variation of Hodge structures of the fibres over $S^*$ is trivial.  In general, $D$ can be either positively or negatively curved.  Nevertheless, a distance decreasing principle of Griffiths \cite{Griffiths_DistanceDecreasingPrincipal} implies the analogous result in general: the variation of Hodge structures of the fibres over $S^*$ is trivial when $S^*$ is a curve of non-negative genus.

In our case, $S = \IP^1$ and $\pi$ has at most two singular values.  So the variation of Hodge structures of the fibres of $\pi$ is trivial.  It was expected that this Hodge theoretic triviality should be witnessed by algebraic cycles (that is, sections) or complex geometric features.  By using symplectic methods, we construct such algebraic cycles and constrain the actual geometry (as opposed to Hodge theory) of these families.
\end{rem}

\begin{rem}\label{rem:RelatedResults}
We point out partially related work.  Graber, Harris, and Starr show that if a proper morphism from a projective variety to a smooth curve has rationally connected general fiber, then it has a section \cite{GraberHarrisStarr_FamiliesOfRationallyConnectedVarieties}.  This was extended to algebraically closed fields with non-zero characteristic by de Jong-Starr \cite{deJongStarr_FamiliesOfRationallyConnectedVarietiesCharP}.  Our results share some overlapping cases with this work; however, our methods are entirely orthogonal to theirs, ours being symplectic not algebro-geometric.

On the symplectic side, Seidel's construction of his eponymously named representation \cite{Seidel_SeidelRepresentation} implies that if $\pi$ (as in \cref{thm:ActualMainTheorem}) is smooth over all of $\IP^1$, then $\pi$ admits a section.  Seidel's original work assumed a relationship between the first Chern class and the symplectic form of the space.  Later, McDuff introduced virtual techniques to prove the result without this assumption \cite{McDuff_VirtualSeidelRepresentation}.  Similar ideas also appear in \cite{LalondeMcDuffPolterovich_RigidityOfhamiltonians}.  Our results generalize this corollary of Seidel's and McDuff's work.  However, our work does not generalize the main results of the above papers, which are on the Seidel representation.  Finally, while our proofs and the proofs of Seidel and McDuff are both symplectic, they are nevertheless vastly different.
\end{rem}

\begin{rem}
Finally, we note: if $\pi: \overline{M} \to \IP^1$ is a morphism of smooth, projective varieties, and $\overline{M}$ is Fano, then by \cite{KollarMiyaokaMori_RationalConnectednessAndBoundednessOfFanoManifolds} $\overline{M}$ is rationally connected.  This implies that $\pi$ is actually uniruled by genus zero multisections.
\end{rem}


\subsection{Sketch of proof}

We now sketch the proofs of our main results.

\emph{Degenerating to the normal cone}:  We consider the degeneration of $\IP^1$ to the normal cone of $\{\infty\}$ in $\IP^1$, that is, the blowup
\[ \beta_0: B \coloneqq Bl_{(0,\infty)}(\IC \times \IP^1) \to \IC \times \IP^1. \]
We have a projection $\pi_{{B}}: {B} \to \IC$ by composing $\beta_0$ with the projection to $\IC$.  $\pi_B^{-1}(z)$ for $z \neq 0$ is biholomorphic to $\IP^1$, and $\pi_B^{-1}(0)$ is a union of two curves $F_0 \cup E_0$, where $E_0$ is the exceptional divisor, and $F_0$ is a curve that is biholomorphic to $\IP^1$.  As $z \to 0$, the curves $\pi_B^{-1}(z)$ converge in the Gromov topology to the nodal curve $\pi_B^{-1}(0)$.  

Now consider the degeneration of $\overline{M}$ to the normal cone of $\pi^{-1}(\infty)$ in $\overline{M}$.\footnote{Technically speaking, if $\pi^{-1}(\infty)$ is not smooth, then the total space of the degeneration to the normal cone will not be smooth.  So instead we will actually work with a resolution of this variety.}  This is a quasi-projective variety $P$ along with a map $\wt{\pi}: P \to B$.  Let $\pi_P \coloneqq \wt{\pi} \circ \pi_B$.  $\pi_P^{-1}(z)$ for $z \neq 0$ is biholomorphic to $\overline{M}$, and $\wt{\pi}$ over $\pi_B^{-1}(z)$ is identified with $\pi: \overline{M} \to \IP^1$.  $\pi_P^{-1}(0) = F \cup E$, where $F$ is a subscheme that is birational to $\overline{M}$ and $E$ is an exceptional divisor.  $\wt{\pi}$ over $F_0$ is identified with $\pi$ up to composing with a birational morphism, and $\wt{\pi}^{-1}(E_0) = E$.  Roughly, we are degenerating the family $\pi: \overline{M} \to \IP^1$ into two families that meet along a fibre.  One family lives over $E_0$, and the other family lives over $F_0$, and can be identified with our original family.  

Consider a sequence of holomorphic disks $u_z: \ID \to \pi_P^{-1}(z)$ for $z \neq 0$ with uniformly bounded energies such that
\[ \wt{\pi} \circ u_z(\ID) = \{ z \in \IP^1 \mid |z| \leq 2\} \subset \IP^1 \cong \pi_B^{-1}(z). \]
As $z \to 0$ (with an appropriate choice of symplectic form on $B$), the $\wt{\pi}\circ u_z$ converge in the Gromov topology to a nodal holomorphic curve with boundary in $\pi_B^{-1}(0)$.  One component of the holomorphic curve has image given by $F_0$ and the other component is a holomorphic disk whose boundary is completely contained in $E_0 \smallsetminus (E_0 \cap F_0)$.  Now consider the sequence $u_z$ in $P$.  A Gromov-type compactness result (see \cref{lem:FishCompactness}) implies that these holomorphic curves with boundaries degenerate to a nodal holomorphic curve with boundary in $\pi_P^{-1}(0)$.  The boundary of the limit curve (as above) is completely contained in $E\smallsetminus (E \cap F)$.  This limit curve also has an irreducible component in $F$ whose projection to $F_0$ is non-constant.  This gives the (multi)section described in our main results since the projection from $F$ to $F_0$ may be identified with $\pi$.    With this procedure, we turn (sufficiently large) holomorphic disks in $\overline{M}$ into closed holomorphic curves in $\overline{M}$.  In actuality, the $u_z$ will be genus zero curves with boundaries; however, this will not change the argument.  We will run a similar degeneration for holomorphic curves (possibly with boundaries) in $\pi_P^{-1}(z)$ that have point constraints, which will give rise to the uniruling parts of our main results.

\emph{Arranging to do Floer theory}: The above procedure reduces our task to producing holomorphic disks in the $\pi_P^{-1}(z)$. The variety $P$ is K\"{a}hler, and, thus, gives rise to a K\"{a}hler form on each $\pi_P^{-1}(z)$.  Via symplectic parallel transport, we symplectically (but not biholomorphically) identify $\pi_P^{-1}(z)$ for $z \neq 0$ with $\pi_P^{-1}(1)$.  Pushing forward the almost complex structure on $\pi_P^{-1}(z)$ to $\pi_P^{-1}(1) \cong \overline{M}$, we can reduce our task to producing holomorphic disks inside of $\overline{M}$ for an arbitrary choice of compatible almost complex structure.

To produce these disks, we will do Floer theory in the complement $M = \overline{M} \smallsetminus \pi^{-1}(\infty)$.  A priori, an arbitrary K\"{a}hler form on $M$ does not need to be convex in any sense that is amenable to doing Floer theory.  So we show in \cref{part:SymplecticDeformations} that we can symplectically embed $M$ into a convex symplectic domain (see \cref{defn:ConvexSymplecticDomain}) that is diffeomorphic to $M_c \coloneqq \pi^{-1}(\ID_c)$ for some $c>0$.  A collar neighborhood of the boundary of $M_c$ will be modeled after a symplectic mapping cylinder.  It has well-defined Floer theory for appropriate choices of admissible Hamiltonians and almost complex structures.  Producing our needed embedding is rather non-trivial.  First, via a symplectic deformation that preserves the K\"{a}hler class, we ``push'' $M$ away from $\pi^{-1}(\infty)$ and into $M_c$, see \cref{sec:ComplementEmbeddings}.  Then we realize $\pi: M_c \to \ID_c$ as a Hamiltonian fibration with a singularity over $0$.  Using this description, we symplectically deform a collar neighborhood of the boundary of $M_c$, ``straightening it out'', to a symplectic mapping cylinder, see \cref{sec:HamFibs}.  To continue, let us assume that the constructed embedding has image lying in $M_a \subset M_c$.

\emph{Action completed symplectic cohomology}: The above embeddings will reduce our task to producing holomorphic disks (for arbitrary compatible almost complex structures) in $M_c$ that are sections over $\ID_a$.  To do this, we will use Floer theory.  Using Pardon's virtual fundamental chains package \cite{Pardon_VFC}, we define the action completed symplectic cohomology group of a compact subset $K$ inside of $M_c$ (and for compact subsets in more general convex symplectic domains, see \cref{subsec:ActionCompletedCohomology}).  For the moment, we denote this group by $\widehat{SH}(K \subset M_c)$ even though it depends on the choice of symplectic form and other data.  Roughly, $\widehat{SH}(K \subset M_c)$ is computed by taking a Novikov completion (see \cref{subsec:NovikovCompletions}) of the colimit of the Hamiltonian Floer chain complexes of an increasing sequence of Hamiltonians on $M_c$ that converge to zero on $K$ and diverge to infinity on $M_c \smallsetminus K$.  A useful feature of these groups is that if $K$ is stably displaceable inside of $M_c$, then $\widehat{SH}(K \subset M_c)\otimes \Lambda$ vanishes, where $\Lambda$ denotes the universal Novikov field, see \cref{subsec:VanishingForStDisp}.  A second feature is that when $K = M_a = \pi^{-1}(\ID_a) \subset M_c$, we have a long exact sequence, see \cref{prop:PropertiesOfSH} \cref{prop:LES},
\[ \xymatrix{ \ar[r] & H(M_c)  \ar[r] & \widehat{SH}(M_a \subset M_c) \otimes \Lambda \ar[r] & \widehat{SH}_+(M_a \subset M_c) \ar[r] & }, \]
where $H(M_c)$ denotes the cohomology of $M_c$ with coefficients in $\Lambda$, and $\widehat{SH}_+(M_a,\subset M_c)$ is a chain complex that is generated by orbits that correspond to Reeb orbits of a stable Hamiltonian structure associated to the collar of $\partial M_c$.  In particular, these are orbits that when projected by $\pi$ to $\ID_a$ wrap positively around $\partial \ID_a$.  This should be thought of as some action completed version of positive symplectic cohomology.  Notice that we work in a non-exact setting.  So this long exact sequence does not arise from an action filtration (as the action functional is multi-valued).  Instead we prove an integrated maximum principle, see \cref{sec:WeaklyConvexDomains}, for our convex symplectic domains and use this to topologically construct a filtration of our Floer chain complexes that gives rise to the above long exact sequence.

From these two features, one finds that if $M_a$ is stably displaceable inside of $M_c$, then the boundary homomorphism of this long exact sequence is an isomorphism over $\Lambda$.  Unwinding the definition of the boundary homomorphism, one finds that there exists a Floer trajectory connecting a Reeb orbit near the boundary of $M_a$ to an index zero critical point of a Hamiltonian that represents the unit inside of the cohomology of $M_c$.  This produce a Floer trajectory whose projection under $\pi$ covers $\ID_a$.  

However, a Floer trajectory is not a holomorphic curve, which is what we desire.  To remedy this, we prove a Gromov-Floer-type compactness result for sequences of Floer trajectories associated to sequences of Hamiltonians that converge to zero on compact subsets.  We show that these sequences of Floer trajectories converge to pearl necklaces\footnote{We warn the reader that we use different nomenclature in \cref{sec:Pearls} to state our result.}, that is, holomorphic curves adjoined by negative gradient trajectories of some background Morse function, see \cref{sec:Pearls}.  In particular, when the negative end of our Floer trajectory is an index zero critical point, the negative end of the pearl necklace must be a non-constant holomorphic curve and not a negative gradient trajectory.  The positive end of the necklace wraps positively around $\partial \ID_a$.

By ignoring the trajectories and curves in between the two ends of the necklace, this ``turning off'' of the Hamiltonian perturbation gives a genuine (possibly disconnected) holomorphic curve that has a component whose projection under $\pi$ covers $\ID_a$ and has a non-constant component with a point constraint.  Assuming that the image of $M$ under our symplectic embeddings above is completely contained in $M_a$, we can use these holomorphic curves to produce our desired disks.  Considering continuation maps associated to varying Hamiltonians and almost complex structures, one can show that the existence and energy of the resulting holomorphic curves above are independent of the choice of point constraint and almost complex structure.  In this manner, we will obtain our desired holomorphic curves with boundaries (assuming that $\widehat{SH}(M_a \subset M_c) \otimes \Lambda \equiv 0$).
	
\emph{Vanishing of action completed symplectic cohomology}: We now explain why $\widehat{SH}(M_a \subset M_c) \otimes \Lambda$ vanishes in our setup.  There are two cases to consider, each corresponding to one of our main theorems.

When $\pi$ is smooth over $0$, the symplectic embedding of $M$ into $M_c$ that we produce above can actually be made into a symplectic embedding into the product symplectic manifold $\ID_c \times F$, where $F$ is the smooth fibre of $\pi$.  Every compact subset of $\IC \times F$ is displaceable in some $M_c$ for $c$ sufficiently large.  In this manner, we obtain vanishing for $\widehat{SH}(M_a \subset M_c) \otimes \Lambda$.

When $M$ has vanishing first Chern class, we proceed as follows.  First, a result of McLean \cite[Corollary 6.21]{McLean_BirationalCalabiYauManifoldsHaveTheSameSmallQuantumProducts} implies that $\pi^{-1}(0)$ is stably displaceable inside of $M_c$.  Consequently, $\widehat{SH}(M_\varepsilon \subset M_c) \otimes \Lambda$ vanishes for some small $\varepsilon > 0$.  However, most likely $M_\varepsilon \subset M_a$.  So we consider a rescaling morphism
\[ \widehat{SH}(M_a \subset M_{c+a-\varepsilon}) \otimes \Lambda \to \widehat{SH}(M_{\varepsilon} \subset M_{c}) \otimes \Lambda. \]
Roughly, we radially rescale our Hamiltonian vector fields and identity the generators of the associated Hamiltonian Floer chain complexes.  This rescaling sends an orbit $x$ to $x \cdot T^{(a-\varepsilon) \cdot w(x)}$, where $T$ is the Novikov parameter, and $w(x)$ denotes the winding number of $x$ about the origin in $\IC^\times$.  In general this morphism is not a bounded map of the associated uncompleted complexes, and, thus, fails to be an isomorphism after completing our complexes.  To remedy this, we use that $c_1(M) = 0$ to obtain a grading on our Hamiltonian Floer chain complexes.  So to conclude that this rescaling morphism is an isomorphism, it now suffices to show that it is degree-wise bounded.  More explicitly, one need to show that if the Conley-Zehnder index of the orbit $x$ is bounded by $n$, then the winding number $w(x)$ is bounded by some constant $C_n$ that is independent of $x$.  To achieve this, we construct an explicit sequence of Hamiltonians that are adapted to a normal crossings resolution of the singular fibre $\pi^{-1}(0)$ that have this index-bounded property, see \cref{sec:DivisorModels}.  The Conley-Zehnder indices of the orbits for these Hamiltonians will be controlled by the winding numbers about the divisors in the normal crossings resolution and the discrepancy of the resolution.  So this sequence of Hamiltonians (when rescaled) will give rise to the desired isomorphism, and (after relabeling the value $c$) the vanishing of $\widehat{SH}(M_a \subset M_c) \otimes \Lambda$.


\subsection{Structure of the exposition}

\cref{part:FloerTheory} discusses Hamiltonian Floer cohomology and the proofs of our main results.  In \cref{sec:ConvexDomains}, we define  convex symplectic domains, which are a class of symplectic domains that satisfy convexity properties with respect to pseudo-holomorphic curves and are the spaces for which we define Hamiltonian Floer cohomology and its analogues.  In \cref{sec:AdmissibleFloerData}, we define admissible families of Hamiltonians and almost complex structures for convex symplectic domains.  In \cref{sec:FloerModuliSpaces}, we discuss moduli spaces of Floer trajectories and give some basic energy estimates for Floer trajectories.  In \cref{sec:HamiltonianFloerCohomology}, we import Pardon's virtual fundamental chains package \cite{Pardon_VFC} to define both Hamiltonian Floer cohomology and action completed symplectic cohomology for convex symplectic domains.  In \cref{sec:PropertiesOfActionCompletedSH}, we discuss some properties of action completed symplectic cohomology, deriving a long exact sequence that relates action completed symplectic cohomology with the Morse cohomology of the symplectic domain and discussing the vanishing of action completed symplectic cohomology groups associated to stably displaceable subsets.  In \cref{sec:DiskUnirulings}, we prove a condition, phrased in terms of the vanishing of action completed symplectic cohomology, that ensures the existence of genus zero, holomorphic curves with boundaries through every point in a domain.  We also show that the energies of these curves are well controlled under deformations of almost complex structures.  In \cref{sec:ProofOfTheorem}, we give the proofs of our main results.

The second part discusses the part of the proof of our main result that is specific to the Calabi-Yau case.  In \cref{sec:CZIndices}, we discuss Conley-Zehnder indices of contractible $1$-periodic orbits of Hamiltonians, and define pseudo Morse-Bott families of orbits of Hamiltonians.  In \cref{sec:RescalingIso}, we establish a rescaling isomorphism for the action completed symplectic cohomology of convex symplectic domains whose collar neighborhoods are symplectic mapping cylinders.  In \cref{sec:DivisorModels}, we discuss symplectic normal crossings divisors and their standard tubular neighborhoods, and using these neighborhoods, we construct sequences of Hamiltonians whose dynamics are well-adapted to the normal crossings structure.

The third part discusses Hamiltonian fibrations and establishes the symplectic deformations that are used to construct the ``modified parallel transport maps'' (the symplectic embeddings) in the proofs of our main results.  In \cref{sec:ComplementEmbeddings}, we prove a self-embedding result for complements of normal crossings divisors in K\"{a}hler manifolds.  In \cref{sec:HamFibs}, we discuss the relationship between Hamiltonian fibrations and convex symplectic domains.  In particular, we focus on Hamiltonian fibrations over $\IC^\times$, relating them to symplectic mapping cylinders.

The forth part establishes Gromov-Floer-type compactness results.  In \cref{sec:CompactnessForDegenerationsALaFish}, we prove a compactness result for sequences of holomorphic curves with boundaries in spaces of the form given by our degeneration to the normal cone setup.  In \cref{sec:Pearls}, we prove a Gromov-Floer type compactness result for sequences of Floer trajectories associated to Hamiltonians that degenerate to the zero Hamiltonian on a compact subset.

The fifth part is a string of appendices.  In \cref{sec:AppendixHA}, we discuss the requisite homological algebra (over the universal Novikov ring) that is needed to define action completed symplectic cohomology.  In \cref{sec:WeaklyConvexDomains}, we establish an integrated maximum principle for our convex symplectic domains.  In \cref{sec:DisDiv}, we discuss stable displaceability properties of neighborhoods of fibres of proper holomorphic maps to Riemann surfaces.


\subsection*{Acknowledgements}
I would like to thank my thesis advisor, Mohammed Abouzaid, for suggesting this project, for numerous conversations, and for his comments on drafts of this paper.  I would like to thank Daniel Litt for partially conjecturing our main results, and for conversations about their algebro-geometric motivation.  Finally, I would like to thank Mark McLean for conversations pertaining to this paper.


\part{Floer theory}\label{part:FloerTheory}

In this part, we define action completed symplectic cohomology groups associated to compact subsets of convex symplectic domains, and establish some of their properties.  The setting in which we do Floer theory is more general than most of the literature.  First, we do not assuming any relationship between our symplectic forms and first Chern classes.  Second, our symplectic domains do not necessarily have contact-type boundaries.  Instead, we consider symplectic domains whose boundaries admit stable Hamiltonian structures with admissible almost complex structures.  To deal with the first issue, we import Pardon's virtual fundamental chain package to define Hamiltonian Floer chain complexes \cite{Pardon_VFC}.  To deal with the second issue, we establish an integrated maximum principle for our convex symplectic domains that ensures that the images of pseudo-holomorphic curves are disjoint from the boundaries.


\section{Convex symplectic domains}\label{sec:ConvexDomains}

We work with compact symplectic manifolds whose boundaries admit stable Hamiltonian structures with admissible almost complex structures.  We refer to such symplectic domains as convex symplectic domains.  Our definition of convex symplectic domains is analogous to the definition of symplectic domains with contact-type boundaries; however, our notion of convexity is more general.  We briefly review stable Hamiltonian structures before giving the definition of convex symplectic domains.

\begin{defn}\label{defn:StableHamiltonianStructure}\cite{HoferZehnder_StableHamDefinition}
A \emph{stable Hamiltonian structure} on a manifold $Y$ of dimension $2n-1$ is a pair $(\omega,\alpha)$ that satisfies:
\begin{enumerate}
	\item $d\omega = 0$,
	\item $\alpha \wedge \omega^{n-1} > 0$, and
	\item $\ker(\omega) \subset \ker(d\alpha)$.
\end{enumerate}
\end{defn}

\begin{defn}\label{defn:HyperplaneDistribution}
The \emph{hyperplane distribution} of a stable Hamiltonian structure $(\omega,\alpha)$ on $Y$ is the distribution $\xi \coloneqq \ker(\alpha)$.  The \emph{Reeb vector field} of a stable Hamiltonian structure $(\omega,\alpha)$ on $Y$ is the (unique) vector field $\sR$ determined by $\iota_{\sR}\omega \equiv 0$ and $\iota_{\sR} \alpha = 1$.  The Reeb vector field is \emph{non-degenerate} if for all orbits $\gamma$ of $\sR$, the linearized return map
\[d(\varphi_T^\sR)_{\gamma(0)}|_{\xi_{\gamma(0)}}: \xi_{\gamma(0)} \to \xi_{\gamma(0)},\]
where $T$ is the period of $\gamma$ and $\varphi_T^\sR$ is the time $T$ flow of $\sR$, has no eigenvalues equal to $1$.
\end{defn}

When $\omega = d\alpha$, $Y$ is a contact manifold with contact $1$-form $\alpha$.  The hyperplane distribution and Reeb vector field in \cref{defn:HyperplaneDistribution} agree with the usual hyperplane distribution and Reeb vector field associated to the contact $1$-form $\alpha$.  Like with contact manifolds, we can take symplectizations of stable Hamiltonian structures and define a class of admissible almost complex structures on them.

\begin{defn}\label{defn:SymplectizationStableHamStructure}
The \emph{symplectization of a stable Hamiltonian structure} $(\omega,\alpha)$ on $Y$ is the symplectic manifold $((1-\varepsilon,1+\varepsilon) \times Y,\Omega)$ with
\[\Omega \coloneqq \omega+ d( (\varepsilon \cdot r) \alpha),\]
where $0 < \varepsilon$ is sufficiently small so that $\Omega$ is symplectic.
\end{defn}

After rescaling, we can (and always will) assume that $\varepsilon = 1$ in \cref{defn:SymplectizationStableHamStructure}.

\begin{defn}\label{defn:AdmissibleJOnSymplectization}
An almost complex structure $J$ on the symplectization of a stable Hamiltonian structure is \emph{admissible} if
\begin{enumerate}
	\item $J$ is $\Omega$-compatible,
	\item $-dr \circ J = r \cdot \alpha$,
	\item $d\alpha|_\xi(\cdot, J \cdot)$ is positive semi-definite on $\xi$, and
	\item $J(\xi) = \xi$.
\end{enumerate}
\end{defn}

\begin{rem}\label{rem:NonContractibilityOfJ}
The third condition in \cref{defn:AdmissibleJOnSymplectization} could make the resulting space of admissible almost complex structures non-contractible.  In essence, the contracting homotopy of the space of compatible almost complex structures (on $\xi$), which is produced by the polarization decomposition (see \cite{Silva_LecturesOnSymplecticGeometry}), a priori, does not need to preserve the subspace of compatible almost complex structures that make $d\alpha|_\xi(\cdot, J \cdot)$ positive semi-definite on $\xi$.  Regardless, this possible non-contractibility does not affect our setup.  See \cref{rem:AdmissibleJNotContractible} for further discussion. 
\end{rem}

\begin{defn}\label{defn:ConvexSymplecticDomain}
A compact symplectic manifold with boundary $(M,\Omega)$ is a \emph{convex symplectic domain} if
\begin{enumerate}
	\item there exists a collar neighborhood $N(\partial M)$ of the boundary that is symplectomorphic to the (bottom half of the) symplectization of a stable Hamiltonian structure $(\Omega|_{\partial M}, \alpha)$ on $\partial M$ for some $1$-form $\alpha$ on $\partial M$\footnote{In fact, by a symplectic neighborhood theorem, it suffices for there to exist a $1$-form $\alpha$ such that $(\Omega|_{\partial M}, \alpha)$ is a stable Hamiltonian structure on $\partial M$.}, and
	\item there exists an $\Omega$-compatible almost complex structure $J$ on $M$ such that $J|_{N(\partial M)}$ is admissible with respect to the symplectization.
\end{enumerate}
Using the identification with the symplectization, the \emph{radial/collar coordinate} is 
\[r: N(\partial M) \cong (0,1] \times \partial M \to (0,1].\]
The associated \emph{Liouville $1$-form} on $M$ is $\lambda \coloneqq r \cdot \alpha$.  Denote a convex symplectic domain by $(M,\Omega,\lambda)$.  Finally, an almost complex structure $J$ as above is \emph{admissible} for $(M,\Omega,\lambda)$.
\end{defn}

\begin{exmp}
Our key example, discussed in \cref{subsec:HamFibs_MappingTori} of a convex symplectic domain is a symplectic domain whose boundary is the mapping torus of a symplectomorphism.
\end{exmp}

\begin{notn}\label{notn:rInverse}
$M_a = r^{-1}([0,a])$, where $r$ is extended by zero outside of the collar.
\end{notn}

\begin{rem}
We stress two features of convex symplectic domains.  First, the dynamics of radial Hamiltonians on the collar $N(\partial M)$ are given by the dynamics of the Reeb vector field $\sR$ of the stable Hamiltonian structure, since the $\Omega$-dual of $-dr$ is given by the Reeb vector field $\sR$.  Second, convex symplectic domains satisfy an integrated maximum principle with respect to radial Hamiltonians.  This allows us to define Hamiltonian Floer cohomology for Hamiltonians that are radial near the boundary.  More importantly, this integrated maximum principle allows for the collar coordinate $r$ to give rise to filtrations of the associated Hamiltonian Floer chain complexes that are preserved by continuation maps.  Normally such filtrations are constructed using the action filtration, say for Liouville domains.  However, in the non-exact setting, the action functional is multi-valued, and not well-suited for defining filtrations.  The filtrations derived from the integrated maximum principle will essentially agree with the usual action filtration for Liouville domains.
\end{rem}

\begin{notn}\label{notn:MaximumPrincipleSetup}
\begin{enumerate}
	\item Let $(M,\Omega,\lambda)$ be a convex symplectic domain.
	\item Let $J_s$ be an $\IR$-family of admissible almost complex structures for $(M,\Omega,\lambda)$.
	\item Let $H: \IR \times S^1 \times M \to \IR$ be a family of Hamiltonians such that $H_s|_{N(\partial M)} = h_s$, where $h_s$ is a family of radial functions in $r$ that satisfies $\partial_s \partial_r h_s \leq 0$.
	\item Let $u: \IR \times S^1 \to M$ be a smooth map that satisfies
	\[0 = (du-X_{H} \otimes dt)^{0,1}.\]
\end{enumerate}
\end{notn}

We prove the following integrated maximum principle in \cref{sec:WeaklyConvexDomains}.

\begin{prop}\label{prop:MaximumPrincipleForConvexDomains}
If
\[ \lim_{s \to \pm \infty} r \circ u(s,t) = c_{\pm} < 1, \]
then $r \circ u \leq \max(c_\pm)$.  Moreover, if $\partial_r \partial_r h_s \geq 0$, then $r \circ u \leq c_+$, and if $c_- = c_+$, then $r \circ u$ is constant.
\end{prop}


\section{Admissible Floer data}\label{sec:AdmissibleFloerData}

We define two semisimplicial complexes associated to a convex symplectic domain $(M,\Omega,\lambda)$.  The sets of $\ell$-simplices will be $\Delta^\ell$-families of pairs of Hamiltonians and admissible almost complex structures on $(M,\Omega,\lambda)$ that satisfy admissibility conditions that allow their associated Floer theoretic objects to be well-defined.

\subsection{Morse flow lines on the simplex}\label{subsec:MorseSimplex}

We review the Morse theory of the standard simplex.  We give this discussion because we need to ``count'' possibly broken, possibly nodal Floer trajectories when defining our Floer differentials, continuations, and higher homotopies.  The Morse theory of the standard simplex gives a way to express our Floer trajectories of interest.

\begin{notn}\label{notn:MorseSimplex}
The standard $\ell$-simplex is
\[ \Delta^\ell = \{ \underline{s} \in \IR^{\ell+2} \mid 0 = s_0 \leq \cdots \leq s_\ell \leq s_{\ell+1} = 1 \}. \]
For $0 \leq i \leq \ell$, the $i$th vertex of $\Delta^\ell$, denoted $\underline{e_i}$, is the point with coordinates
\[ s_0 = \cdots = s_{\ell-i} \gap \mbox{and} \gap s_{\ell-i+1} = \cdots = s_{\ell+1}.\]
Let $f_\ell: \Delta^\ell \to \IR$ denote the Morse function
\[ f_\ell(\underline{s}) = \sum_{i=1}^\ell \cos(\pi s_i). \]
The vertex $\underline{e_i}$ is the unique index $\ell-i$ critical point of $f_\ell$.  The gradient of $f_\ell$ is 
\[ \nabla f_\ell(\underline{s}) = -\pi \cdot (0,\sin(\pi s_1),\dots,\sin(\pi s_\ell),0), \]
which is tangent to faces of $\Delta^\ell$.  Moreover, restrictions of these gradient vector fields are compatible with inclusions/restrictions of faces $\Delta^{\ell-1} \hookrightarrow \Delta^\ell$.
\end{notn}

\begin{defn}\label{defn:MorseFlowline}
A \emph{Morse flow line} from $\underline{e_0}$ to $\underline{e_\ell}$ in $\Delta^\ell$ is a map
\[ \gamma = \sqcup_{i=1}^k \gamma_i: \sqcup_{i = 1}^k \IR \to \Delta^\ell \]
such that $\gamma_i$ is a (possibly constant) negative gradient flow line of $f_\ell$ and
\[ \underline{e_0} = \lim_{s \to -\infty} \gamma_1(s), \dots, \lim_{s \to + \infty} \gamma_i(s) = \lim_{s \to - \infty} \gamma_{i+1}(s), \dots, \lim_{s \to + \infty} \gamma_{k}(s) = \underline{e_\ell}.\]
Two Morse flow lines $\gamma$ and $\gamma'$ are isomorphic if there exist translations $\varphi_i: \IR \to \IR$ such that $\gamma_i = \gamma_i' \circ \varphi_i$.
\end{defn}

\subsection{Kan complex of admissible Floer data}\label{subsec:KanAdmissibleStructures}

We define the semisimplicial sets of admissible Floer data.  Let $(M,\Omega,\lambda)$ be a convex symplectic domain with collar coordinate $r$.

\begin{defn}\label{defn:AdmissibleJSimplex}
A family of almost complex structures $J^\sigma: \Delta^\ell \to \End(TM)$ is \emph{admissible}
if for each $\underline{s} \in \Delta^\ell$, $J^\sigma_{\underline{s}}$ is an admissible almost complex structure for $(M,\Omega,\lambda)$.  Denote the set of $\Delta^\ell$-families of admissible almost complex structures by $\sJ_\ell(M,\Omega,\lambda)$.
\end{defn}

\begin{defn}\label{defn:AdmissibleFamiliesOfHam}
A family of Hamiltonians $H^\sigma: \Delta^\ell \times S^1 \times M \to \IR$ is \emph{admissible} if
\begin{itemize}
	\item the Hamiltonian $H^\sigma_{\underline{e_i}}$ associated to each vertex is non-degenerate, and
	\item $H^\sigma_{\underline{s}}$ is locally constant about each vertex $\underline{e_i}$ in $\Delta^\ell$,
\end{itemize}
and near $\partial M$
\begin{itemize}
	\item $H^\sigma_{\underline{s}}$ is a linear function in $r$, and
	\item $\partial_{\underline{s}} \partial_r H^\sigma_{\underline{s}} \circ \gamma_i'(s) \leq 0$ for each Morse flow line $\gamma = \sqcup \gamma_i$ in $\Delta^\ell$.
\end{itemize}
Denote the set of admissible $\Delta^\ell$-families of Hamiltonians by $\sH_\ell(M,\Omega,\lambda)$.
\end{defn}

To define action completed symplectic cohomology, we will need our continuation maps to be defined over the universal Novikov ring.  This will be ensured for continuations associated to families of the following Hamiltonians.

\begin{defn}\label{defn:MonotoneAdmissibleFamiliesOfHam}
A family of admissible Hamiltonians $H^\sigma: \Delta^\ell \times S^1 \times M \to \IR$ is \emph{monotonically admissible} if $\partial_{\underline{s}} H^\sigma_{\gamma_i(s)} \circ \gamma_i'(s) \leq 0$ for each Morse flow line $\gamma = \sqcup \gamma_i$ in $\Delta^\ell$.  Denote the set of $\Delta^\ell$ monotonically admissible families of Hamiltonians by $\sH^+_\ell(M,\Omega,\lambda)$.
\end{defn}

Ranging over $\ell$, the sets $\sJ_\ell(M,\Omega,\lambda)$, $\sH_\ell(M,\Omega,\lambda)$ and $\sH_\ell^+(M,\Omega,\lambda)$ assemble into semisimplicial sets, denote $\sJ(M,\Omega,\lambda)$, $\sH(M,\Omega,\lambda)$, and $\sH^+(M,\Omega,\lambda)$ respectively.  Define the product semisimplicial set $\sJ \sH(M,\Omega,\lambda)$ whose $\ell$-simplicies are pairs $(H^\sigma,J^\sigma) \in \sH_\ell(M,\Omega,\lambda) \times \sJ_\ell(M,\Omega,\lambda)$.  Similarly, define $\sJ \sH^+(M,\Omega,\lambda)$.

\begin{rem}
We explain when $\ell$-simplices of Floer data will be used in our constructions.
\begin{enumerate}
	\item As is standard, we need $0$-simplices and $1$-simplices to define Hamiltonian Floer cohomology groups and continuation maps between them respectively
	\item We need $1$-simplices to define action completed symplectic cohomology; however, $2$-simplices and $3$-simplices are further needed to show that action completed symplectic cohomology is well-defined and independent of all choices.
	\item We also use $2$-simplices and $3$-simplices to derive further invariance properties of action completed symplectic cohomology in \cref{subsec:FunctionInvarience}.
\end{enumerate}
So the reader can ignore $\ell$-simplices of Floer data for $\ell \geq 4$.
\end{rem}

\subsection{Connectivity}\label{subsec:Connectivity}
The semisimplicial sets $\sH(M,\Omega,\lambda)$ and $\sH^+(M,\Omega,\lambda)$ are not contractible (or even $0$-connected) because of the condition $\partial_{\underline{s}} \partial_r H^\sigma_{\underline{s},t} \circ \gamma_i'(s) \leq 0$ (and the monotone property for $\sH^+(M,\Omega,\lambda)$).  We point out that while these semisimplicial sets are not contractible, their geometric realizations are contractible.  This is because in forming the geometric realization one has to ``add in'' all of the ``opposite'' edges and simplices.  Nevertheless, to obtain invariance properties for our constructions, the connectivity given by \cref{lem:ConnectivityOfFloerData} suffices.

\begin{lem}\label{lem:ConnectivityOfFloerData}
Every map of $\partial \Delta^\ell$ to either $\sH(M,\Omega,\lambda)$ or $\sH^{+}(M,\Omega,\lambda)$ extends to $\Delta^\ell$ for $\ell \geq 2$.
\end{lem}

\begin{proof}
This follows from the proof of \cite[Proposition 3.2.13]{Varolgunes_MayerVietorisPropertyForRelativeSymplecticCohomology} and the work leading up to it.  The only difference between our setup and the setup of \cite{Varolgunes_MayerVietorisPropertyForRelativeSymplecticCohomology} is that we need to incorporate the condition that $\partial_{\underline{s}} \partial_r H^\sigma_{\underline{s}} \circ \gamma_i'(s) \leq 0$ when interpolating families of Hamiltonians.  However, incorporating this condition is the same as incorporating the condition that families be monotone as in the proof of \cite[Proposition 3.2.13]{Varolgunes_MayerVietorisPropertyForRelativeSymplecticCohomology}.  Incorporating these conditions follows from a parameterized version of the Whitney Extention theorem (or rather a study of the construction used to prove the Whitney Extension theorem).
\end{proof}

\begin{rem}\label{rem:AdmissibleJNotContractible}
The analogue of \cref{lem:ConnectivityOfFloerData} for $\sJ\sH(M,\Omega,\lambda)$ and $\sJ\sH^+(M,\Omega,\lambda)$ need not hold, since the semisimplicial set $\sJ(M,\Omega,\lambda)$ could be non-connected.  To obtain an analogue of \cref{lem:ConnectivityOfFloerData} for $\sJ\sH(M,\Omega,\lambda)$ and $\sJ\sH^+(M,\Omega,\lambda)$, one could do one of the following.
\begin{enumerate}
	\item Begin with the assumption that $\sJ(M,\Omega,\lambda)$ is contractible (which occurs if $\partial M$ is of contact-type or is a mapping torus of a symplectomorphism).
	\item Fix an admissible almost complex structure $J$ near $\partial M$ and replace $\sJ(M,\Omega,\lambda)$ with the semisimplicial set of $\Delta^\ell$-families of admissible almost complex structures that agree with $J$ near the boundary of $M$.  This replacement semisimplicial set is contractible.
\end{enumerate}
When we define Hamiltonian Floer cohomology and action completed symplectic cohomology, we implicitly assume that either $\sJ(M,\Omega,\lambda)$ is contractible, or we have fixed data for the second item.  If $\sJ(M,\Omega,\lambda)$ is not contractible, then our constructions depend on this extra data.
\end{rem}


\section{Moduli spaces of Floer trajectories}\label{sec:FloerModuliSpaces}

Let $(M,\Omega,\lambda)$ be a convex symplectic domain with collar coordinate $r$.  We define moduli spaces of Floer trajectories associated to simplices in $\sJ\sH(M,\Omega,\lambda)$ (and $\sJ\sH^+(M,\Omega,\lambda)$).

\subsection{Moduli spaces of Floer trajectories}\label{subsec:FloerModuliSpaces}

The following definitions are extremely well-known.  So we give them in our notation with little discussion.

\begin{defn}\label{defn:FloersEquation}
Let $\gamma:\IR \to \Delta^\ell$ be a (possibly constant) Morse flow line and let $u: \IR \times S^1 \to M$ be a smooth map.  The \emph{continuation operator} associated to a simplex $\sigma$ in $\sJ \sH_\ell(M,\Omega,\lambda)$ is
\begin{align*}
 \overline{\partial}_{\sigma}(\gamma,u) = du + J^{\sigma}(\gamma,u) \circ du \circ j - X_{H^\sigma}(\gamma, u) \otimes dt - J^\sigma(\gamma,u) \circ X_{H^\sigma}(\gamma,u) \otimes ds.
\end{align*}
\end{defn}

\begin{defn}\label{defn:ParameterizedFloerTrajectories}
Let $\sigma$ be a simplex in $\sJ\sH_\ell(M,\Omega,\lambda)$.  Let $x_0$ and $x_\ell$ be $1$-periodic orbits of $H^\sigma_{\underline{e_0}}$ and $H^\sigma_{\underline{e_\ell}}$ respectively.  A \emph{Floer trajectory} of type $(\sigma,x_0,x_\ell)$ is a pair $(\gamma,u)$, where
\begin{enumerate}
	\item $\gamma = \sqcup_{i=1}^k \gamma_i: \sqcup_{i=1}^k \IR \to \Delta^\ell$ is a Morse flow line with
	\[ \underline{e_0} = \lim_{s \to -\infty} \gamma_1(s), \dots, \lim_{s \to + \infty} \gamma_i(s) = \lim_{s \to - \infty} \gamma_{i+1}(s), \dots, \lim_{s \to + \infty} \gamma_{k}(s) = \underline{e_\ell},\]
	(Write $v_i \coloneqq \lim_{s \to -\infty} \gamma_{i+1}(s)$ for $0 \leq i \leq k-1$ and $v_k \coloneqq \underline{e_\ell}$.) and
	\item $u : C \coloneqq \sqcup_{i=1}^k C_i \to M$ is a smooth map with each $C_i$ being a nodal curve of type $(0,2)$ with punctures $p_i^\pm$ both contained in the same irreducible component, denoted $C_i^0$.  (So $C_i^0 \cong \IR \times S^1$ and $C_i$ is $\IR \times S^1$ with bubble trees attached along a finite set of points in $\IR \times S^1$.)  Write $C^0 = \sqcup_{i=1}^k C_i^0$.
\end{enumerate}
The pair $(\gamma,u)$ satisfies the following.  There exists $1$-periodic orbits $x_i$ of $H^\sigma_{v_i}$ (for $1 \leq i \leq k-1$) such that
\begin{enumerate}
	\item $\lim_{s \to -\infty} u_{i+1}(s,t) = x_{i}(t)$, and
	\item $\lim_{s \to +\infty} u_i(s,t) = x_i(t)$.
\end{enumerate}
The components $u_i$ satisfy
\begin{enumerate}
	\item $\overline{\partial}_\sigma(\gamma_i,u_i|_{C_i^0}) = 0$, and
	\item $u_i|_{C_i \smallsetminus C_i^0}$ is holomorphic.
\end{enumerate}
\end{defn}

\begin{defn}\label{defn:IsomorphismOfParameterizedFloerTrajectories}
Two Floer trajectories $u$ and $u'$ are \emph{isomorphic} if there exist an isomorphism of curves of type $(0,2)$, $\sqcup_{i=1}^k  \psi_i: \sqcup_{i=1}^k C_i \to \sqcup_{i=1}^k C_i'$, and translations $\sqcup_{i=1}^k \phi_i: \sqcup_{i=1}^k \IR \to  \sqcup_{i=1}^k\IR$, such that
$u'_i \circ \psi_i = u_i$ and $\gamma_i' \circ \phi_i = \gamma_i$.
A Floer trajectory is \emph{stable} if its automorphism group (that is, the group of self-isomorphisms) is finite.
\end{defn}

\begin{defn}\label{defn:ModuliSpaceOfFloerTrajectories}
The \emph{moduli space of Floer trajectories of type $(\sigma,x_0,x_\ell)$}, denoted $\overline{\sM}(\sigma,x_0,x_\ell)$,
is the space of isomorphism classes of stable Floer trajectories of type $(\sigma,x_0,x_\ell)$, endowed with the Gromov topology.
\end{defn}

\begin{notn}\label{notn:NoGamma}
We will drop $\gamma$ from our notation for a Floer trajectory, and refer to the map $u$ as the Floer trajectory.
\end{notn}

\subsection{Energy}\label{subsec:Energy}

We discuss the geometric energy and topological energy of a Floer trajectory. 
\begin{defn}\label{defn:Energy}
The \emph{geometric energy} of a Floer trajectory $u$ of type $(\sigma,x_0,x_\ell)$ is
\begin{equation*}
E_{geo}(u) =  \frac{1}{2} \int_C  \|\partial_s u \|_{J^\sigma}^2  \ ds \wedge dt .
\end{equation*}
The \emph{topological energy} of $u$ is
\begin{equation*}
E_{top}(u) = \int_C u^*\Omega + \int_{S^1} H_{\underline{e_0}}^\sigma(x_0) \ dt - \int_{S^1} H_{\underline{e_\ell}}^\sigma(x_\ell) \ dt.
\end{equation*}
\end{defn}

\begin{lem}\label{lem:EnergyRelation}
Let $u$ be a Floer trajectory of type $(\sigma,x_0,x_\ell)$.
\begin{align*}
E_{geo}(u) = E_{top}(u) + \int_{C^0} ( \partial_{\underline{s}} H^{\sigma})(\gamma,u) \circ \gamma'  \ ds \wedge dt.
\end{align*}
In particular, if $\sigma \in \sJ \sH_\ell^+(M,\Omega,\lambda)$, then
\[ E_{top}(u) \geq E_{geo}(u) \geq 0.\]\qed
\end{lem}

We will typically apply energy estimates to families given by the following construction.

\begin{con}\label{con:TheBump}
Consider two Hamiltonians $H_0$ and $H_1$ with $H_0 \leq H_1$, globally, and $\partial_r H_0 \leq \partial_r H_1$, near $\partial M$.  Let $\ell: [0,1] \to [0,1]$ be a smooth function that satisfies:
\begin{enumerate}
	\item $\ell(s) \equiv 1$ near $0$,
	\item $\ell(s) \equiv 0$ near $1$, and
	\item $\ell' \leq 0$.
\end{enumerate}
Define $H^\sigma_s \coloneqq \ell(s) \cdot H_1 + (1-\ell(s)) \cdot H_0$.  It is straight-forward to verify that this family is monotonically admissible and satisfies $H^\sigma_{\underline{e_1}} = H_0$ and $H^\sigma_{\underline{e_0}} = H_1$.
\end{con}

\begin{cor}\label{cor:EnergyRelation1}
Let $u$ be a Floer trajectory of type $(\sigma,x_0,x_1)$, where $H^\sigma$ is as in \cref{con:TheBump} with $H_1 - H_0 \geq c$.
\[ E_{top}(u) \geq E_{geo}(u) + c.\]\qed
\end{cor}

\subsection{Compactness}\label{subsec:Compactness}

\begin{lem}\label{lem:ModuliWellDefined}
The image of every element of $\overline{\sM}(\sigma,x_0,x_\ell)$ in $M$ is disjoint from the boundary of $M$ and the subspace
\[ \left\{ u \in \overline{\sM}(\sigma,x_0,x_\ell) \mid E_{top}(u) \leq E_0 \right\} \]
is a compact, Hausdorff space for every constant $E_0$.
\end{lem}

\begin{proof}
The proof when $\partial M$ is contact carries over almost identically to our setting.  We omit standard details; however, we mention where our proof would diverge from the standard proof.  If $\partial M$ is not necessarily of contact-type, then $\Omega$ need not be exact near $\partial M$.  So images of Floer trajectories near $\partial M$ could have non-trivial bubbles; however, the argument for \cref{prop:MaximumPrincipleForWeaklyConvexDomains} with $h_s \equiv 0$ shows that any holomorphic map $\IP^1 \to M$ that meets the collar $N(\partial M)$ must be contained in a single slice $\{r\} \times \partial M$ of the collar.  By continuity of $u$, the image of the bubbles of $u$ are disjoint from $\partial M$ if and only is the image of the main component of $u$ is disjoint from $\partial M$.  With this in mind, the standard proof carries through.
\end{proof}


\section{Hamiltonian Floer Cohomology}\label{sec:HamiltonianFloerCohomology}

To prove our main result, we need to define Hamiltonian Floer cohomology without assuming a relationship between the first Chern class and the symplectic form. So we need to use virtual techniques.  For a guide through the literature on virtual fundamental chains, see \cite{Pardon_VFC}, \cite{McDuffTehraniFukayaJoyce_VFC}, and the references therein.  Such techniques have been studied in the context of Hamiltonian Floer cohomology by Fukaya-Ono \cite{FukayaOno_ArnoldConjectureAndGromovWittenInvariant}, Liu-Tian \cite{LiuTian_FloerHomologyAndArnoldConjecture}, Pardon \cite{Pardon_VFC}, and Abouzaid-Blumberg \cite{AbouzaidBlumber_ArnoldConjectureAndMoravaKTheory}.  We use Pardon's virtual fundamental chains package \cite{Pardon_VFC} to define Hamiltonian Floer cohomology for convex symplectic domains.  After giving the definition, we discuss the usage of radial Hamiltonians to define Hamiltonian Floer cohomology.  Finally, we define action completed symplectic cohomology groups associated to compact subsets of convex symplectic domains.  More explicitly, we adapt the construction given by \cite{Varolgunes_MayerVietorisPropertyForRelativeSymplecticCohomology} to convex symplectic domains.  Other flavors of action completed symplectic cohomology in various settings were studied in \cite{Varolgunes_MayerVietorisPropertyForRelativeSymplecticCohomology}, \cite{Groman_FloerTheoryAndReducedCohomologyOnOpenManifolds}, \cite{Venkatesh_RabinowitzFloerHomologyAndMirrorSymmetry}, and
\cite{McLean_BirationalCalabiYauManifoldsHaveTheSameSmallQuantumProducts}.

\subsection{Pardon's virtual fundamental chains package}\label{subsec:PardonVFC}

Roughly, to define Hamiltonian Floer cohomology, we need a (contravariant) functor from $\sJ\sH(M,\Omega,\lambda)$ (respectively $\sJ\sH^+(M,\Omega,\lambda)$) to the $\infty$-category of chain complexes over the universal Novikov field $\Lambda$ (respectively universal Novikov ring $\Lambda_{\geq 0}$).  That is, to every vertex, we associate a chain complex.  To every $1$-simplex, we associated a chain map.  To every $2$-simplex, we associate a chain homotopy.  Etc..

By \cref{lem:ModuliWellDefined}, all Floer trajectories lie in the interior of $M$.  So the analysis performed in \cite{Pardon_VFC} goes through without change.  We provide a cursory outline of Pardon's construction and explain how we obtain our desired functor.

To begin, the moduli spaces $\overline{\sM}(\sigma,{x_0},{x_\ell})$, furnished with choices of implicit atlases $\sA$ and coherent orientations $\fo$, can be assembled into a flow category $\overline{\sM}$ over $\sJ\sH(M,\Omega,\lambda)$.  From this data, Pardon constructs a trivial Kan fibration
\[\wt{\sJ \sH}(M,\Omega,\lambda) \to \sJ\sH(M,\Omega,\lambda).\]
Roughly, a section of this fibration is a coherent choice of virtual fundamental chains for the above moduli spaces.  Pardon constructs a diagram
\[ \xymatrix@C=1in{ \widetilde{\sJ \sH}(M,\Omega,\lambda)^{\op} \ar[d]^\pi \ar[r]^{\widetilde{\IH}(\sM,\sA,\fo)} & \mbox{N}_{\dg}(\Ch(\Lambda)) \ar[d]^{\mbox{forget}} \\ \sJ\sH(M,\Omega,\lambda)^{\op} \ar@{.>}[r]^{\IH(\sM,\sA,\fo)} & H^0(\Ch(\Lambda)),} \]
where $\mbox{N}_{\dg}(\Ch(\Lambda))$ is the differential graded nerve of $\Ch(\Lambda)$ and $H^0(\Ch(\Lambda_{\geq 0}))$ is the associated homotopy category, see \cite[Construction 1.3.1.6]{Lurie_HigherAlgebra}.  To obtain the dashed arrow $\IH(\sM,\sA,\fo)$, one must choose a section of $\pi$.  The functor $\IH(\sM,\sA,\fo)$ does not depend on the choice of section.  Fixing a coherent choice of virtual fundamental chains (a section of $\pi$),
\[\fc:{\sJ\sH(M,\Omega,\lambda)^{\op} \to \widetilde{\sJ\sH}(M,\Omega,\lambda)^{\op}},\]
gives a contravariant functor
\[ \IH(\sM,\sA,\fo,\fc): \sJ \sH(M,\Omega,\lambda)^{\op} \to \mbox{N}_{\dg}(\Ch(\Lambda)), \]
where we have denoted all the dependencies.  A concrete description of $\IH(\sM,\sA,\fo,\fc)$ for simplices of dimensions $\ell = 0$ and $1$ is given below.  

In what follows, we do not show that our Floer theoretic objects are independent of the above choices, that is, the choice of implicit atlas $\sA$, the choice of coherent orientations $\fo$, and the choice of coherent virtual fundamental chains $\fc$.  We do not establish independence because we do not need it to prove our main result.  All we need is the existence of such data.  Consequently, in our constructions, we can either fix some universal choice of data associated to each simplex in $\sJ\sH(M,\Omega,\lambda)$ or we can fix such choices inductively as we work through our constructions.  The latter is the approach taken by Pardon in \cite{Pardon_VFC}.  The former and latter approaches are also both considered by Varolgunes in \cite{Varolgunes_MayerVietorisPropertyForRelativeSymplecticCohomology}.  Regardless, we omit such data from our notation.

\subsection{Hamiltonian Floer chain complexes}\label{subsec:CF}
We give a more explicit description of the functor from the above discussion and discuss some of its properties.

\begin{notn}\label{notn:HomotopyClassesOfFloerTrajectories}
Let $\sigma$ be a simplex in $\sJ \sH_\ell(M,\Omega, \lambda)$.  Given orbits $x_0$ and $x_\ell$, let $\pi(x_0,x_\ell)$ denote the set of homotopy classes of Floer trajectories of type $(\sigma,x_0,x_\ell)$.  Define
\[\overline{\sM}(\sigma,x_0,x_\ell;A) = \left\{ u \in \overline{\sM}(\sigma,x_0,x_\ell) \mid [u] = A \in \pi(x_0,x_\ell) \right\}. \]
If $u$ and $u'$ are both in $\overline{\sM}(\sigma,x_0,x_\ell;A)$, then $E_{top}(u) = E_{top}(u')$.  So given $A \in \pi(x_0,x_\ell)$, define $E_{top}(A) \coloneqq E_{top}(u)$, where $u$ is any Floer trajectory with $[u] = A$.
\end{notn}

\begin{defn}\label{defn:FloerchainComplex}
The \emph{Floer chain complex} associated to a $0$-simplex $(H,J) \in \sJ \sH_0(M,\Omega,\lambda)$ is the chain complex $(CF^\bullet(H,J;M,\Omega,\lambda), \partial)$ where
\begin{itemize}
	\item $(CF^\bullet(H,J;M,\Omega,\lambda)$ is the $\IZ$-graded, free $\Lambda$-module generated by the {contractible} $1$-periodic orbits of $H$, that is, 
	\[ CF^{\bullet}(H,J;M,\Omega,\lambda) \coloneqq \Lambda \cdot \langle x \mid x \mbox{ is a 1-periodic orbit of } H, \, \ind(x) = \bullet \mod 2 \rangle \]
	\item $\partial: CF^\bullet(H,J;M,\Omega,\lambda) \to CF^{\bullet+1}(H,J;M,\Omega,\lambda)$ is the $\Lambda$-linear map given by
	\[ \partial(x_+) = \sum_{\tiny{\begin{matrix}x_-\\ \ind(x_-)-\ind(x_+) = 1 \mod 2\\ \end{matrix}}} \left( \sum_{A \in \pi(x_-,x_+)} \#_{vir}(\overline{\sM}(\sigma,x_-,x_+;A)) \cdot x_- \cdot  T^{E_{top}(A)} \right)\]
	where $\#_{vir}(\overline{\sM}(\sigma,x_-,x_+;A)) \in \IQ$ is determined by Pardon's virtual fundamental chains package.
\end{itemize}
The \emph{continuation map} associated to a $1$-simplex $\sigma \in \sJ \sH_1(M,\Omega,\Lambda)$ is the $\Lambda$-linear map
\[ c(\sigma;M,\Omega,\lambda): CF^\bullet(H^\sigma_{\underline{e_1}}, J^\sigma_{\underline{e_1}};M,\Omega,\lambda) \to CF^\bullet(H^\sigma_{\underline{e_0}},J^\sigma_{\underline{e_0}};\Omega,\lambda) \]
given by
\[ c(\sigma;M,\Omega,\lambda)(x_+) = \sum_{\tiny{\begin{matrix}x_-\\ \ind(x_-)-\ind(x_+) = 0 \mod 2\\ \end{matrix}}} \left( \sum_{A \in \pi(x_-,x_+)} \#_{vir}(\overline{\sM}(\sigma,x_-,x_+;A)) \cdot x_- \cdot  T^{E_{top}(A)} \right)\]
where $\#_{vir}(\overline{\sM}(\sigma,x_-,x_+;A)) \in \IQ$ is determined by Pardon's virtual fundamental chains package.
\end{defn}

A word on grading conventions is in order.  $\ind(x)$ denotes the Conley-Zehnder index of the orbit $x$, which is only well-defined modulo $2$.  So $CF^\bullet(H,J;M,\Omega,\lambda)$ is, a priori, only a $\IZ/2$-graded complex.  Here we extend $2$-periodically to obtain a $\IZ$-graded complex.  When $(M,\Omega)$ has vanishing first Chern class, $\ind(x)$ is a well-defined integer and gives $CF^\bullet(H,J;M,\Omega)$ an honest $\IZ$-graded.  The ``$2$-periodification'' of this $\IZ$-grading agrees with the a fore mentioned $\IZ$-grading.  When the grading of $CF^\bullet(H,J;M,\Omega)$ is important, we will mention it; however, if there is no need for gradings, we will drop the grading notation from our chain complexes.

\begin{notn}\label{notn:SimplifyToCF}
When the context is clear, we will drop $M$, $\Omega$, $\lambda$, and sometimes $J$ from the notation for $CF(H,J;M,\Omega,\lambda)$ and $c(\sigma;M,\Omega,\lambda)$, writing $CF(H)$ and $c(\sigma)$.
\end{notn}

By \cref{lem:EnergyRelation}, if $\sigma$ lies in $\sJ \sH^+(M,\Omega,\lambda)$, one can replace $\Lambda$ in \cref{defn:FloerchainComplex} with $\Lambda_{\geq 0}$.  To conclude, we point out some features of Pardon's virtual fundamental chains package. 

\begin{thm}\label{thm:PardonEssentials}
\begin{enumerate}
	\item \cite[Lemma 5.2.6]{Pardon_VFC} If $\overline{\sM}(\sigma,x_0,x_\ell;A)$ is a single point and it is regular, then
	\[\#_{vir}(\overline{\sM}(\sigma,x_0,x_\ell;A)) \neq 0.\]
	\item If $\#_{vir}(\overline{\sM}(\sigma,x_0,x_\ell;A)) \neq 0$, then $\overline{\sM}(\sigma,x_0,x_\ell;A)$ is non-empty.\qed
\end{enumerate}
\end{thm}

\subsection{Using radial Hamiltonians}\label{subsec:RadialHams}

It is convenient to work with Hamiltonians that only depend on $r$ near $\partial M$ to define Hamiltonian Floer chain complexes.  Let $(M,\Omega)$ be a convex symplectic domain whose boundary admits the stable Hamiltonian structure $(\Omega|_{\partial M},\alpha)$.  Recall, $\lambda = r \alpha$.  We assume that the Reeb vector field $\sR$ of the stable Hamiltonian structure is non-degenerate.  The $\Omega$-dual of $-dr$ is $\sR$.  So the Hamiltonian vector field $X_h$ of any radial Hamiltonian $h: (0,1] \times \partial M \to \IR$ satisfies: $X_{h} = \partial_rh \cdot \sR$.  So the $1$-periodic orbits of $X_h$ are of the form $(r_0,y(\partial_rh(r_0) \cdot t)) \in (0,1] \times \partial M$, where $y(t)$ is a periodic orbit of $\sR$ with period $\partial_r h(r_0)$.

\begin{defn}\label{defn:RadiallyAdHam}
For $a \in (0,1)$, a Hamiltonian $H: S^1 \times M \to \IR$ is \emph{$a$-radially admissible} if
\begin{enumerate}
	\item on $M \smallsetminus r^{-1}([0,a])$, $H$ is non-degenerate, and
	\item on $r^{-1}([a,1])$, $H = h$ for some radial function $h: [a,1] \to \IR$ that satisfies
	\begin{enumerate}
		\item $h'(r)$ is locally constant near $a$ and $1$, 
		\item $h'(a)>0$ is smaller than the smallest period of all the orbits of $\sR$,
		\item $h''(r) \geq 0$, and
		\item $h''(r) = 0$ implies that $h'(r)$ is not a period of an orbit of $\sR$.
	\end{enumerate}
\end{enumerate}
\end{defn}

\begin{defn}\label{defn:RadiallyAdFamHam}
A family of Hamiltonians $H^\sigma: \Delta^\ell \times S^1 \times M \to \IR$ is \emph{$a$-radially admissible} if
\begin{enumerate}
	\item each $H^\sigma_{\underline{e_i}}$ is $a$-radially admissible,
	\item $H^\sigma_{\underline{s}}$ is locally constant about each vertex $\underline{e_i}$ in $\Delta^\ell$, and
	\item $\partial_{\underline{s}} H^\sigma_{\gamma_i(s)} \circ \gamma_i'(s) \leq 0$ for each Morse flow line $\gamma = \sqcup \gamma_i$ in $\Delta^\ell$,
\end{enumerate}
and on $r^{-1}([a,1])$
\begin{enumerate}
	\item $H^\sigma_{\underline{s}}$ is a radial function in $r$,
	\item $\partial_{\underline{s}} \partial_r H^\sigma_{\underline{s}} \circ \gamma_i'(s) \leq 0$ for each Morse flow line $\gamma = \sqcup \gamma_i$ in $\Delta^\ell$.
\end{enumerate}
\end{defn}

\begin{rem}\label{rem:MorseBottBreakings}
A radially admissible Hamiltonian is a degenerate Hamiltonian as the orbits corresponding to Reeb orbits occur in $S^1$-families (parameterized by a choice of starting point).  As is standard, we introduce a time-dependent perturbation locally about each $S^1$-family of orbits.  The perturbed Hamiltonian has two non-degenerate orbits for each $S^1$-family of orbits of the unperturbed Hamiltonian.  Both of these orbits can be made to lie in the same $r$-slice as the original $S^1$-family.  This produces an admissible Hamiltonian (in the sense of \cref{defn:AdmissibleFamiliesOfHam}).  Similarly, perturbing near each of the vertices of a radially admissible family of Hamiltonians will yield a monotonically admissible family of Hamiltonians (in the sense of \cref{defn:MonotoneAdmissibleFamiliesOfHam}).  Whenever we refer to radially admissible data, we implicitly assume that we have introduced these needed perturbations and will not specify them.
\end{rem}

Radial Hamiltonians have filtrations on their Hamiltonian Floer chain complexes for purely topological reasons.  The filtrations follow via the integrated maximum principle, \cref{prop:MaximumPrincipleForConvexDomains}.

\begin{lem}\label{lem:TopologicalFiltration}
Let $H^\sigma$ be a radially admissible family of Hamiltonians.  If $r(x_\ell) < r(x_0)$, then $\overline{\sM}(\sigma,x_0,x_\ell)$ is empty.\qed
\end{lem}

\subsection{Action completed symplectic cohomology}\label{subsec:ActionCompletedCohomology}

We define action completed symplectic cohomology for convex symplectic domains.  The reader may wish to recall the notions of rays, mapping telescopes, and completions of mapping telescopes over $\Lambda_{\geq 0}$ in \cref{sec:AppendixHA}.

Let $K \subset M \smallsetminus \partial M$ be a compact subset and let $f: M \to \IR$ be a continuous function.  Define the lower semi-continuous function\footnote{Recall that a function $\phi: X \to \IR \cup \{\infty\}$ is \emph{lower semi-continuous} if $\phi^{-1}((c,\infty))$ is open for each $c \in \IR$.} $f_K: M \to \IR$ by
\[ f_K(p) = \begin{cases} f(p), & p \in K \\ \infty, & p \not \in K.\\ \end{cases} \]
Define $\sH^+(K \subset M, \Omega,\lambda; f)$ to be the full subcomplex of $\sH^+(M,\Omega,\lambda)$ spanned by Hamiltonians that are strictly less than $f_K$.  The opposite edge relation, denoted $\preccurlyeq$, on the vertices of $\sH^+(K \subset M, \Omega,\lambda;f)$ says $H_0 \preccurlyeq H_1$ if and only if pointwise $H_0 \leq H_1$ and (near $\partial M$) $\partial_r H_0 \leq \partial_r H_1$.

\begin{lem}\label{lem:Directedness}
The opposite edge relation endows $H_0^+(K \subset M, \Omega,\lambda; f)$ with the structure of a directed system that has countable cofinality\footnote{A \emph{directed system} is a poset $(\sD,\leq)$, which satisfies the \emph{common refinement condition}, that is, for each $A_0,A_1 \in \sD$ there exists $A_2 \in \sD$ such that $A_0,A_1 \leq A_2$.  A subposet $\sD'$ of a directed system $\sD$ is \emph{cofinal} if for each $A \in \sD$ there exists an $A' \in \sD'$ such that $A \leq A'$.  A directed system has \emph{countable cofinality} if there exists a cofinal subposet whose collection of objects is countable.}.
\end{lem}

\begin{proof}
It suffices to prove that the poset structure given by $\preccurlyeq$ admits a countable, cofinal sequence.  It further suffices to construct a cofinal sequence of admissible Hamiltonians $H_n$ such that pointwise $H_n < H_{n+1}$ and (near the boundary) $\partial_r H_0 < \partial_r H_1$.  Because we can introduce small time dependent perturbations to continuous functions to obtain non-degenerate Hamiltonians, it further suffices to construct a cofinal sequence of continuous functions $F_n: M \to \IR$ such that pointwise $F_n < F_{n+1}$ and (near $\partial M$) $\partial_r F_0 < \partial_r F_1$.  

Note, if $f$ and $g$ are two continuous functions on $M$ that both agree over $K$, then $f_K = g_K$.  So without loss of generality, assume that $f \equiv 0$ near $\partial M$.  Also any cofinal sequence $F_n$ for $(0)_K$ gives rise to a cofinal sequence for $f_K$ by considering the sequence $F_n + f$.  So it suffices to prove the claim for $f \equiv 0$.

We construct such $F_n$ by hand.  Fix a background metric on $M$ with associated distance function $dist$.  Fix $r_K \in (0,1)$ such that $r(K) < r_{K}$.  Fix an integer $N \in \IN$ such that $dist(K,r^{-1}(r_K)) > 1/N$.  For $n \geq N$, we define the following compact subsets:
\[ C_n = \{ p \in M \mid dist(K,p) \leq 1/n \} \]
and
\[ A_n = M \smallsetminus (\intrr(C_n) \cup r^{-1}((r_K,1])).\]
Define continuous functions $g_n: [0,1/n] \to \IR$ that satisfy:
\begin{enumerate}
	\item $g_n(0) = -1/n$,
	\item $g_n(1/n) = n$,
	\item $g_n$ is increasing, and
	\item $g_n < g_{n+1}$.
\end{enumerate}
Define continuous functions $f_n: [r_K,1] \to \IR$ that satisfy:
\begin{enumerate}
	\item $f_n \equiv n$ locally about $r_K$, 
	\item $f_n$ is a linear function of slope $m_n$ locally about $1$,
	\item $f'_n \geq 0$,
	\item $f_n < f_{n+1}$,
	\item $m_n < m_{n+1}$, and
	\item $\lim_{n \to \infty} m_n = +\infty$.
\end{enumerate}
Define continuous functions $F_n: M \to \IR$ by
\[F_n(p) = \begin{cases} g^n \circ d(K,p), & p \in C_n \\ n, & p \in A_n \\ f^n \circ r(p), & p \in [r_K,1]. \\ \end{cases} \]

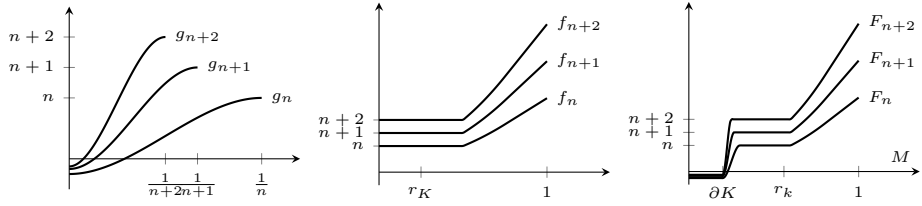
\begin{figure}
\begin{tikzpicture}
\begin{axis}[width=.333\linewidth, axis lines=center, xtick={1/4,1/3,1/2},xticklabels= {$\frac{1}{n+2}$,$\frac{1}{n+1}$, $\frac{1}{n}$}, ytick={2,3,4},yticklabels={$n$,$n+1$,$n+2$}, xmin=0, xmax=.6, ymin=-1, ymax=5,label style={font=\tiny}, tick label style={font=\tiny}]
\addplot[color=black, domain=0:1/2, samples=100, smooth, thick]{-40*(x)^3+30*(x)^2-1/2} node[right,pos=1] {\tiny $g_n$};
\addplot[color=black, domain=0:1/3, samples=100, smooth, thick]{-180*(x)^3+90*(x)^2-1/3} node[right,pos=1] {\tiny $g_{n+1}$};
\addplot[color=black, domain=0:1/4, samples=100, smooth, thick]{-8*68*(x)^3+3*68*(x)^2-1/4} node[right,pos=1] {\tiny $g_{n+2}$};
\end{axis}
\end{tikzpicture}
\begin{tikzpicture}[declare function={
    func0(\x)= (\x<=2) * (2)   +
     and(\x>2, \x<=3) * (\x*\x*2/6+2/3)     +
     and(\x>3,  \x<=4) * (2*\x-7*2/6);
  }, declare function={
    func1(\x)= (\x<=2) * (3)   +
     and(\x>2, \x<=3) * (\x*\x*3/6+3/3)     +
     and(\x>3,  \x<=4) * (3*\x-3*7/6);
  }, declare function={
    func2(\x)= (\x<=2) * (4)   +
     and(\x>2, \x<=3) * (\x*\x*4/6+4/3)     +
     and(\x>3,  \x<=4) * (4*\x-4*7/6);
  }]
\begin{axis}[width=0.333\linewidth, axis lines=center, xtick={1,4},xticklabels= {$r_K$,1}, ytick={2,3,4},yticklabels={$n$,$n+1$,$n+2$}, xmin=0, xmax=5.5, ymin=-1, ymax=13,label style={font=\tiny}, tick label style={font=\tiny}]
\addplot[color=black, domain=0:4, samples=100, smooth, thick]{func0(x)} node[right,pos=1] {\tiny $f_n$};
\addplot[color=black, domain=0:4, samples=100, smooth, thick]{func1(x)} node[right,pos=1] {\tiny $f_{n+1}$};
\addplot[color=black, domain=0:4, samples=100, smooth, thick]{func2(x)} node[right,pos=1] {\tiny $f_{n+2}$};
\end{axis}
\end{tikzpicture}
\begin{tikzpicture}[declare function={
    Func0(\x)= (\x<=1) * -1/2 +
    and(\x>1, \x<=3/2) * (-40*(\x-1)*(\x-1)*(\x-1)+30*(\x-1)*(\x-1)-1/2) +
     and(\x>3/2,  \x<=3) * (2)   +
     and(\x>3, \x<=4) * ((\x-1)*(\x-1)*2/6+2/3)     +
     and(\x>4,  \x<=5) * (2*(\x-1)-7*2/6);
  }, declare function={
    Func1(\x)= (\x<=1) * -1/3 +
    and(\x>1, \x<=4/3) * (-180*((\x-1))^3+90*((\x-1))^2-1/3) +
    and(\x>4/3, \x<=3) * (3)   +
     and(\x>3, \x<=4) * ((\x-1)*(\x-1)*3/6+3/3)     +
     and(\x>4,  \x<=5) * (3*(\x-1)-3*7/6);
  }, declare function={
    Func2(\x)=  (\x<=1) * -1/4 +
    and(\x>1, \x<=5/4) * (-8*68*((\x-1))^3+3*68*((\x-1))^2-1/4) +
    and(\x>5/4, \x<=3) * (4)   +
     and(\x>3, \x<=4) * ((\x-1)*(\x-1)*4/6+4/3)     +
     and(\x>4,  \x<=5) * (4*(\x-1)-4*7/6);
  }]
\begin{axis}[width=.333\linewidth,axis lines=center, xtick={1,2.8,5},xticklabels= {$\partial K$, $r_k$, 1}, xlabel={$M$}, ytick={2,3,4},yticklabels={$n$,$n+1$,$n+2$}, xmin=0, xmax=6.8, ymin=-1, ymax=13,label style={font=\tiny}, tick label style={font=\tiny}]
\addplot[color=black, domain=0:5, samples=100, smooth, thick]{Func0(x)} node[right,pos=1] {\tiny $F_n$};
\addplot[color=black, domain=0:5, samples=100, smooth, thick]{Func1(x)} node[right,pos=1] {\tiny $F_{n+1}$};
\addplot[color=black, domain=0:5, samples=100, smooth, thick]{Func2(x)} node[right,pos=1] {\tiny $F_{n+2}$};
\end{axis}
\end{tikzpicture}
\caption{An example of a choice of functions $g_n$, $f_n$, and $F_n$ for \cref{lem:Directedness}.}
\label{fig:Directednessgn}
\end{figure}

For all $n$, $F_n < F_{n+1} < (0)_K$ and near $\partial M$, $\partial_r F_n < \partial_r F_{n+1}$.  It remains to show that for each $H$ in $\sH_0^+(K \subset M, \Omega,\lambda; 0)$ there exists $n\gg 0$ so that pointwise $H < F_n$ and (near $\partial M$) $\partial_r H < \partial_r F_n$.  Since each $C_{n}$ is compact, there exists some fixed $n_0$ so that $H|_{C_{n_0}} < -1/n_1$ for all $n_1 \gg n_0$.  Since $M$ is compact, we have that $H \leq n_2$ for some $n_2$.  Again by compactness, there exists $n_3$ such that $(\partial_r H)(1) < m_{n_3}$.  Setting $N = n_2 + n_3 \gg n_0$ gives that pointwise $H < F_N$ and (near $\partial M$) $\partial_r H < \partial_r F_N$, as desired.
\end{proof}

Continuing our construction, let $(H_n,J_n)$ be a sequence of vertices in $\sJ \sH_0^+(K \subset M,\Omega,\lambda;f)$ such that
\begin{enumerate}
	\item there exists $1$-simplices $\sigma_n \in   \sJ \sH_1^+(M,\Omega,\lambda)$ with $(H^{\sigma_n}_{\underline{e_1}}, J^{\sigma_n}_{\underline{e_1}}) = (H_n,J_n)$ and $(H^{\sigma_n}_{\underline{e_0}}, J^{\sigma_n}_{\underline{e_0}}) = (H_{n+1},J_{n+1})$, and
	\item the sequence $H_n$ is cofinal.
\end{enumerate}
By \cref{lem:Directedness}, such a sequence exists.  This determines a ray of chain complexes over $\Lambda_{\geq 0}$,
\[ \sC\sF(\{H_n\}) \coloneqq \xymatrix{ CF(H_0,J_0) \ar[r]^{c(\sigma_0)} & CF(H_1,J_1) \ar[r]^{c(\sigma_1)} & CF(H_2,J_2) \ar[r]^{c(\sigma_2)} & \cdots}. \]
By \cref{lem:Directedness} and \cref{lem:ConnectivityOfFloerData} (In particular, one has to utilize $2$-simplices and $3$-simplices.), the quasi-isomorphism type of $\widehat{\Tel}(\sC\sF(\{H_n\}))$ is independent of the choice of sequence $\sigma_n$ (and thus also independent of the choice of sequence $(H_n,J_n)$), for example, see \cite[Section 3.3]{Varolgunes_MayerVietorisPropertyForRelativeSymplecticCohomology}.  So we can make the following definition.

\begin{defn}\label{defn:CompletedSH}
The \emph{action completed symplectic cohomology of $(M,\Omega,\lambda)$ with respect to $K$ and $f$} is the $\Lambda_{\geq 0}$-module 
\[\widehat{SH}(K \subset M; f) \coloneqq H(\widehat{\Tel}(\sC\sF(\{H_n\}))).\]
\end{defn}

The above independence gives flexibility to future arguments.  Different sequences of Hamiltonians will be more amenable to establishing different properties of action completed symplectic cohomology.

We conclude by discussing a variant of \cref{defn:CompletedSH}.  Let $\sH^+(K \subset M, \Omega,\lambda; f,\tau)$ to be the full subcomplex of $\sH^+(M,\Omega,\lambda)$ spanned by Hamiltonians that are strictly less than $f_K$ and have slope $\partial_r H = \tau $ near the boundary of $M$.  Analogous to \cref{lem:Directedness}, the opposite edge relation endows $\sH_0^+(K \subset M, \Omega,\lambda; f,\tau)$ with the structure of a directed system that has countable cofinality.\footnote{Indeed, the only element of the proof that changes is the definition of the functions $f_n$.  One would change $f_n: [r_K,1] \to \IR$ to now satisfy:
\begin{enumerate}
	\item $f_n \equiv n$ locally about $r_K$, 
	\item $f_n$ is a linear function of slope $\tau$ locally about $1$,
	\item $f'_n \geq 0$, and
	\item $f_n < f_{n+1}$.
\end{enumerate} }
As with $\sH^+(K \subset M, \Omega,\lambda;f)$, we can consider a cofinal sequence of Hamiltonians $H_{n,\tau}$ in $\sH^+(K \subset M, \Omega,\lambda;f,\tau)$ connected by monotonically admissible $1$-simplicies $\sigma_{n,\tau}$ and obtain a well-defined invariant.

\begin{defn}\label{defn:CompletedSH}
The \emph{action completed symplectic cohomology of $(M,\Omega,\lambda)$ of slope $\tau$ with respect to $K$ and $f$} is the $\Lambda_{\geq 0}$-module 
\[\widehat{SH}(K \subset M; f,\tau) \coloneqq H(\widehat{\Tel}(\sC\sF(\{H_{n,\tau}\}))).\]
\end{defn}

In \cref{prop:PropertiesOfSH} \cref{prop:InvarianceOfSlope}, we show that
\[ \widehat{SH}(K \subset M; f,\tau) \cong \widehat{SH}(K \subset M; f ). \]
We introduce $\widehat{SH}(K \subset M; f,\tau)$ (despite this isomorphism) because in the setting of our main result, we can construct cofinal sequences of Hamiltonians in $\sH^+(K \subset M, \Omega,\lambda;f,\tau)$ (but not in $\sH^+(K \subset M, \Omega,\lambda;f)$) that satisfy certain index-bounded assumptions.  Such sequences of Hamiltonians are used to establishing a rescaling isomorphism for action completed symplectic cohomology in the Calabi-Yau setting, see \cref{sec:RescalingIso}.


\section{Properties of action completed symplectic cohomology}\label{sec:PropertiesOfActionCompletedSH}

We establish the following properties of action completed symplectic cohomology.

\begin{prop}\label{prop:PropertiesOfSH}
\begin{enumerate}
	\item  \label{prop:InvarienceOfFunction}Given any two continuous functions $f: M \to \IR$ and $g: M \to \IR$, there is an isomorphism that is induced from a zig-zag of continuation maps
	\[ \widehat{SH}(K \subset M; f) \otimes \Lambda \cong \widehat{SH}(K \subset M; g) \otimes \Lambda.\]
	\item \label{prop:VanishingForStDisp}If $K$ is a stably displaceable subset of $(M,\Omega)$, then for every continuous function $f$
	\[ \widehat{SH}(K \subset M; f) \otimes \Lambda = 0.\footnote{The result also holds with $\widehat{SH}(K \subset M; f)$ replaced by $\widehat{SH}(K \subset M; f,\tau)$}\]
	\item \label{prop:LES}There is a function $f$ (see \cref{con:RadialCofinal}) so that there exists a long exact sequence
	\[ \xymatrix{ \cdots \ar[r] & H(M;\Lambda) \ar[r] & \widehat{SH}(M_a \subset M; f) \otimes \Lambda \ar[r] & \widehat{SH}_+(M_a \subset M;f) \ar[r] & \cdots }, \]
	where $\widehat{SH}_+(M_a \subset M;f)$ is a completion of a complex that is generated pairs of orbits that correspond to the Reeb orbits of $\partial M$.
	\item \label{prop:InvarianceOfSlope}There is an isomorphism for all $\tau > 0$
	\[ \widehat{SH}(K \subset M; f,\tau) \cong \widehat{SH}(K \subset M; f). \]
\end{enumerate}
\end{prop}

We prove each item in turn.  We will assume that the boundaries of our convex symplectic domains have non-degenerate Reeb vector fields.

\subsection{Partial invariance under changing the function}\label{subsec:FunctionInvarience}

We discuss the dependence of $\widehat{SH}(K \subset M; f)$ on the function $f$, proving \cref{prop:PropertiesOfSH} \cref{prop:InvarienceOfFunction}.  We begin with the following.

\begin{lem}\label{lem:InvarienceOfFunction1}
Given a continuous function $f: M \to \IR$ and a constant $c\geq0$, there is an isomorphism induced from a continuation map
\[ \widehat{SH}(K \subset M; f) \otimes \Lambda  \to  \widehat{SH}(K \subset M; f+c) \otimes \Lambda. \]
\end{lem}

\begin{proof}
Let $H_n$ be a cofinal sequence of Hamiltonians for $f_K$.  Define $\sigma_n \in \sJ \sH_1^+(M,\Omega,\lambda)$ using \cref{con:TheBump},
\[ H^{\sigma_n}_s = \ell(s) \cdot H_{n+1} + (1- \ell(s)) \cdot H_n. \]
This data computes $\widehat{SH}(K \subset M; f)$.  The sequences $H_n' \coloneqq H_n + c$ and $\sigma_n' \in \sJ \sH_1^+(M,\Omega,\lambda)$ with associated families $H^{\sigma_n}+c$ compute $\widehat{SH}(K \subset M;f+c)$.  There is a bijection between the data of $CF(H_n)$ and the data of $CF(H_n')$.  We upgrade these bijections to isomorphisms.

Consider the continuation map of the $1$-simplex $\tau_n$ with $H^{\tau_n}_{\underline{e_0}} = H_{n}'$ and $H^{\tau_n}_{\underline{e_1}} = H_n$ constructed using \cref{con:TheBump}.  So $H^{\tau_n}_s = \ell(s) \cdot c + H_n$, and $X_{H^{\tau_n}}$ is independent of $s$.  Consider the constant Floer trajectories, which are isolated, regular, and have topological energy equal to $c$.  By \cref{cor:EnergyRelation1}, all other Floer trajectories have topological energy strictly greater than $c$.  So by \cref{thm:PardonEssentials},
\[c(\tau_n)(\cdot) = (\cdot) \cdot T^{c} + \mbox{higher order terms in }T\footnote{or rather this quantity scaled by an element of $\IQ^\times$ which depends on the input orbit.  For conciseness, we omit this constant here and in future arguments.}.\]
We used the chain level bijection from above to specify our map.  These continuation maps fit into a homotopy commutative diagram:
\begin{align*} \xymatrix{ CF(H_0) \ar[r]^{c(\sigma_0)} \ar[d]^{c(\tau_0)} \ar[rd]^{\nu_0} & CF(H_1) \ar[r]^{c(\sigma_1)} \ar[d]^{c(\tau_1)}  \ar[rd]^{\nu_1} & \cdots \\ 
CF(H_0') \ar[r]^{c(\sigma_0')} & CF(H_1') \ar[r]^{c(\sigma_1)} & \cdots.\\} \end{align*}
with homotopies $\nu_n: CF(H_n) \to CF(H_{n+1}')[1]$.  To get this homotopy commutative diagram, consider the $I \times I$ family of Hamiltonians
\[ H^{\nu_n}_{\underline{s}} = (1-\ell(s_1)) \cdot H_n + \ell(s_1) \cdot H_{n+1} + \ell(s_2) \cdot c \]
with $\underline{s} = (s_1,s_2) \in I \times I$ and $\ell$ as in \cref{con:TheBump}.  One divides this $I \times I$ family into two $2$-simplicies glued along the diagonal of $I \times I$.  As a matrix,
\[ \partial_{\underline{s}} H^{\nu_n}_{\underline{s}} = \begin{pmatrix} \ell'(s_1) \cdot (H_{n+1} - H_{n}) & \ell'(s_2) \cdot c \\ \end{pmatrix}. \]
So arguing as in \cref{cor:EnergyRelation1}, $\nu_n(\cdot)$ has valuation strictly greater than $c$.  Equivalently, $\nu_n(\cdot) \in \Lambda_{>c} \cdot CF(H_{n+1}')[1]$.  By the discussion at the end of \cref{subsec:Conventions}, we obtain a morphism of the associated mapping telescopes of the rays of $\sigma_n$ and $\sigma_n'$ that is given by
\[(\cdot) \mapsto (\cdot) \cdot T^{c} + \mbox{higher order terms in }T.\]
From this description, we see that after tensoring with $\Lambda$, this map induces an isomorphism on the associated completed mapping telescopes.  This proves the claim.
\end{proof}

\begin{lem}\label{lem:InvarienceOfFunction2}
Given a continuous function $f: M \to \IR$ and a constant $a$ such that $f< a$, there is an isomorphism that is induced from a continuation map
\[ \widehat{SH}(K \subset M; f) \otimes \Lambda \to \widehat{SH}(K \subset M; a) \otimes \Lambda. \]
\end{lem}

\begin{proof}
Fix constants $a<b<c$ so that $f < a < f + b < c$.  As in the proof of \cref{lem:InvarienceOfFunction1}, let $H_n$ be a cofinal sequence of Hamiltonians for $f_K$.  Define $\sigma_n \in \sJ \sH_1^+(M,\Omega,\lambda)$ using \cref{con:TheBump}.  So
\[ H^{\sigma_n}_s = \ell(s) \cdot H_{n+1} + (1- \ell(s)) \cdot H_n. \]
This data computes $\widehat{SH}(K \subset M; f)$.  Similarly, let $\widetilde{H}_n$ and $\widetilde{\sigma}_n$ denote the data that computes $\widehat{SH}(K \subset M; a)$, defining $\widetilde{\sigma}_n$ via \cref{con:TheBump}.  Assume without loss of generality that pointwise $H_n < H_n'$ and near $\partial M$, $\partial_r H_n = \partial_r H_n'$ for all $n$.  Again, as in the proof of \cref{lem:InvarienceOfFunction1}, the data $H_n' \coloneqq H_n+b$ and $H^{\sigma_n'} \coloneqq H^{\sigma_n} + b$ computes $\widehat{SH}(K \subset M,;f+b)$.  Similarly, the data $\widetilde{H}_n' \coloneqq \widetilde{H}_n+c-a$ and $H^{\widetilde{\sigma}_n'
} \coloneqq H^{\widetilde{\sigma}_n} + c-a$ computes $\widehat{SH}(K \subset M; c)$.

Using \cref{lem:ConnectivityOfFloerData}, we may fix continuations $\tau_n^i$ and homotopies $\nu_n^i$ for $i = 0,1,2$ and obtain a homotopy commutative diagram
\[ \xymatrix{  \ar[d] &  \ar[d] &  \ar[d] &  \ar[d] & \\
CF(H_n) \ar[r]^{c(\tau_n^0)} \ar[d]^{c(\sigma_n)} \ar[rd]^{\nu_n^0} & CF(\widetilde{H}_n) \ar[r]^{c(\tau_n^1)} \ar[d]^{c(\widetilde{\sigma}_n)} \ar[rd]^{\nu_n^1} & CF(H_n') \ar[r]^{c(\tau_n^2)} \ar[d]^{c(\sigma_n')} \ar[rd]^{\nu_n^2} & CF(\widetilde{H}_n') \ar[d]^{c(\widetilde{\sigma}_n')}\\
CF(H_{n+1}) \ar[r]^{c(\tau_{n+1}^0)} \ar[d] & CF(\widetilde{H}_{n+1}) \ar[r]^{c(\tau_{n+1}^1)} \ar[d] & CF({H}_{n+1}') \ar[r]^{c(\tau_{n+1}^2)} \ar[d] & CF(\widetilde{H}_{n+1}') .\ar[d] \\
 &  &  & \\} \]
One can fix $\tau_n^i$ and $\nu_n^i$ arbitrarily so long as they are monotonically admissible.  This diagram induces maps on completed mapping telescopes
\[ \xymatrix{\widehat{\Tel}(\sC\sF(\{H_n\})) \ar[r]^{\Psi_0} & \widehat{\Tel}(\sC\sF(\{\widetilde{H}_n\})) \ar[r]^{\Psi_1} & \widehat{\Tel}(\sC\sF(\{H_n'\})) \ar[r]^{\Psi_2}&  \widehat{\Tel}(\sC\sF(\{\widetilde{H}_n'\})) \\}. \]
To prove the claim, it suffices to show that $\Psi_1 \circ \Psi_0$ and $\Psi_2 \circ \Psi_1$ are both isomorphisms after tensoring with $\Lambda$ and passing to cohomology.  By symmetry, it suffices to prove this for $\Psi_1 \circ \Psi_0$.

Using \cref{lem:ConnectivityOfFloerData}, construct a homotopy commutative diagram
\[ \xymatrix{
& CF(H_0)  \ar[rr]^{c(\sigma_n)} \ar'[d][dd]_{\Ione} \ar[dl]_{c(\tau_0^{1}) \circ c(\tau_0^0)}
&& CF(H_1) \ar[rr]^{c(\sigma_1)} \ar'[d][dd]_{\Ione} \ar[dl]_{c(\tau_1^{1}) \circ c(\tau_1^0)}
&& \cdots 
\\
CF(H_0')  \ar[rr]^{c(\sigma_0')} \ar[dd]_{\Ione}
&& CF(H_1')  \ar[rr]^{c(\sigma_1')} \ar[dd]_{\Ione}
&& \cdots  
\\
& CF(H_0) \ar'[r][rr]^{c(\sigma_0)} \ar[dr]^{\nu_0} \ar[dl]_{c(\tau_0)}
&& CF(H_1)  \ar[rr]^{c(\sigma_1)} \ar[dr]^{\nu_1} \ar[dl]_{c(\tau_1)}
&& \cdots
\\
CF(H_0') \ar[rr]^{c(\sigma_0')}
&& CF(H_1')  \ar[rr]^{c(\sigma_1')}
&& \cdots, 
}\]
where, $\tau_i$ is the continuation considered in \cref{lem:InvarienceOfFunction1}, and our unspecified homotopies (and higher homotopies) can be chosen arbitrarily so long as they are monotonically admissible.  In this manner, we obtain a homotopy commutative diagram
\[ \xymatrix{\widehat{\Tel}(\sC\sF(\{H_n'\})) \ar[d]& \widehat{\Tel}(\sC\sF(\{H_n\})) \ar[l] \ar[d] & \\ \widehat{\Tel}(\sC\sF(\{H_n'\})) & \widehat{\Tel}(\sC\sF(\{H_n\})), \ar[l] \\} \]
where the top map is $\Psi_1 \circ \Psi_0$, the bottom map is the map from \cref{lem:InvarienceOfFunction1}, and the vertical maps are the identities.  So $\Psi_1 \circ \Psi_0$ must be an isomorphism after tensoring with $\Lambda$ and passing to cohomology.  This completes the proof.
\end{proof}

We now give the proof of \cref{prop:PropertiesOfSH} \cref{prop:InvarienceOfFunction}.

\begin{proof}
Fix a constant $c$ such that $c>f$ and $c>g$.  By \cref{lem:InvarienceOfFunction2}, we have a zig-zag of continuations maps, inducing isomorphisms
\[ \xymatrix{ \widehat{SH}(K \subset M; f)\otimes \Lambda \ar[r] & \widehat{SH}(K \subset M; c) \otimes \Lambda & \widehat{SH}(K \subset M; g) \otimes \Lambda. \ar[l]& } \]
\end{proof}

\begin{rem}
Notice that the proof of \cref{prop:PropertiesOfSH} \cref{prop:InvarienceOfFunction} also proves the natural invariance of $\widehat{SH}(K \subset M; f,\tau) \otimes \Lambda$ under changing the function $f$.
\end{rem}

\subsection{Vanishing for stably displaceable subsets}\label{subsec:VanishingForStDisp}

The group $\widehat{SH}(K \subset M; f)$ encodes dynamical information about the subset $K$ in $M$.  The key consequence is \cref{prop:PropertiesOfSH} \cref{prop:VanishingForStDisp}.  We elaborate below.

\begin{defn}\label{defn:HamiltonianDisplaceable}
A subset $B$ of a symplectic manifold $(M,\Omega)$ is \emph{Hamiltonian displaceable} if there exists a compactly supported Hamiltonian diffeomorphism $\phi: M \to M$ that is supported away from the boundary of $M$ such that $\phi(B) \cap B = \varnothing$.
\end{defn}
Let us fix polar coordinates $(\sigma, \tau)$ for $\IR \times  S^1$. 
\begin{defn}\label{defn:StablyDisplaceable}
A subset $B$ of a symplectic manifold $(M,\Omega)$ is \emph{stably displaceable} if the subset $B \times S^1$ of $(M \times \IR \times S^1, \Omega \oplus (d\sigma \wedge d\tau))$ is Hamiltonian displaceable.
\end{defn}

Using \cref{prop:PropertiesOfSH} \cref{prop:InvarienceOfFunction}, it suffices to prove \cref{prop:PropertiesOfSH} \cref{prop:VanishingForStDisp} $f \equiv 0$, which follows from the same line of arguing given in \cite[Section 4.2]{Varolgunes_Thesis}.

\subsection{A long exact sequence}\label{sec:LES}

For Liouville domains, the action long exact sequence relates the (Morse) cohomology and the symplectic cohomology of the Liouville domain.  Here we deduce the analogous long exact sequence for action completed symplectic cohomology, \cref{prop:PropertiesOfSH} \cref{prop:LES}.  To construct it, we construct a specific function $f$ on $M$ that we will use to relate $H(M;\Lambda)$ with $\widehat{SH}(K \subset M; f) \otimes \Lambda$ for $K = M_a$ as in \cref{notn:rInverse}.

\begin{con}\label{con:RadialCofinal}
We fix $a > \delta > 0$.  Define a function $\widetilde{f}: M \to \IR$ by
\[ \widetilde{f}(x) = \begin{cases} -\delta, & x \in M_{a-\delta} \\ -\sqrt{\delta^2-(r(x)-(a-\delta))^2}, & a-\delta \leq r(x) \leq a \\ 0, & r(x) \geq a. \\ \end{cases} \]
For $0 \leq s < \delta$, consider the family $\widetilde{f}_s: M \to \IR$ given by
\[ \widetilde{f}_s(x) = \begin{cases} \widetilde{f}(x), & x \in M_{a-\delta+s} \\ \widetilde{f}'(a-\delta+s) \cdot (r(x) - (a - \delta+s)) + \widetilde{f}(a-\delta+s), & r(x) \geq a- \delta +s. \\ \end{cases}\]
Fix $\delta > \varepsilon > 0$ so that for all $r \leq a-\delta+\varepsilon$, $\widetilde{f}'(r)$ is strictly less than all periods of the Reeb orbits of $\sR$.  Let $f: M \to \IR$ satisfy:
\begin{enumerate}
	\item $f$ is a negative $C^2$-small Morse function on $M_{a-\delta}$,
	\item $f \leq \widetilde{f}$, and
	\item $f \equiv \widetilde{f}$ for $r(x) \geq a-\delta+\varepsilon$.
\end{enumerate}
For $\varepsilon \leq s < \delta$, define a family
\[ f_s(x) = \begin{cases} f(x), & x \in M_{a-\delta+s} \\ \widetilde{f}_s(x), & r(x) \geq a - \delta + s. \\ \end{cases} \]
Fix a strictly increasing sequence of real numbers $\varepsilon < s_0 < s_1 < \cdots < s_n < \cdots$ that converges to $\delta$.  Define $H_n:M \to \IR$ by (or rather, we consider an appropriately small smoothing of these Hamiltonians) $H_n \coloneqq f_{s_n} - {1}/{n}$.

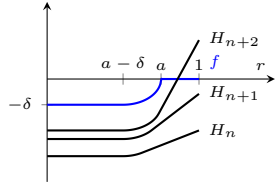
\begin{figure}
\begin{tikzpicture}[declare function={
     func(\x)= (\x<=1/2) * (-1/4)   +
      and(\x>1/2, \x<0.75) * -sqrt((1/4)^2-(\x-(1/2))^2)     +
      and(\x>=0.75,  \x<=1) * (0);
	}, declare function={
	    func0(\x)= (\x<=10/16) * (func(\x)-1/2)   +
     and(\x>10/16, \x<=1) * (func(10/16)-((2/16)/func(10/16))*(\x-10/16)-1/2);
  }, declare function={
	    func1(\x)= (\x<=11/16) * (func(\x)-1/3)   +
     and(\x>11/16, \x<=1) * (func(11/16)-((3/16)/func(11/16))*(\x-11/16)-1/3);
  }, declare function={
	    func2(\x)= (\x<=47/64) * (func(\x)-1/4)   +
     and(\x>47/64, \x<=1) * (func(47/64)-((15/64)/func(47/64))*(\x-47/64)-1/4);
  }]\begin{axis}[width=0.33\linewidth, axis lines=center, xtick={1/2,3/4,1}, xticklabels = {$a-\delta$,$a$,$1$}, xlabel={$r$}, ytick={-1/4}, yticklabels={$-\delta$}, xmin=0, xmax=1.5, ymin=-1, ymax=.75, label style={font=\tiny}, x tick label style={font=\tiny,above=2}, y tick label style={font=\tiny}]
  \addplot[color=blue, domain=0:1, samples=300, smooth, thick]{func(x)} node[xshift=6, above ,pos=1] {\tiny $f$};
\addplot[color=black, domain=0:1, samples=100, smooth, thick]{func0(x)} node[right,pos=1] {\tiny $H_n$};
\addplot[color=black, domain=0:1, samples=100, smooth, thick]{func1(x)} node[right,pos=1] {\tiny $H_{n+1}$};
\addplot[color=black, domain=0:1, samples=100, smooth, thick]{func2(x)} node[right,pos=1] {\tiny $H_{n+2}$};
\end{axis}
\end{tikzpicture}
\caption{Depictions of the choice of functions $f$, and $H_n$ for \cref{con:RadialCofinal}.}
\label{fig:Directednessgn}
\end{figure}

Notice that the sequence $H_n$ is cofinal sequence among admissible Hamiltonians that are less than $f_{M_a}$, each $H_n$ is $(a-\delta+\varepsilon)$-radially admissible (for a generic choice of sequence $s_n$), and $\partial_r H_{n+1} > \partial_r H_{n}$.  Using \cref{con:TheBump}, build $(a-\delta+\varepsilon)$-radially admissible families of Hamiltonians $H^{\sigma_n}$ with $H^{\sigma_n}_{\underline{e_1}} = H_n$ and $H^{\sigma_n}_{\underline{e_0}} = H_{n+1}$.
\end{con}

We now construct the desired long exact sequence.  Assume the notation and data from \cref{con:RadialCofinal}.  $CF(H_n)$ is generated by the critical points of $f$ on $M_{a-\delta}$ and the pairs of orbits that correspond to Reeb orbits of $\sR$ with periods less than $\widetilde{f}_{a-\delta+s_n}'(a-\delta+\varepsilon)$.  By \cref{lem:TopologicalFiltration}, the critical points form a subcomplex
\[ CF_0(H_n) \hookrightarrow CF(H_n) \]
with quotient
\[ CF_+(H_n) \coloneqq CF(H_n)/CF_0(H_n) \]
generated by the orbits that correspond to Reeb orbits.  Since each $\sigma_n$ is radially admissible, their associated continuation maps preserve this filtration.  So we obtain (a map of) short exact sequences
\[\xymatrix{ 0 \ar[r] & \Tel( \sC\sF_0(\{H_n\})) \ar[r] \ar[d]& \Tel(\sC\sF(\{H_n\}))  \ar[r]\ar[d] & \Tel( \sC\sF_+(\{H_n\})) \ar[r] \ar[d] & 0 \\ 0 \ar[r] & \colim( \sC\sF_0(\{H_n\})) \ar[r] & \colim(\sC\sF(\{H_n\}))  \ar[r] & \colim( \sC\sF_+(\{H_n\})) \ar[r] & 0. } \]
Here the $\colim$ means the ordinary colimit of our directed systems.\footnote{The reason that we consider both colimits and mapping telescopes (that is, homotopy colimits) is two fold.
\begin{itemize}
	\item There associated completions are quasi-isomorphic via \cref{lem:TelToColimQI}.  So when constructing our desired long exact sequence, we can choose to work with either.
	\item The chain complexes $\colim(\sC\sF_0(\{H_n\}))$ and $\colim(\sC\sF_+(\{H_n\}))$ are, actually, generated by the critical points of $f$ and the pairs of orbits that correspond to Reeb orbits of the boundary's stable Hamiltonian structure respectively.  However, these descriptions are not needed to construct our long exact sequence.  
\end{itemize}}
The right-hand-terms are filtered colimits of finitely generated, free $\Lambda_{\geq}$-modules.  So by \cite[\href{https://stacks.math.columbia.edu/tag/058G}{Tag 058G}]{StacksProject}, they are flat $\Lambda_{\geq 0}$-modules.  So by \cite[\href{https://stacks.math.columbia.edu/tag/0315}{Tag 0315}]{StacksProject}, we have short exact sequences of completions with vertical maps being quasi-isomorphisms (by \cref{lem:TelToColimCompleteQI}).
\[\xymatrix{ 0 \ar[r] & \widehat{\Tel}( \sC\sF_0(\{H_n\})) \ar[r] \ar[d]& \widehat{\Tel}(\sC\sF(\{H_n\}))  \ar[r]\ar[d] & \widehat{\Tel}( \sC\sF_+(\{H_n\})) \ar[r] \ar[d] & 0 \\ 0 \ar[r] & \widehat{\colim}( \sC\sF_0(\{H_n\})) \ar[r] & \widehat{\colim}(\sC\sF(\{H_n\}))  \ar[r] & \widehat{\colim}( \sC\sF_+(\{H_n\})) \ar[r] & 0. } \]
By similar arguments as in \cref{lem:InvarienceOfFunction1} and \cite[Theorem 10.7.1]{Pardon_VFC}, the cohomology of the left-hand-terms tensored with $\Lambda$ is the cohomology of $M_{a}$ with coefficients in $\Lambda$, which is also the cohomology of $M$ with coefficients in $\Lambda$.  From this discussion, we can deduce \cref{prop:PropertiesOfSH} \cref{prop:LES} by defining $\widehat{SH}_+(M_a \subset M)$ to be the cohomology of the right-hand-terms tensored with $\Lambda$.  The long exact sequence does not depend on any of our choices; however, we do not prove or further discuss this independence as it is not needed.

\begin{rem}
To explicitly construct the complexes in \cref{prop:PropertiesOfSH} \cref{prop:LES}, one needs the sequence of Hamiltonians $H_n$ to be connected via radially admissible families of Hamiltonians.  There are several possible choices of such sequences and we only give one here.  However, the reader should note that a naive construction with $f$ being given by the zero function will not necessarily produce the desired complexes as the values of the $r$-derivatives of the associated cofinal sequence of Hamiltonians might eventually decrease with increasing $n$.  Consequently, the associated continuation maps need not preserve the filtration.
\end{rem}

\subsection{Invariance of the slope near the boundary}

We show that $\widehat{SH}(K \subset M; f,\tau)$ does not depend on $\tau$, proving \cref{prop:PropertiesOfSH} \cref{prop:InvarianceOfSlope}.

\begin{proof}

Let $H_{n,\tau}$ be a cofinal sequence of Hamiltonians in $\sH^+(K \subset M; f,\tau)$.  Without loss of generality, assume that
\[H_{n,\tau} = \tau \cdot r + n\] 
on $\{ 1-\delta \leq r \leq 1\}$.  
Let $\sigma_{n,\tau}$ be $1$-simplices in $\sJ\sH^+(M,\Omega,\lambda)$ with $H^{\sigma_{n,\tau}}_{\underline{e_0}} = H_{n+1,\tau}$ and $H^{\sigma_{n,\tau}}_{\underline{e_1}} = H_{n,\tau}$ as in \cref{con:TheBump}.  This means
\[ H^{\sigma_{n,\tau}}_s = \ell(s) + n + \tau \cdot r \]
on $\{ 1 -\delta \leq r \leq 1\}$ where $\ell$ is the cut-off function in \cref{con:TheBump}.  Define $H_n$ and $\sigma_n$ as follows.  First, fix $s_n<\delta$ such that
\[ \widetilde{f}'(1-\delta+s_n) = \tau\cdot n, \]
where $\widetilde{f}$ is as in \cref{con:RadialCofinal} with $a = 1$.  For convenience, define
\[ c_n \coloneqq - \widetilde{f}(1-\delta+s_1) + \tau(1-\delta+s_1) + n.\]
Define
\[ H_n(x) = \begin{cases} H_{n,\tau}(x), & r(x) \leq 1 - \delta + s_1 \\
\widetilde{f}(x) + c_n, & 1-\delta+s_1 \leq r(x) \leq 1-\delta+s_n \\
n \cdot \tau \cdot r - n \cdot \tau \cdot(1-\delta+s_n) + \widetilde{f}(1-\delta+s_n) + c_n, & 1 - \delta +s_n \leq r(x). \\ \end{cases} \]

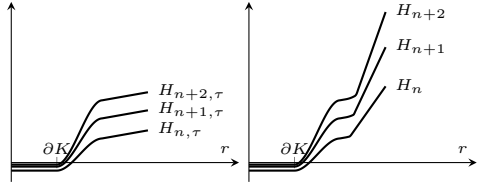
\begin{figure}
\begin{tikzpicture}[
	declare function={
     func(\x,\a,\b,\c,\d) = (\a*\x*\x*\x + \b*\x*\x + \c*\x + \d);},
     declare function={
     funcN(\x,\n) = func(\x,-\n,(3*\n)/2,0,-1/(2^(\n)));},
	declare function={
     midpoint(\n) = (1/2+sqrt(1-1/(12*\n))/2);},
     declare function={
     funcSlope(\x,\n) = (\x<=midpoint(\n))*funcN(\x,\n) + and(\x>midpoint(\n),x<=4)*(funcN(midpoint(\n),\n)+(\x-midpoint(\n))/4);},
     declare function={
	funcShift(\x,\n) = (\x<=1)*(-1/(2^(\n))) + and(\x>1,\x<=4)*(funcSlope(\x-1,\n);},
	declare function={
	funcCircle(\x) = (\x<.5)*(-sqrt(1/4-\x^2))+and(\x>=.5,\x<=4)*(0);},
	declare function={
	circleMid(\n) = sqrt(\n*\n/(64+4*\n*\n));}]
\begin{axis}[width=0.33\linewidth, axis lines=center, xtick={1}, xticklabels = {$\partial K$}, xlabel={$r$}, ytick={0},yticklabels={} xmin=0, xmax=5, ymin=-.6, ymax=5, label style={font=\tiny}, x tick label style={font=\tiny,above=2}]
\addplot[color=black, domain=0:3, samples=100, smooth, thick]{funcShift(x,2)} node[right ,pos=1] {\tiny $H_{n,\tau}$};
\addplot[color=black, domain=0:3, samples=100, smooth, thick]{funcShift(x,3)} node[right ,pos=1] {\tiny $H_{n+1,\tau}$};
\addplot[color=black, domain=0:3, samples=100, smooth, thick]{funcShift(x,4)} node[right ,pos=1] {\tiny $H_{n+2,\tau}$};
\end{axis}
\end{tikzpicture}
\begin{tikzpicture}[
	declare function={
     func(\x,\a,\b,\c,\d) = (\a*\x*\x*\x + \b*\x*\x + \c*\x + \d);},
     declare function={
     funcN(\x,\n) = func(\x,-\n,(3*\n)/2,0,-1/(2^(\n)));},
	declare function={
     midpoint(\n) = (1/2+sqrt(1-1/(12*\n))/2);},
     declare function={
     funcSlope(\x,\n) = (\x<=midpoint(\n))*funcN(\x,\n) + and(\x>midpoint(\n),x<=4)*(funcN(midpoint(\n),\n)+(\x-midpoint(\n))/4);},
     declare function={
	funcShift(\x,\n) = (\x<=1)*(-1/(2^(\n))) + and(\x>1,\x<=4)*(funcSlope(\x-1,\n);},
	declare function={
	funcCircle(\x) = (\x<.5)*(-sqrt(1/4-\x^2))+and(\x>=.5,\x<=4)*(0);},
	declare function={
	circleMid(\n) = sqrt(\n*\n/(64+4*\n*\n));},
	declare function={
	funcCircleSlope(\x,\n) = (\x<=circleMid(1))*((\x-circleMid(1))/4 + funcCircle(circleMid(1))) + 
	and(\x>circleMid(1),\x<=circleMid(\n))*funcCircle(\x) +
	and(\x>circleMid(\n),\x<=4)*((\x-circleMid(\n))*\n + funcCircle(circleMid(\n)));},
	declare function={
	funcInc(\x,\n) = (\x<2)*funcShift(\x,\n) +
	and(\x>=2,\x<=4)*(funcShift(2,\n)+.515+funcCircleSlope(\x-2,\n));}]

\begin{axis}[width=0.33\linewidth, axis lines=center, xtick={1}, xticklabels = {$\partial K$}, xlabel={$r$}, ytick={0},yticklabels={} xmin=0, xmax=5, ymin=-.6, ymax=5, label style={font=\tiny}, x tick label style={font=\tiny,above=2}]
\addplot[color=black, domain=0:3, samples=100, smooth, thick]{funcInc(x,2)} node[right ,pos=1] {\tiny $H_{n}$};
\addplot[color=black, domain=0:3, samples=100, smooth, thick]{funcInc(x,3)} node[right ,pos=1] {\tiny $H_{n+1}$};
\addplot[color=black, domain=0:3, samples=100, smooth, thick]{funcInc(x,4)} node[right ,pos=1] {\tiny $H_{n+2}$};
\end{axis}

\end{tikzpicture}
\caption{Depictions of the choice of functions $H_{n,\tau}$, and $H_n$ for \cref{prop:PropertiesOfSH} \cref{prop:InvarianceOfSlope}.}
\label{fig:InvarOfSlope}
\end{figure}

Define $H^{\sigma_n}$ via \cref{con:TheBump}.  Notice that $CF(H_{n})$ has two types of generators: The generators that correspond to the generators of $CF(H_{n,\tau})$ that have associated orbits lying in the region where $r \leq 1-\delta + s_1$, and the generators that have associated orbits lying in the region where $r \geq 1-\delta+s_1$.  By \cref{lem:TopologicalFiltration}, the first type of generators may be canonically identified with $CF(H_{n,\tau})$ and be realized as a subcomplex $i_n: CF(H_{n,\tau}) \hookrightarrow CF(H_n)$.  The map $i_n$ is not necessarily induced from a continuation.\footnote{As a result of this, our isomorphism will not be natural in the sense that it is induced from a continuation map.  Of course, we could consider continuations to construct this subcomplex; however, the strictly commutative diagram in this argument will become only homotopy commutative.  In light of this, we would need to consider the associated mapping cones (as opposed to quotients) and a more delicate homological analysis would be required (similar to, but different than the homological analysis given in the proof of \cref{lem:InvarienceOfFunction1}).}
In this manner, we obtain a strictly commutative diagram
\[ \xymatrix{ CF(H_{0,\tau}) \ar[r]^{c(\sigma_{\tau,0})} \ar[d]^{i_0} & CF(H_{1,\tau}) \ar[r]^{c(\sigma_{\tau,1})} \ar[d]^{i_1} & \cdots \\ CF(H_{0}) \ar[r]^{c(\sigma_0)} & CF(H_{1}) \ar[r]^{c(\sigma_1)} & \cdots.} \]
Indeed, when $r \leq 1- \delta+s_1$, the Floer data $H_{n,\tau}$ and $\sigma_{n,\tau}$ agrees identically with the Floer data $H_n$ and $\sigma_n$.  Moreover, by \cref{lem:TopologicalFiltration}, there is a bijection between their associated Floer trajectories that have ends lying in the region where $r \leq 1 - \delta +s_1$.  So with the above identifications of subcomplexes the diagram strictly commutes.

Consider the quotient
\[ C(i_n) = CF(H_n)/CF(H_{n,\tau}),\]
which is generated by the orbits that lie in the region where $r \geq 1-\delta+s_1$.  The continuation maps $c(\sigma_n)$ induce maps $\varphi_n: C(i_n) \to C(i_{n+1})$ as follows.  Let $x$ be a generator in $C(i_n)$, lift it to the corresponding generator $\widetilde{x}$ in $CF(H_n)$, define $\varphi_n(x)$ to be the composition of $c(\sigma_n)(\widetilde{x})$ with the quotient map from $CF(H_n) \to C(i_n)$.  We have a strictly commutative diagram of rays:
\[ \xymatrix{ 0 \ar[r] & CF(H_{\tau,0}) \ar[r]^{i_0} \ar[d]^{c(\sigma_{\tau,0})} & CF(H_0) \ar[r] \ar[d]^{c(\sigma_0)} & C(i_0) \ar[r] \ar[d]^{\varphi_0} & 0 \\ 0 \ar[r] & CF(H_{\tau,1}) \ar[r]^{i_1} \ar[d]^{c(\sigma_{\tau,1})} & CF(H_1) \ar[r] \ar[d]^{c(\sigma_1)} &C(i_1) \ar[d]^{\varphi_1} \ar[r] & 0. \\  &   &  &  &   \\ }\]
Notice that the complex $\Tel(\{C(i_n)\})$ is a filter colimit of finitely generated, free $\Lambda_{\geq 0}$-modules.  So by \cite[\href{https://stacks.math.columbia.edu/tag/058G}{Tag 058G}]{StacksProject}, it is a flat $\Lambda_{\geq 0}$-module.  So by \cite[\href{https://stacks.math.columbia.edu/tag/0315}{Tag 0315}]{StacksProject}, we have a short exact sequence of completed mapping telescopes:
\[ \xymatrix{ 0 \ar[r] & \widehat{\Tel}(\sC\sF(H_{n,\tau}\}) \ar[r] & \widehat{\Tel}(\sC\sF(H_n)) \ar[r] & \widehat{\Tel}(\{C(i_n)\}) \ar[r] & 0 \\ }.\]
The cohomology of the left complex is $\widehat{SH}(K \subset M; f, \tau)$.  The cohomology of the middle complex is $\widehat{SH}(K \subset M; f)$.  So to complete the proof, it suffices to show that $\widehat{\Tel}(\{C(i_n)\})$ is acyclic.  Using \cref{lem:TelToColimQI}, we prove the stronger claim that $\widehat{\colim}(\{C(i_n)\})$ is trivial.  We will show that for each generator $x \in C(i_n)$, the valuation of the sequence of elements $\varphi_{m+n} \circ \cdots \circ \varphi_n(x)$ diverges with $m$.  This implies that each element of the completed colimit is equivalent (as a Cauchy sequence) to the trivial element.  

The generators of $C(i_n)$ may be identified with a subset of the generators of $C(i_{n+1})$.  Namely, the generators of $C(i_{n+1})$ whose orbits lie in the region where $1 - \delta+s_1 \leq r \leq 1 - \delta+s_n$.  With this in mind, by \cref{lem:TopologicalFiltration}, we write
\[ \varphi_n(x) = \lambda_0 \cdot x + \sum_{i} \lambda_i \cdot y_i, \]
where $y_i$ is a generator (not equal to $x$) of $C(i_{n+1})$ whose orbits lie in the region where $1 - \delta+s_1 \leq r \leq 1 - \delta+s_n$.  Considering the constant Floer trajectory at $x$, which is regular, isolated, and has topological energy equal to $1$.  \cref{thm:PardonEssentials} and \cref{cor:EnergyRelation1} gives that $\val(\lambda_0) \geq 1$.  If $r(y_i) = r(x)$, then by \cref{prop:MaximumPrincipleForWeaklyConvexDomains} any Floer trajectory of type $(\sigma_n,y_i,x)$ must be contained in the $r$-slice $\{r = r(x)\}$.  Consequently, by \cref{cor:EnergyRelation1} and our description of $\sigma_n$ (which induces the map $\varphi_n$), we have that $\val(\lambda_i) \geq 1$.  If $r(y_i) \neq r(x)$, then $r(y_i) < r(x)$ by \cref{lem:TopologicalFiltration}.  However in this case, by induction, we may conclude that $\val(\varphi_{m+n} \circ \cdots \circ \varphi_{n+1}(y_i))$ diverges with $m$.  So we see that $\varphi_{m+n} \circ \cdots \circ \varphi_n(x)$ diverges with $m$, as desired.
\end{proof}


\section{Unirulings by disks}\label{sec:DiskUnirulings}

This section presents a Floer theoretic result that is needed to establish our main result.  It is divided into two parts. The first states the main result and gives its proof, assuming a technical lemma.  The second proves the technical lemma.

\subsection{Main result}
We consider a convex symplectic domain $(M,\Omega,\lambda)$ whose boundary has a non-degenerate Reeb vector field.  We also assume that the symplectic form $\Omega$ is integral\footnote{The symplectic form $\Omega$ is \emph{integral} if for every map $u: S^2 \to M$, the integral $\int u^*\Omega \in \IZ$.}.

\begin{prop}\label{prop:DiskUniruling}
Suppose that $\widehat{SH}(M_a \subset M; 0) \otimes \Lambda = 0$.  For each $\delta>0$ there exists a constant $E>0$ such that for each
\begin{itemize}
	\item $\Omega$-compatible almost complex structure $J$ on $M_{a-\delta}$, and
	\item $p \in \intrr(M_{a-\delta})$
\end{itemize}
there exists a non-empty open subset $\sI \subset (0,a-\delta]$ that depends on $J$ and $p$ such that for each $a' \in \sI$ there exists a (possibly disconnected) genus zero, compact, holomorphic curve $u: \Sigma \to M_{a'}$ with non-empty boundary that satisfies:
\begin{enumerate}
	\item $u(\partial \Sigma) \subset \partial M_{a'}$,
	\item $u$ is non-constant on each of its components,
	\item $u(\Sigma) \cap p \neq \varnothing$, and
	\item the energy of $u$ with respect to $\Omega$ and $J$ is bounded by $E$.
\end{enumerate}
\end{prop}

We will use the curves from \cref{prop:DiskUniruling} to produce unirulings and (multi)sections.  To prove the above result, we apply the following technical lemma, the convergence result from \cref{prop:ConvergenceToPearls}, and the removal of singularities theorem.

\begin{lem}\label{lem:FloerUniruling}
Under the assumptions of \cref{prop:DiskUniruling}, there exist constants $E>0$, and $\varepsilon < \delta$, and $p_0 \in \intrr(M_{a-\delta})$ and a sequence of $(a-\delta+\varepsilon)$-radially admissible Hamiltonians $H_\nu$ on $M$ with
\begin{itemize}
	\item $H_\nu = H_{\nu+1}$ for $a-\delta+\varepsilon \leq r(x)$ , and
	\item $H_\nu = g/\nu$ on $M_{a-\delta}$, where $g$ is a $C^2$-small Morse function that is Morse-Smale with respect to $J$,
\end{itemize}
and a sequence of Floer trajectories $u_\nu \in \overline{\sM}((H_\nu,J),x_-,x_+)$, where
\begin{itemize}
	\item $r(x_-)$ is an index zero critical point of $g$ at $p_0$,
	\item $r(x_+) > a-\delta+\varepsilon$, and
	\item $E(u_\nu) < E$.
\end{itemize}
Moreover, $E$ is independent of $J$.
\end{lem}

We prove \cref{lem:FloerUniruling} in the next subsection.  Assuming it, we prove \cref{prop:DiskUniruling}.

\begin{proof}
We immediately reduce to the case where $p = p_0$.  By \cite[Lemma 5.16]{McLean_GrowthRate}, there exists a Hamiltonian diffeomorphism $\varphi$ that is supported in $\intrr(M_{a-\delta})$ so that $\varphi(p_0) = p$.  Notice that the pull-back of $J$ along $\varphi$ is $\Omega$-compatible.  So in order to find the desired $J$-holomorphic curve through $p$, it suffices to find a $\varphi^*J$-holomorphic curve through $p_0$.  So we have reduced our work to proving the result for fixed $p_0 \in \intrr(M_{a-\delta})$.

Using the sequence of Floer trajectories from \cref{lem:FloerUniruling}, \cref{prop:ConvergenceToPearls} implies that for each $J$ there exists a Morse-Bott broken Floer trajectory of type $((H_\infty,J),x_-,x_+)$, say
\[(\gamma^0,u^1,\gamma^1,\dots,u^k, \gamma^k, u^{k+1},u^{k+2},\dots,u^m),\]
that satisfies:
 \begin{enumerate}
 	\item $\underline{ev}(\gamma^0) = x_-$,
 	\item $\lim_{s \to +\infty} u^m(s,t) = x_+$, and
 	\item $\sum_{i=1}^m E(u^i) \leq E$.
 \end{enumerate}
Since $x_- = p_0$ is an index zero critical point of the Morse function $g$, \cref{prop:ConvergenceToPearls} implies (without loss of generality) that $\gamma^0$ is constant and $u^1$ is non-constant.

Suppose that the image of $u^i$ is contained inside $\intrr(M_a)$ for all $i < \ell$.  It follows that $\lim_{s \to + \infty}u^{\ell}(s,t)$ lies in the region where $r \geq a-\delta + \varepsilon$.  By Sard's theorem, there exists a non-empty open subset $\sI \subset (0,a-\delta]$ that depends on $J$ such that for each $a' \in \sI$, $(u^\ell)^{-1}(M_{a'}) \cap C^\ell_0$ is a smooth Riemann surface with non-empty boundary.  There is an irreducible component of this Riemann surface with a single puncture whose image is (asymptotic to) a point on $\intrr(M_a)$.  Let $\Sigma'$ denote the surface composed of $C^1$ and the truncated curve $(u^\ell)^{-1}(M_{a'}) \cap C^\ell_0$.  Let $u': \Sigma' \to M_{a'}$ be the $J$-holomorphic (In this region, the Hamiltonian is zero and, thus, the maps are genuinely holomorphic.) curve given by $u' = \sqcup_{i=0, \ell} u^i$.  Notice that $u'$ has energy bounded by $E$.  So by the removal of singularities theorem, see \cite[Theorem 4.1.2]{McDuffSalamon_Jholo}, $u'$ may be extended to a holomorphic map over its punctures.  Write this extension as $u: \Sigma \to M_{a'}$.  

We wrap up.  First, the classification of surfaces with boundaries gives that $\Sigma$ is a genus zero curve with boundary.  Second, by construction, the image of boundary of $\Sigma$ under $u$ is contained in $\partial M_{a'}$.  Third, by construction, $u$ is non-constant on each of its components.  Fourth, since $\lim_{s \to -\infty} u^1(s,t) = p_0$, $u(\Sigma)$ meets the point $p_0$.  Fifth, the energy of $u$ with respect to $\Omega$ and $J$ is bounded by $E$ independently of $J$.  This completes the proof.  We mention that the image of $u$ does not need to be connected (however, it will be connected when the $\gamma^i$ are constant or $u^\ell = u^1$).
\end{proof}

\subsection{Proof of technical lemma}

To prove \cref{lem:FloerUniruling}, we consider the sequence of radially admissible Hamiltonians as in \cref{con:RadialCofinal} and explicitly exhibit a non-trivial differential in a Floer chain complex that computes $\widehat{SH}(M_{a-\delta} \subset M; f)$.  This non-trivial differential will give rise to the desired Floer trajectory.  

Put differently, the vanishing of action completed symplectic cohomology gives an isomorphism $H(M;\Lambda) \cong \widehat{SH}_+(M_{a-\delta} \subset M;f)$.  Under this isomorphism, a (completed) element of $\widehat{SH}_+(M_{a-\delta} \subset M;f)$ is mapped onto the point class in $H(M;\Lambda)$.  We show that the count of Floer trajectories that defines this mapping is non-zero and deduce that there exists a Floer trajectory connecting a Reeb orbit to the point class of the interior Morse function.  Of course, we need to preserve these (virtual) counts of Floer trajectories for varying $J$ and $g$ and show that the resulting Floer trajectories have energies bounded independent of these choices.  The main technical difficulty is that the point class will only be hit by a completed element, that is, we will only hit the point class plus higher order terms in the Novikov parameter.  When deforming $J$ or $g$, we need to ensure that these higher order terms do not pair off with the point class, leaving us with no differential hitting the point class.  This requires a chain level homological analysis of Floer complexes.  We use the integrality of the symplectic form to carry out this analysis.	

To begin, by rescaling $\Omega$, we may assume that for every map $u: S^2 \to M$, $\int u^*\Omega \in 5 \cdot \IZ$.  We point out two simple consequences of this integrality and \cref{lem:EnergyRelation}.

\begin{lem}\label{lem:LowEnergyFloer}
Consider a monotonically admissible family of Hamiltonians $H^\sigma$.  Let $u$ be a Floer trajectory of type $(\sigma,x_0,x_\ell)$ such that $x_0$, and $x_\ell$ are constant.  If
\[ 0 \leq H^{\sigma}_{\underline{e_0}}(x_0) - H^\sigma_{\underline{e_\ell}}(x_\ell) < 5,\]
then
\[ E_{top}(u) \geq  H^{\sigma}_{\underline{e_0}}(x_0) - H^\sigma_{\underline{e_\ell}}(x_\ell).\]
If 
\[ -5 \leq H^{\sigma}_{\underline{e_0}}(x_0) - H^\sigma_{\underline{e_\ell}}(x_\ell) < 0,\]
then
\[ E_{top}(u) \geq  5+ H^{\sigma}_{\underline{e_0}}(x_0) - H^\sigma_{\underline{e_\ell}}(x_\ell).\]\qed
\end{lem}

\begin{notn}
Let $J_0$ be an admissible almost complex structure on $M$.  Let $g_0: M_{a-\delta} \to \IR$ be a negative $C^2$-small Morse function as in \cref{con:RadialCofinal}.  We assume that $J_0$ and $g_0$ further satisfy:
\begin{itemize}
	\item $g_0$ is Morse-Smale with respect to the metric $\Omega(\cdot, J_0 \cdot)$,
	\item $p_0$ is the unique index zero critical point of $g_0$,
	\item $g_0(p_0)/2 - g_0(x) \leq 0$ for all critical points $x \neq p_0$, and
	\item $|g_0| \leq 1$.
\end{itemize}
Consider the cofinal sequence of Hamiltonians $H_n$ and connecting families $\sigma_n$ as in \cref{con:RadialCofinal}, where the Morse function part is $g_0$.
\end{notn}

\begin{lem}\label{lem:InitialBE}
There exists $\lambda' \in \Lambda_{\geq 0}$ and $n \in \IN$ such that
\[ \val(\lambda' \cdot c(\sigma_{n-1}) \circ \cdots \circ c(\sigma_0)(p_0) - d(y)) > \val(\lambda') + 3\]
for some $y \in CF(H_n,J_0)$, where $p_0$ denotes the critical point of $g_0$ at $p_0$.
\end{lem}

\begin{proof}
By \cref{lem:TopologicalFiltration} and \cite[Theorem 10.7.1]{Pardon_VFC}, $p_0$ is closed in each $CF(H_i,J_0)$.  In particular, $p_0$ is closed in $CF(H_0,J_0)$.  By assumption (and \cref{prop:PropertiesOfSH} \cref{prop:InvarienceOfFunction}),
\[ 0 = \widehat{SH}(M_{a-\delta} \subset M; 0) \otimes \Lambda \cong \widehat{SH}(M_{a-\delta} \subset M ; f) \otimes \Lambda. \]
By \cref{lem:TelKillingLemma}, there exists $\lambda' \in \Lambda_{\geq 0}$ and $y \in CF(H_n,J_0)$ (for some $n$) such that
\[ \val(\lambda' \cdot c(\sigma_{n-1}) \circ \cdots \circ c(\sigma_0)(p_0) - d(y)) > \val(\lambda') + 3,\]
as desired.
\end{proof}

To understand \cref{lem:InitialBE}, consider elements of $\overline{\sM}(\sigma_i,x,y)$.  By \cref{lem:TopologicalFiltration}, all the elements lie in $\intrr(M_{a-\delta})$ when $y$ is a critical point of $g_0$.  Notice that the associated family of Hamiltonians on $\intrr(M_{a-\delta})$ is given by
\[H^{\sigma_i} = g_0 - \ell(s) \cdot \frac{1}{i+1} - (1-\ell(s)) \cdot \frac{1}{i}.\]
By our choice of $g_0$ and \cref{lem:LowEnergyFloer}, all elements of $\overline{\sM}(\sigma_i,x,y)$, where $x$ and $y$ are critical points of $g_0$, have energy greater than or equal to 
\[\begin{cases} g_0(x) - g_0(y) + 1/i-1/(i+1), & g_0(x) - g_0(y) + 1/i-1/(i+1) \geq 0 \\ 3, & \mbox{else.} \\ \end{cases} \]
So by considering the constant Floer trajectory in $\overline{\sM}(\sigma_i,p_0,p_0)$ (as we have done before), \cref{cor:EnergyRelation1} implies that
\[ c(\sigma_i)(p_0) = T^{\frac{1}{i} - \frac{1}{i+1}} \cdot ( \eta_i \cdot p_0 + C_i ),\]
where $\val(\eta_i) = 0$, and $C_i$ is a $\Lambda_{\geq 0}$-linear combination of critical points of $g_0$ (not including $p_0$) such that the valuation of $C_i$ along any critical point $x$ is at least $g_0(x) - g_0(p_0)$.  Define $\lambda = \lambda' \cdot T^{1-1/n} \cdot \prod_i \eta_i$.  Inductively (considering the constant Floer trajectories in the $\overline{\sM}(\sigma_i,x,x)$), we have that
\[d(y) = \lambda \cdot p_0 + C + D,\]
where $C$ is a $\Lambda_{\geq 0}$-linear combination of critical points of $g_0$ (not including $p_0$) such that the valuation of $C$ along any critical point $x$ is at least $g_0(x) - g_0(p_0) + \val(\lambda)$, and $D$ is a $\Lambda_{\geq 0}$-linear combination of orbits with $\val(D) > \val(\lambda)+2$.  In particular, we have shown that
\[d(y) = \lambda \cdot p_0 + \mbox{higher order terms in }T.\]

Consider $H_n$ from \cref{lem:InitialBE}.  We will modify it and the almost complex structure $J_0$ to a sequence of Hamiltonians $H_\nu$ and the desired almost complex structure $J$ that satisfy the conditions of the lemma.  We complete the proof of \cref{lem:FloerUniruling}.  

\begin{proof}
For $\nu \in \IN$, fix $(a-\delta+\varepsilon)$-radially admissible Hamiltonians $G_\nu$ that satisfy:
\begin{itemize}
	\item $G_\nu \leq G_{\nu+1}$
	\item for $r \leq a-\delta$, $G_\nu = g_0/2^{\nu}$, and
	\item for $r \leq a - \delta+\varepsilon$, $G_\nu = H_n+\delta+1/n$.
\end{itemize}
Using \cref{con:TheBump}, we have continuations $G^{\sigma_\nu}$ with $G^{\sigma_\nu}_{\underline{e_0}} = G_{\nu+1}$ and $G^{\sigma_\nu}_{\underline{e_1}} = G_{\nu}$.  We consider how $d(y) = \lambda \cdot p_0 + C+D$ transforms under these continuations.

To begin, the continuations $\sigma_\nu$ are monotone.  So 
\[\val( c(\sigma_{\nu}) \circ \cdots \circ c(\sigma_0)(D)) > \val(\lambda)+2.\]
Suppose that $x$ and $y$ are critical points of $g_0$.  When $g_0(x)/2 - g_0(y) \geq 0$, \cref{lem:LowEnergyFloer} implies that all elements of $\overline{\sM}(\sigma_\nu,x,y)$ have energy greater than or equal to $g_0(x)/2^{\nu+1} - g_0(y)/2^\nu$.  When $g_0(x)/2 - g_0(y) < 0$, \cref{lem:LowEnergyFloer} implies that all elements of $\overline{\sM}(\sigma_\nu,x,y)$ have energy greater than or equal to
\[ 5 + g_0(x)/2^{\nu+1} - g_0(y)/2^\nu \geq 3.\]
In particular, since $g_0(p_0)/2 - g_0(x) < 0$ for all critical points $x \neq p_0$, we (inductively) have that the valuation of 
\[ c(\sigma_{\nu}) \circ \cdots \circ c(\sigma_0)(C) \]
along $p_0$ is at least $\val(\lambda)+3$.  Also, we can (inductively) conclude that the valuation of
\[c(\sigma_{\nu}) \circ \cdots \circ c(\sigma_0)(C)\]
along any critical point $x \neq p_0$ is at least
\[ -g_0(p_0) + g_0(x)/2^{\nu+1} > -g_0(p_0) + g_0(p_0)/2^{\nu+1}\]
The only element of $\overline{\sM}(\sigma_\nu,p_0,p_0)$ with energy equal to $-g_0(p_0)(1/2^{\nu}-1/2^{\nu+1})$ is the constant Floer trajectory.  So inductively, combining the above observations, we have that
\[ c(\sigma_{\nu}) \circ \cdots \circ c(\sigma_0)(d(y)) = \lambda_\nu \cdot p_0  + \mbox{higher order terms in }T,\]
where $\lambda_\nu \in \Lambda_{\geq 0}$ has valuation given by $\val(\lambda) -g_0(p_0) + g_0(p_0)/2^\nu < \val(\lambda) +2$.
	
Fix a negative $C^2$-small Morse function $g: M_{a-\delta} \to \IR$ that satisfies:
\begin{itemize}
	\item $g$ is Morse-Smale with respect to $\Omega(\cdot,J\cdot)$,
	\item $g$ has a unique index zero critical point at $p_0 \in \intrr(M_{a-\delta})$, and
	\item $g \geq g_0$ with $g(p_0) = g_0(p_0)$.
\end{itemize}
Fix $(a-\delta+\varepsilon)$-radially admissible Hamiltonians $H_\nu$ that satisfy:
\begin{itemize}
	\item $H_\nu \leq H_{\nu+1}$
	\item for $r \leq a-\delta$, $H_\nu = g/2^\nu$, and
	\item for $r \leq a - \delta+\varepsilon$, $H_\nu = H_n+\delta+1/n$.
\end{itemize}
Let $J_1$ denote an admissible almost complex structure on $M$ that agrees with $J_0$ for $r \geq a-\delta+\varepsilon$ and agrees with $J$ on $M_{a-\delta}$.  Define simplices $\tau_\nu \in  \sJ\sH_1^+(M,\Omega,\lambda)$ connecting $G_\nu$ and $H_\nu$ via \cref{con:TheBump}, and by fixing a family of almost complex structures that connects $J_0$ to $J_1$ that is the identity for $r \geq a - \delta + \varepsilon$.  This gives continuation maps
\[ c(\tau_\nu): CF(G_{\nu},J_0) \to CF(H_\nu,J_1). \]
By considering the constant Floer trajectories in $\overline{\sM}(\sigma_\nu, p_0,p_0)$, which have zero energy, and the monotonicity of the $\tau_\nu$, the above work for the $\sigma_\nu$ gives that
\[ c(\tau_\nu) \circ c(\sigma_{\nu}) \circ \cdots \circ c(\sigma_0)(d(y)) = \lambda_\nu \cdot p_0  + \mbox{higher order terms in }T.\]

We wrap up.  First, the sequence of Hamiltonians $H_\nu$ defines a sequence of Hamiltonians as stated in the lemma modulo interchanging $\nu$ and $2^\nu$.  Set $E = \val(\lambda)+2$.  This bound is independent of the choice of almost complex structure $J$.  By \cref{thm:PardonEssentials}, \cref{lem:EnergyRelation}, and our explicit description of the differential for Floer chain complexes, there exists an element $u_\nu \in \overline{\sM}((H_\nu,J_1),p_0,x_+)$ with $E_{geo}(u_\nu) \leq E_{top}(u_\nu) \leq E$.  Finally, by \cref{prop:PropertiesOfSH} \cref{prop:LES} and \cite[Theorem 10.7.1]{Pardon_VFC}, $x_+$ corresponds to a Reeb orbit and thus satisfies $r(x_+) \geq a - \delta + \varepsilon$.
\end{proof}

\begin{rem}
In lieu of our above approach, one could observe that, since transversality for the associated moduli spaces of Morse flow lines can be achieved, the work of Pardon prior to \cite[Theorem 10.7.1]{Pardon_VFC} gives that (after fixing an appropriate choice of additional data) the only Floer trajectories needed to compute the above continuation maps are, in fact, the Morse flow lines.  This simplifies the homological analysis.  However, we have opted for our more elementary approach.  We also point out that if $c_1(M) = 0$, then using the resulting $\IZ$-grading on the Floer chain complexes greatly simplifies the above analysis.
\end{rem}

\subsection{An addendum on the product case}\label{subsec:ProductUniruling}

Consider $M = \IC \times F$, where $(F,\omega_F)$ is a closed symplectic manifold with symplectic form $\omega_\IC + \omega_F$, where $\omega_\IC$ is the standard symplectic form on $\IC$.  The associated stable Hamiltonian structure has a degenerate Reeb vector field.  However, with the radial coordinate, $|\cdot|^2$, on $\IC$, the orbits arise as $S^1 \times F$-families of orbits.  So as in \cref{rem:MorseBottBreakings}, we may appropriately define Hamiltonian Floer theory for these spaces by using radial Hamiltonians and introducing small perturbations near the families of orbits, which we now suppress.

Each $\ID_a \times F$ is displaceable in $\ID_b \times F$ for some $b \gg a$.  So \cref{prop:PropertiesOfSH} \cref{prop:VanishingForStDisp} implies that \cref{prop:DiskUniruling} gives a uniruling via disks.  However, we can conclude this result with a stronger conclusion.

\begin{prop}\label{prop:ProductUniruling}
In the above setup, \cref{lem:FloerUniruling} holds with the additional conclusion that $x_+$ corresponds to a Reeb orbit whose projection down to $\IC$ winds once around the origin.  
\end{prop}

So the holomorphic curves in \cref{prop:DiskUniruling} all have an irreducible component (the one that contains the boundary) whose degree is one when projected to $\IC$ (since the Gromov-Floer compactness used in the proof of \cref{prop:ConvergenceToPearls} preserves this winding).

\begin{proof}
To establish this winding condition, one needs to understand the bounding cochain $y$ in the proof of \cref{lem:FloerUniruling}.  Let $J_0 = j \oplus J_F$ be the product almost complex structure (which is admissible), where $j$ is the standard complex structure on $\IC$ and $J_F$ is an $\Omega$-compatible almost complex structure on $F$.  Fix a $C^2$-small Morse function $f$ on $F$ that is Morse-Smale with respect to the metric $\omega_F(\cdot, J_F \cdot)$ and let $H = c \cdot \exp(|\cdot|^2) + f$, where $c$ is some small positive constant.  The associated Hamiltonian vector field is given by the sum of the $\omega_F$-dual of $-df$ and the $\omega_\IC$-dual of $-d (c \cdot \exp(|\cdot|^2))$.  So the Floer data decouples, and we obtain a Kunneth decomposition
\[ CF(H,J_0) = CF(f,J_F) \otimes CF(c \cdot \exp(|\cdot|^2), j).\]
Here, the left-hand-side computes the symplectic cohomology of $M$ as in \cite{Seidel_ABiasedViewOfSymplecticCohomology}.  But now an explicit analysis of the complex $CF(c \cdot \exp(|\cdot|^2), j)$ shows that it is acyclic and the orbit at $0$ is the boundary of an orbit that corresponds to a Reeb orbit (in $\IC$) that winds once around the origin.  So via Kunneth, the point class of $f$ is the boundary of an orbit $y$ that corresponds to a Reeb orbit (in $M$) that winds once around the origin.  Using this data, one runs \cref{lem:FloerUniruling} and studies the bounding cochain $y$.  One can arrange to ``turn off'' $H$ away from the Reeb orbits of the Hamiltonian $H$, and \cref{lem:TopologicalFiltration} ensures that the continuation of $y$ under our deformations never has image composed of orbits with winding numbers greater than $1$.  In this manner, we get the additional conclusion.
\end{proof}


\section{Proof of main theorem}\label{sec:ProofOfTheorem}

We do not delay giving the proofs of our main results any further.  We will pull results from later sections; however, these results can easily be black boxed for the moment.

\begin{notn}
\begin{enumerate}
	\item Let $\pi: \overline{M} \to \IP^1$ be a proper morphism of smooth projective varieties that is smooth over $\IP^1 \smallsetminus \{0,\infty\}$.
	\item Let $M \coloneqq \overline{M} \smallsetminus \pi^{-1}(\infty)$ and also write $\pi: M \to \IC$ for the restriction of $\pi$ to $M$.
	\item Let $p$ be a point in $M$.
	\item Let $M_a \coloneqq \pi^{-1}(\ID_a)$, where $\ID_a$ is the disc of radius $a$ in $\IC$.
\end{enumerate}
\end{notn}

To orient the reader: We prove \cref{thm:ActualMainTheorem} and \cref{thm:ActualMainTheoremCY} by using \cref{prop:DiskUniruling} to produce holomorphic curves with boundaries in $M$ and then apply \cref{lem:FishCompactness} to turn these holomorphic curves with boundaries into closed holomorphic curves in $\overline{M}$ that have the desired properties.  We divide this into a sequence of lemmas, which we discuss below; however, we defer the proofs until the later subsections.

We first define a space in which we will apply \cref{lem:FishCompactness}.  We will use the degeneration of $\overline{M}$ to the normal cone of $\pi^{-1}(\infty)$ (or rather a resolution of this possibly singular space.).  This will be the $P$ given in the lemma below. 

\begin{lem}\label{lem:DegenerationSpace}
There exists a smooth, quasi-projective variety $P$, equipped with a proper, surjective, holomorphic map $\pi_{P}: P \to \IC$ which satisfies:
\begin{enumerate}
	\item $\pi_{P}^{-1}(\IC^\times)$ is isomorphic to $\IC^\times \times \overline{M}$ with $P_z \coloneqq \pi_{P}^{-1}(z) \cong \overline{M}$ for $z \neq 0$, and
	\item $\pi_{P}^{-1}(0) = F \cup E$, where $E$ is a (possibly singular) subscheme and $F$ is a (possibly singular) proper subscheme that is birational to $\overline{M}$ via a morphism that is an isomorphism over $F \smallsetminus (E \cap F) \to M$,
\end{enumerate}
and a map $h: P \to \IP^1$, which satisfies:
\begin{enumerate}
	\item $F \subset h^{-1}(0)$, and
	\item $h$ restricted to $\pi_P^{-1}(\IC^\times) \cong \IC^\times \times \overline{M}$ is identified with the composition
	\[ \xymatrix{\IC^\times \times \overline{M} \ar[r]^{pr_{\overline{M}}} & \overline{M} \ar[r]^{\pi} & \IP^1}.\]
\end{enumerate}
We denote the restriction of $h$ to $P_z$ by $h_z: P_z \to \IP^1$. 
\end{lem}

For each $k>0$ and any K\"{a}hler form $\Omega_P$ on $P$, $h^{-1}(\ID_k) \subset P$ satisfies the setup \cref{lem:FishCompactness}.  So if there exists a sequence of holomorphic curves with boundaries in the fibres $h_{1/\nu}^{-1}(\ID_k)$ that satisfy the conditions of \cref{lem:FishCompactness}, then this compactness will give closed holomorphic curves in $\overline{M}$.  The precise sequences of holomorphic curves that we construct are below.

\begin{lem}\label{lem:SequenceOfCurves}
Given a K\"{a}hler form $\Omega_P$ on $P$, there exists a constant $k>0$ and sequences of (possibly disconnected) genus zero, compact holomorphic curves $u_\nu: \Sigma_\nu \to Q$ with boundaries, where $Q \coloneqq h^{-1}(\ID_k)$.  These curves satisfy:
\begin{enumerate}
	\item $\pi_P \circ u_\nu \equiv z_\nu \in \IC$ with $z_\nu \to 0$,
	\item $u_\nu(x_\nu) \to p \in F \smallsetminus E \cap F$ for some $x_\nu \in \Sigma_\nu$,
	\item $u_\nu$ is non-constant on each of its components,
	\item $u_\nu(\partial \Sigma_\nu) \subset \partial Q$, and
	\item the energies of the $u_\nu$ are uniformly bounded in $\nu$ independent of $p$.
\end{enumerate}
Moreover, if $\pi$ is smooth over $0$, then a component of $u_\nu$ with boundary projects to $\ID_k$ via a degree one map.
\end{lem}

To construct these curves, we construct a symplectic form $\Omega$ on $M$ for which the end $M \smallsetminus M_a$ is modeled after a symplectic mapping cylinder for some $a>0$.  This gives the subsets $M_c$ for $c>a$ the structures of convex symplectic domains, see \cref{subsec:HamFibs_MappingTori}.  Given $\Omega$, we may define action complete symplectic cohomology for $(M_c,\Omega)$.  Under the additional assumption of either $\pi$ being smooth over $0$ or $M$ having vanishing first Chern class,
\[\widehat{SH}(M_b \subset M_c;0) \otimes \Lambda = 0,\]
where $c>b>a$.  We also construct a symplectic embedding of each fibre $(h_z^{-1}(\IC), \Omega_P|_{h_z^{-1}(\IC)})$ into $(M,\Omega)$ with image in $M_a$.  Via these embeddings, we push forward the almost complex structures of the fibres $h_z^{-1}(\IC)$ to $\Omega$-compatible almost complex structures on $M_a$.  Applying \cref{prop:DiskUniruling}, we produce holomorphic curves for these pushed forward almost complex structures.  Finally, to produce the afore mentioned curves in the fibres $h_z^{-1}(\IC)$, we pull-back the curves in $M_a$ along the embeddings and appropriately ``trim'' their domains using Sard's theorem to lie in some appropriate $h^{-1}(\ID_k)$.  The explicit description of the symplectic form $\Omega$ and the symplectic embeddings, which the reader should think of as some type of modified parallel transport maps, is given in the following two lemmas.  They correspond to \cref{thm:ActualMainTheorem} and \cref{thm:ActualMainTheoremCY} respectively.  We prove \cref{lem:ModifiedParallelTransport} in a later subsection of this section; however, we defer the proof of \cref{lem:ModifiedParallelTransportCY} to a later part of this exposition as it is quite a bit more involved.

\begin{lem}\label{lem:ModifiedParallelTransport}
Suppose that $\pi$ is smooth over $0$ with fibre $F$.  There exists a symplectic form $\Omega = \omega_{\IC} + \omega_F$ on $M \cong \IC \times F$, where $\omega_\IC$ is the standard symplectic form on $\IC$ and $\omega_F$ is a symplectic form on $F$, constants $c > b > a > 0$, and symplectic embeddings
\[\psi_z: (h_z^{-1}(\IC), \Omega_P|_{P_z}) \hookrightarrow (M, \Omega)\]
for each $z \in \IC^\times$ that satisfy
\begin{enumerate}
	\item $\Omega$ is integral,
	\item $\widehat{SH}(M_b \subset M_c; 0) \otimes \Lambda = 0$,
	\item $\psi_z(h_z^{-1}(\IC)) \subset M_a$, and
	\item $\psi_z(h_z^{-1}(0)) = \pi^{-1}(0) \cong 0 \times F$.
\end{enumerate}
\end{lem}

\begin{lem}\label{lem:ModifiedParallelTransportCY}
Suppose that $M$ has vanishing first Chern class.  There exists a symplectic form $\Omega$ on $M$, constants $c > b > a > 0$, and symplectic embeddings
\[\psi_z: (h_z^{-1}(\IC), \Omega_P|_{P_z}) \hookrightarrow (M, \Omega)\]
for each $z \in \IC^\times$ that satisfy
\begin{enumerate}
	\item $\Omega$ is integral,
	\item $(M \smallsetminus M_a, \Omega)$ is symplectomorphic to a symplectic mapping cylinder, whose associated Reeb vector field is non-degenerate,
	\item $\widehat{SH}(M_b \subset M_c; 0) \otimes \Lambda = 0$, and
	\item $\psi_z(h_z^{-1}(\IC)) \subset M_a$.
\end{enumerate}
\end{lem}

Assuming the above lemmas, we give the proofs of \cref{thm:ActualMainTheorem} and \cref{thm:ActualMainTheoremCY}.  They are analogous.  So we give them simultaneously.

\begin{proof}
Let $k$, $Q$, and $u_\nu$ be as in \cref{lem:SequenceOfCurves}.  We first discuss unirulings, and then discuss (multi)sections.

Each $u_\nu$ has a component with the point constrain at $x_\nu$ with $u_\nu(x_\nu) \to p$.  If for $\nu \gg 0$, these components are genus zero curves with boundaries in $\partial Q$, then by \cref{lem:FishCompactness}, there exists a non-constant rational curve $u: \Sigma \to \overline{M}$ with $u(x) = p$ for some $x \in \Sigma$.  If for $\nu \gg 0$, these components have empty boundaries, then applying the classical Gromov compactness to the $u_\nu$ gives a non-constant rational curve $u: \Sigma \to \overline{M}$ with $u(x) = p$ for some $x \in \Sigma$.  To obtain non-constant rational curves for points $p \in \pi^{-1}(\infty) = \overline{M} \smallsetminus M$, one observes that the curves produced above all have uniformly bounded energies.  Consequently, by the classical Gromov compactness theorem, we can obtain rational curves through all $p \in \overline{M}$.  This produces a uniruling.

Each $u_\nu$ has a component with boundary.  After fixing a marked point for $u_\nu$ on this component that contains the boundary that passes through $\pi^{-1}(0)$ (which occurs for winding number considerations), \cref{lem:DegenerationSpace} and \cref{lem:SequenceOfCurves} imply that $Q$ and $u_\nu$ satisfy the hypotheses of \cref{lem:FishCompactness}.  So by \cref{lem:FishCompactness}, there exists genus zero, compact, holomorphic curve $u: \Sigma \to \overline{M}$ with empty boundary such that the image of $u$ is connected and intersects both $\pi^{-1}(0)$ and $\pi^{-1}(\infty)$.  In particular, there exists an irreducible component of the domain of $u$ such that $\pi \circ u$ restricted to this irreducible component is non-constant.  This produces a (multi)section.  

Finally, when $\pi$ is smooth over $0$ as in \cref{thm:ActualMainTheorem}, we need to show that $u$ restricted to an irreducible component is an actual section.  It suffices to show that $u$ restricted to some irreducible component has degree one.  This follows from the additional conclusion of \cref{lem:SequenceOfCurves} and by noting that the degree is preserved in the limiting curve $u$.
\end{proof}

The remainder of this section is divided into three additional subsections, each of which contains a proof of one of the above three lemmas.

\subsection{Defining the degeneration space}

We prove \cref{lem:DegenerationSpace}.  But first, we discuss the degeneration of $\overline{M}$ to the normal cone of $\pi^{-1}(\infty) \subset \overline{M}$, see \cite[Chapter 5]{Fulton_IntersectionTheory}.

\begin{con}\label{con:DeformationToNormalCone}
Define
\[Z \coloneqq \{0\} \times \pi^{-1}(\infty) \subset \IC \times \overline{M}.\]
Define the blowup
\[\beta: P' \coloneqq Bl_{Z} ( \IC \times \overline{M}) \to \IC \times \overline{M}.\]
$P'$ is a (possibly singular) quasi-projective variety.  The map $\beta$ is referred to as the degeneration of $\overline{M}$ to the normal cone $C$ of $\pi^{-1}(\infty) \subset \overline{M}$.  It is an isomorphism over $(\IC \times \overline{M}) \smallsetminus Z$.  Post-composing $\beta$ with the projection to the first factor $\IC \times \overline{M} \to \IC$ gives a surjective, holomorphic map $\pi_{P'}: P' \to \IC$ such that $\pi_{P'}^{-1}(\IC^\times) \cong \IC^\times \times \overline{M}$.  In particular, $\pi^{-1}_{P'}(z) \cong \overline{M}$ for $z \in \IC^\times$.  Over $0$, $\pi_{P'}^{-1}(0) = F' \cup E'$ and the following holds.
\begin{enumerate}
	\item $E'$ is the exceptional divisor of the blowup.  It is isomorphic to the projective completion of the normal cone of $\pi^{-1}(\infty)$ in $\overline{M}$, that is, $E' \cong \IP(C\oplus \IC)$.
	\item $F'$ is isomorphic to the blow up of $\overline{M}$ along $\pi^{-1}(\infty)$.  So $F'$ is birational to $\overline{M}$ via $\beta|_{F'}: F' \to \overline{M}$, which is an isomorphism over $F' \smallsetminus (E' \cap F') \to M$.
	\item $E' \cap F'$ is given by the projectivization of the normal cone of $\pi^{-1}(\infty)$ in $\overline{M}$, that is, $E' \cap F' \cong \IP(C)$.  The inclusion $E' \cap F' \hookrightarrow E'$ is the inclusion of the hyperplane at infinity in $\IP(C \oplus \IC)$.  The inclusion $E' \cap F' \hookrightarrow F'$ is the inclusion of the exceptional divisor of the blowup $\beta|_{F'}: F' \to \overline{M}$.
\end{enumerate}
\end{con}

\begin{con}
Apply \cref{con:DeformationToNormalCone} to the identity $\IP^1 \to \IP^1$.  Define the blowup
\[ \beta_0: {B} \coloneqq Bl_{(0,\infty)}(\IC \times \IP^1) \to \IC \times \IP^1 \]
and the projection $\pi_{{B}}: {B} \to \IC$ as in \cref{con:DeformationToNormalCone}.

Write $\pi_B^{-1}(0) = F_0 \cup E_0$.  $F_0$ is a $(-1)$-curve.  So we may blow down $F_0$ to obtain a smooth, quasi-projective variety, which in this case is isomorphic to a Hirzeburch surface with the fibre over infinity removed.  This space is isomorphic $\IC \times \IP^1$.  Denote this blow-down map by $\beta_1: B \to \IC \times \IP^1$.
\end{con}

We now prove \cref{lem:DegenerationSpace}.

\begin{proof}
By \cite[II.7.15]{Hartshorne_AG}, we have a morphism $\pi_{Bl}: {P'} \to {B}$ and a commutative diagram
\[ \xymatrix{ \IC^\times \times \overline{M} \ar@{^{(}->}[r]  \ar[d]^{\Ione \times \pi} & {P'} \ar[r]^\beta \ar[d]^{\pi_{Bl}} & \IC \times \overline{M} \ar[d]^{\Ione \times \pi} \\ \IC^\times \times \IP^1 \ar@{^{(}->}[r]&  {B} \ar[r]^{\beta_0} & \IC \times \IP^1. \\ } \]
Since $\pi$ is proper, $\pi_{Bl}$ is proper.

${P'}$ is not necessarily smooth (smoothness could fail in $E'$).  By Hironaka's resolution of singularities \cite{Hironaka_ResolutionOfSingularities}, there exists a smooth, quasi-projective variety $P$ (which is the $P$ in the statement of the lemma) and a birational morphism $S: P \to P'$ that is an isomorphism over ${P'} \smallsetminus E'$.  This gives a commutative diagram
\[ \xymatrix{ \IC^\times \times \overline{M} \ar@{^{(}->}[r]  \ar[d]^{\Ione \times \pi} & {P} \ar[r]^{\beta \circ S} \ar[d]^{\pi_{Bl} \circ S} & \IC \times \overline{M} \ar[d]^{\Ione \times \pi} \\ \IC^\times \times \IP^1 \ar@{^{(}->}[r]&  {B} \ar[r]^{\beta_0} & \IC \times \IP^1. \\ } \]
Define:
\begin{enumerate}
	\item $\pi_{P} \coloneqq \pi_{B} \circ \pi_{Bl} \circ S = \pi_{{P'}} \circ S$,
	\item $h \coloneqq pr_{\IP^1} \circ \beta_1 \circ \pi_{Bl} \circ S$,
	\item $F \coloneqq \overline{S^{-1}(F' \smallsetminus E')}$, and
	\item $E \coloneqq S^{-1}(E')$.
\end{enumerate}
By \cite[II.7.16]{Hartshorne_AG}, $S|_{F}: F \to F'$ is a blowup of $F'$ that is supported away from $(F' \smallsetminus E' \cap F') \cong M$.  By \cref{con:DeformationToNormalCone}, $\beta|_{F'} \circ S|_{F}: F \to \overline{M}$ is a birational map that is an isomorphism over $F \smallsetminus (F \cap E) \to M$, as desired.  The remaining claims of the lemma are easily verified.
\end{proof}

\begin{figure}[h]
\centering
\includegraphics[width=.8\linewidth]{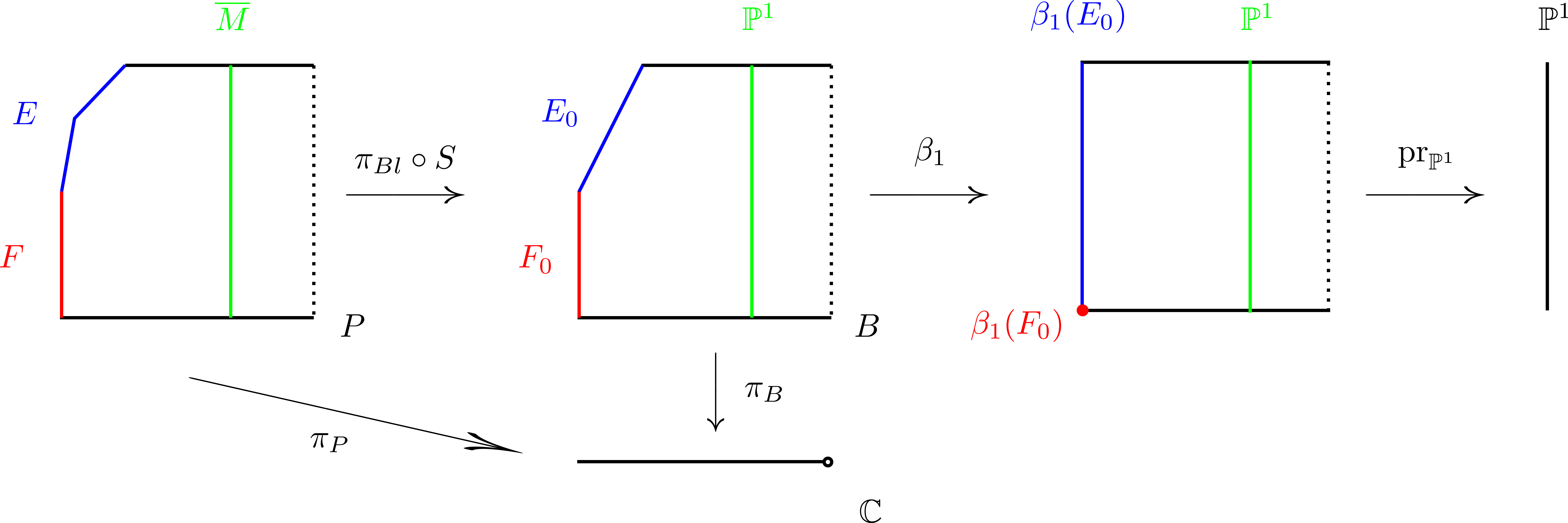}
\label{fig:DegenerationToNormalWithh}
\caption{Depictions of the spaces and maps from \cref{lem:DegenerationSpace}.}
\end{figure}

\subsection{Finding symplectic embeddings}

We prove \cref{lem:ModifiedParallelTransport}, using work in \cref{sec:ComplementEmbeddings} and \cref{prop:FlatteningOverC}.

\begin{proof}
Let $\Omega_P$ be any integral K\"{a}hler form on $P$ (which exists since $P$ is projective).  Let $\Omega_z \coloneqq \Omega_P|_{P_z}$.  For each $z \in \IC^\times$, we have a symplectic parallel transport map from $(P_z, \Omega_z)$ to $(P_1,\Omega_1)$.  By \cite[Lemma 5.16]{McLean_GrowthRate}, we may pre-compose this parallel transport map with a symplectomorphism such that the composition, $\varphi_z: (P_z, \Omega_z) \to (P_1,\Omega_1)$, maps $h_z^{-1}(\infty)$ to $h_1^{-1}(\infty)$ and $h_z^{-1}(0)$ to $h_1^{-1}(0)$.  We identify $(P_1,\Omega_1)$ and $h_1$ with $(\overline{M},\Omega_1)$ and $\pi$ respectively.  Applying \cref{prop:KahlerComplementEmbedding} to $\pi^{-1}(\infty)$ in $\overline{M}$, there exists a K\"{a}hler form $\widetilde{\Omega}$ on $M$ and a symplectic embedding $\psi: (M,\Omega_1) \hookrightarrow (M,\widetilde{\Omega})$ that is the identity near $\pi^{-1}(0)$ and whose image, $\psi(M)$, is a compact subset of $M$.  Since the subsets $M_a$ for $a>0$ give a compact exhaustion of the space $M$, we have that there exists a constant $a>0$ such that $\psi(M)$ is contained in $M_{a/2}$. 

Since $\widetilde{\Omega}$ is K\"{a}hler, it is compatible with the almost complex structure that makes the projection $\pi: M \to \IC$ holomorphic.  Consequently, the fibres of $\pi$ are symplectic submanifolds with respect to $\widetilde{\Omega}$, and $(M,\pi, \widetilde{\Omega})$ defines a Hamiltonian fibration over $\IC$.  By \cref{prop:FlatteningOverCStarNonDeg}, there exists a symplectic form ${\Omega} = \omega_{\IC} + \omega_F$ on $M$, where $\omega_F$ is the restriction of $\wt{\Omega}$ to $F$, and a symplectic embedding $\mu: (M,\widetilde{\Omega}) \hookrightarrow (M,{\Omega})$.  The image of $\mu \circ \psi(M)$ lands in $M_a$ (after we replace our old $a$ by a possibly larger $a$).  Moreover, $\mu$ fixes $\pi^{-1}(0) \cong 0 \times F$.

Now for each $c>0$, $(M_c,\Omega)$ admits the structure of a convex symplectic domain.  So action completed symplectic cohomology is well-defined for the subdomains $(M_c,\Omega)$.  Since $M_c$ is symplectically given by the product $\ID_c \times F$, for some $c\gg b> a$, $M_b$ is displaceable in $(M_c,{\Omega})$.  Thus, by \cref{prop:PropertiesOfSH} \cref{prop:VanishingForStDisp}, $\widehat{SH}(M_b \subset M_c;0) \otimes \Lambda = 0$.

Define
\[ \psi_z: (h_z^{-1}(\IC), \Omega_z) \to (M,\Omega) \]
by the composite $\psi_z \coloneqq \mu \circ \psi \circ \varphi_z$.  $\psi_z$ fixes $\pi^{-1}(0)$.  To see that $\Omega$ is integral, one can either notice that $\Omega$ is cohomologous to the original K\"{a}hler form, which is integral (as all of the above embeddings are produced via Moser's argument), or that $\omega_F$ is the restriction of an integral form, and, thus, is integral.  So these maps and the symplectic form $\Omega$ satisfy the conclusions of the lemma.
\end{proof}

\subsection{Producing a sequence of curves}

We prove \cref{lem:SequenceOfCurves}.

\begin{notn}
\begin{enumerate}
	\item Let $z_\nu \coloneqq 1/\nu \in \IC$ for $\nu \in \IZ_{\geq 1}$.
	\item Let $p_\nu \in h_{z_\nu}^{-1}(\ID_1)$ be a sequence of points such that $\lim_{\nu\to \infty} p_{\nu} = p \in M$.
	\item Let $A_\nu \coloneqq h_{z_\nu}^{-1}([0,3])$.
	\item Let $\Omega_\nu \coloneqq \Omega_P|_{P_{z_\nu}}$.
	\item Let $J_\nu \coloneqq J_P|_{A_{\nu}}$.
\end{enumerate}
\end{notn}

\begin{proof}
We assume the notation and results of \cref{lem:ModifiedParallelTransport} and \cref{lem:ModifiedParallelTransportCY}.  Since $\psi_{z_\nu}|_{A_\nu}$ is a symplectic embedding, we may extend the almost complex structure $(\psi_{z_\nu}|_{A_\nu})_*(J_\nu)$ to an admissible almost complex structures on $(M_c,\Omega)$, which we denote by $J_\nu'$.  Also, we set $p_\nu' \coloneqq  \psi_{z_\nu}(p_\nu)$. 

By \cref{prop:DiskUniruling}, there exists a constant $E>0$ and (possibly disconnected) genus zero, compact, $J_\nu'$-holomorphic curves $u_\nu': \Sigma_\nu' \to M$ with boundaries that satisfy:
\begin{enumerate}
	\item $u_\nu'(\partial \Sigma_\nu') \subset M \smallsetminus M_a$,
	\item $u_\nu'$ is nowhere constant,
	\item $u_\nu'(\Sigma') \cap p_\nu' \neq \varnothing$, and
	\item the energy of $u_\nu'$ with respect to $\Omega$ and $J_\nu'$ is bounded by $E$.
\end{enumerate}
The image of $\psi_{z_\nu}$ is contained in $M_a$.  So by Sard's theorem, there exists a measure $1$ subset of $[1,2]$ such that for $\nu$ and all $k$ in said subset
\[ \Sigma_\nu \coloneqq (u_\nu')^{-1}(\psi_{z_\nu}(h_{z_\nu}^{-1}(\ID_k))) \]
is a smooth Riemann surfaces with boundary.  Define $Q \coloneqq h^{-1}(\ID_k)$
and $u_\nu: \Sigma_\nu \to Q$ by $u_\nu \coloneqq (\psi_{z_\nu})^{-1} \circ u_\nu'|_{\Sigma_\nu}$.  The curves $u_\nu$ satisfy the conditions of the lemma by construction.  The degree one assumption on the $u_\nu$ in the case where $\pi$ is smooth over $0$ follows from \cref{lem:ModifiedParallelTransport},  \cref{prop:ProductUniruling}, and the fact that this degree is preserved by the maps $\psi_z$ (since they are homotopy equivalences that fix the central fibre $\pi^{-1}(0)$).
\end{proof}


\part{Discussion of the Calabi-Yau case}

In this part, we prove \cref{lem:ModifiedParallelTransportCY}.  First, we discuss Conley-Zehnder indices of orbits of Hamiltonians.  Second, we discuss a rescaling isomorphism for action complete symplectic cohomology groups of convex symplectic domains with vanishing first Chern class whose ends are modeled after symplectic mapping cylinders.  This isomorphism is established by assuming the existence of cofinal sequences of Hamiltonians that satisfy certain index bounded assumptions.  Thirdly, we discuss symplectic normal crossings divisors and a sequence of Hamiltonians whose dynamics are nicely captured by the normal crossings structure, and, in particular, satisfy an index bounded assumption.  Finally, we will combined our work from the above three sections with the symplectic deformations from \cref{part:SymplecticDeformations} to conclude \cref{lem:ModifiedParallelTransportCY}.


\section{Conley-Zehnder Indices}\label{sec:CZIndices}

We discuss Conley-Zehnder indices of contractible $1$-periodic orbits of Hamiltonians.

\begin{defn}\label{defn:CZIndex}
The \emph{Conley-Zehnder index} $CZ(A_t)$ is a half-integer associated to a path of symplectic matrices $(A_t)_{t \in [a,b]}$ \cite{RobbinSalamon_CZIndex}.  It satisfies the following properties \cite{Gut_GeneralizedCZIndex}:
\begin{enumerate}[label= (CZ\arabic*)]
	\item \label{CZ1} $CZ((\exp(i t))_{t \in [0,2\pi]}) = 2$,
	\item \label{CZ2} $CZ((A_t)_{t \in [a,b]} \oplus (B_t)_{t \in [a,b]}) = CZ(A_t) + CZ(B_t)$,
	\item \label{CZ3} the Conley-Zehnder index of a concatenation of two paths of symplectic matrics is the sum of the Conley-Zehnder indices of the individual paths,
	\item \label{CZ4} the Conley-Zehnder index is invariant under homotopies of paths of symplectic matrices relative to the end points, and
	\item \label{CZ5} if $CZ(A_t)$ is a loop of symplectic matrices, then $CZ(A_t)$ is equal to two times the Maslov class of the loop $A_t$.
\end{enumerate}
\end{defn}

Consider an arbitrary symplectic manifold $(M,\Omega)$.  When $c_1(M) \neq 0$, a capping of a period orbit defines a Conley-Zehnder index for the orbit, which depends on the choice of capping.

\begin{defn}
Let $\gamma: S^1 \to M$ be a contractible periodic orbit of a Hamiltonian $H$.  Let $v: \ID \to M$ be a smooth map with $v(\exp(it)) = \gamma(t)$.  Let $\tau_v: v^* TM \to \ID \times \IC^n$ be a symplectic trivialization.  The \emph{Conley-Zehnder index of $\gamma$ with respect to $v$}, $CZ(\gamma,v)$, is the Conley-Zhender index of the path of symplectic matrices
\[ (\tau_v|_{\gamma(t)}) \circ (d \phi_t^H|_{\gamma(0)}) \circ (\tau_v|_{\gamma(0)})^{-1}. \]
\end{defn}

The index $CZ(\gamma,v)$ is independent of the choice of trivialization $\tau_v$.  Moreover, given any two cappings of $\gamma$, say $v_0$ and $v_1$,
\[ CZ(\gamma,v_0) = CZ(\gamma,v_1) + 2 \langle c_1(M), v_0 \# (-v_1) \rangle. \]

Now suppose $c_1(M)=0$.  A trivialization of the anti-canonical bundle of $M$ defines a Conley-Zehnder index for a periodic orbit.

\begin{notn}
Let $\sK^*$ denote the anti-canonical bundle of $M$.  We assume that $c_1(M) = 0$.  So there is a trivialization $\tau: \sK^* \to M \times \IC$.
\end{notn}

\begin{defn}
Let $\gamma: S^1 \to M$ be a contractible periodic orbit of a Hamiltonian $H$.  Fix a trivialization $\tau_\gamma: \gamma^*TM \to M \times \IC^n$ such that the top exterior power of $\tau_\gamma$ agrees with $\tau$.  The \emph{Conley-Zehnder index of $\gamma$ with respect to $\tau$}, $CZ(\gamma,\tau)$, is the Conley-Zhender index of the path of symplectic matrices
\[ (\tau_\gamma|_{\gamma(t)}) \circ (d \phi_t^H|_{\gamma(0)}) \circ (\tau_\gamma|_{\gamma(0)})^{-1}. \]
\end{defn}

Extending the argument of \cite[Lemma 4.3]{McLean_MinimalDiscrepancy} from Reeb orbits to contractible Hamiltonian orbits gives:

\begin{lem}\label{lem:IndOfCZc1Vanishing}
Suppose $c_1(M)=0$.  If $\gamma$ is a contractible periodic orbit of a Hamiltonian $H$ on $M$, then
\[ CZ(\gamma,\tau) = CZ(\gamma,v)\] 
for any choice of capping $v$ and any choice of trivialization $\tau$.  In particular, both definitions agree and are independent of all choices.\qed
\end{lem}

\begin{notn}
In light of \cref{lem:IndOfCZc1Vanishing}, if $\gamma$ is a contractible periodic orbit of a Hamiltonian $H$ on $M$ and $c_1(M)=0$, write
\[ CZ(\gamma) \coloneqq CZ(\gamma,\tau). \] 
\end{notn}

We now discuss Conley-Zehnder indices of pseudo Morse-Bott families of $1$-periodic orbits of Hamiltonians and how these indices change under small perturbations.  Our discussion is analogous to that of McLean's \cite{McLean_MinimalDiscrepancy}, who defined pseudo Morse-Bott families of Reeb orbits.

\begin{defn}
Let $\Gamma$ be a collection of $1$-periodic orbits of a Hamiltonian $H$.  The \emph{fixed points} of $\Gamma$ is the set
\[ \Gamma(0) = \{ \gamma(0) \in M \mid \gamma \in \Gamma\} .\]
The collection $\Gamma$ is \emph{isolated} if there exists an open neighborhood $U_\Gamma$ of $\Gamma(0)$ such that if $\gamma(0)$ is in $U_\gamma$, then $\gamma \in \Gamma$.
\end{defn}

\begin{defn}\label{defn:PseudoMorseBott}
A \emph{pseudo Morse-Bott family} of $1$-periodic orbits of a Hamiltonian $H$ is an isolated collection of $1$-periodic orbits $\Gamma$ such that
\begin{enumerate}
	\item $\dim(\ker(d \varphi_1^H - \Ione))$ is  constant on $\Gamma(0)$, and
	\item $\Gamma(0)$ is path connected.
\end{enumerate}
The \emph{size} of $\Gamma$ is defined to be $\dim( \ker( d \varphi_t^H -\Ione))$.
\end{defn}

We now assume for the remainder of this subsection that $c_1(M)=0$.

\begin{defn}
The \emph{Conley-Zehnder index} of a pseudo Morse-Bott family $\Gamma$ is $CZ(\Gamma) \coloneqq CZ(\gamma)$ for any $\gamma \in \Gamma$.
\end{defn}

If a Hamiltonian has a pseudo Morse-Bott family (that is not comprised of a single isolated orbit), then the Hamiltonian will be degenerate.  So we would like to perturb the Hamiltonian to obtain a non-degenerate Hamiltonian.  The following lemma relates the Conley-Zehnder index of the pseudo Morse-Bott family and the Conley-Zehnder indices of the orbits of the perturbed Hamiltonian.  It follows from extending the argument of \cite[Lemma 4.10]{McLean_MinimalDiscrepancy} from Reeb orbits to contractible Hamiltonian orbits.

\begin{lem}\label{lem:PerturbWithCZ}
Let $\Gamma$ be a pseudo Morse-Bott family of size $k$.  There exists a constant $\delta > 0$ and a neighborhood $U$ of $\Gamma(0)$ such that if $\wt{H}$ is a Hamiltonian that satisfies
\begin{enumerate}
	\item $\wt{H}|_{U}$ is non-degenerate, and
	\item $|\wt{H} - H|_{C^2} < \delta$,
\end{enumerate}
and $\wt{\gamma}$ is a $1$-periodic orbit of $\wt{H}$ with $\wt{\gamma}(0) \in U$, then
\[ CZ(\Gamma) - k/2 \leq CZ(\wt{\gamma}) \leq  CZ(\Gamma) + k/2. \]\qed
\end{lem}


\section{A rescaling isomorphism}\label{sec:RescalingIso}

We assume $(M,\Omega, \lambda)$ is a convex domain with stable Hamiltonian boundary given by the mapping torus of a symplectomorphism.  We will be using some of the discussion on these spaces from \cref{subsec:HamFibs_MappingTori}.

Fix polar coordinates for the collar of the boundary, $(0,1] \times \partial M$, and write $\Omega = \eta + dr \wedge d \theta$, where $\eta = \Omega|_{\partial M}$, as in \cref{lem:1FormDeterminesMappingTorus}.  With this stable Hamiltonian structure, $\lambda = r d\theta$.  Let $J$ be an admissible almost complex structure for $(M,\Omega,\lambda)$.  As in \cref{lem:MappingTorusAdmissibleJ}, write
\[ J = \begin{pmatrix} 0 & -1 & 0 \\ 1 & 0 & 0 \\ 0 & 0 & J_F \end{pmatrix} \]
with respect to the splitting $r \cdot \partial_r \oplus \widetilde{\partial_\theta} \oplus TF$.
We assume that the family $J_F$ is independent of $r$.  Define $M_c$ for $c>1$ by
\[ M_c = M \cup ([1,-1+c) \times \partial M). \]
The data $\Omega$, $\lambda$, and $J$ on $(0,1] \times \partial M$ extend to data on $M_c$ in the obvious manner.

Suppose that $K$ is a compact, codimension zero submanifold of $M_c$ with connected boundary such that $\partial K$ is properly contained in the collar neighborhood of the boundary of $M_c$.  Let
\[ K_\delta \coloneqq \{ (r,x) \in M_c \mid (r-\delta,x) \in K\},\]
{with the convention that if $r - \delta < 0$, then $(r-\delta,x) \in K$.

\begin{figure}
\centering
\includegraphics[width=.25\linewidth]{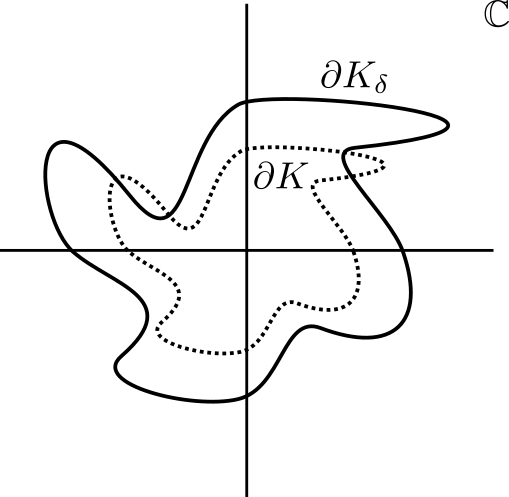}
\caption{An example of subsets $K$ and $K_\delta$ when $M = \IC$.}
\label{fig:RadialStretching}
\end{figure}

For example, $(M_b)_\delta = M_{b+\delta}$.  The goal is to establish sufficient conditions for an isomorphism
\[\widehat{SH}(K \subset M_c; 0, \tau) \otimes \Lambda \cong \widehat{SH}(K_{\delta} \subset M_{c+\delta}; 0, \tau) \otimes \Lambda. \]
We study how the Floer theory of our spaces behaves with respect to rescaling of the collar coordinate.

Given constants $0 < a < b < c$ and $\delta>0$, let $\ell: (0,+\infty) \to (0,+\infty)$ be a smooth function that satisfies:
\begin{itemize}
	\item $\ell(r) = r$ for $r \leq a$,
	\item $\ell(r) = \delta + r$ for $r \geq b$, and
	\item $\ell'(r) \geq 1$ for all $r$.
\end{itemize}
We have a diffeomorphism $(0,c] \times \partial M \to (0,\delta + c] \times \partial M$ given by
\[ \varphi(r,x) = (\ell(r),x). \]
As $\varphi$ is the identity for $r \leq a$, it extends to a diffeomorphism $M_{c} \to M_{c+\delta}$.  Notice that $\varphi$ is not a symplectomorphism; nevertheless, the almost complex structure $J$ transforms nicely with respect to $\varphi$.

\begin{lem}
The almost complex structure $\varphi_* J = \varphi_* \circ J \circ \varphi_*^{-1}$ is $\Omega$-compatible.
\end{lem}

\begin{proof}
We need to show that $\Omega(\cdot, \varphi_* \circ J \circ \varphi_*^{-1} \cdot)$ is a metric.  It suffices to show that pulling-back this $2$-tensor along $\varphi$ produces a metric.  Computing
\[ \varphi^*(\Omega(\cdot, \varphi_* \circ J \circ \varphi_*^{-1} \cdot)) = (\varphi^* \Omega)(\cdot, J \cdot) = \eta(\cdot, J \cdot) + \ell' \cdot dr \wedge d\theta (\cdot,J \cdot) = \Omega(\cdot, J \cdot) + (\ell'-1) \cdot dr \wedge d\theta(\cdot, J \cdot)\]
$\Omega(\cdot,J \cdot)$ is a metric since $J$ is $\Omega$-compatible.  Also, since $\ell' \geq 1$ and the local projection $N(\partial M) \to (0,c] \times S^1$ is holomorphic, $(\ell'-1) \cdot dr \wedge d\theta (\cdot, J \cdot)$ is symmetric and non-negative definite.  So the pull-back is a metric, as desired.
\end{proof}

In the collar neighborhood, $\varphi_* J$ is not an admissible almost complex structure in the sense of \cref{defn:AdmissibleJOnSymplectization}.  Indeed, for $r \geq c + \delta - \varepsilon$ with $\varepsilon >0$ sufficiently small,
\[ \varphi_* J = \begin{pmatrix} 0 & \frac{-(r-\delta)}{r} & 0 \\ \frac{r}{r-\delta} & 0 & 0 \\ 0 & 0 & J_F \\ \end{pmatrix} \]
with respect to the basis obtained from $r \cdot \partial_r$ and $\widetilde{\partial_\theta}$ and a basis for the vertical distribution.  Define $\kappa: [c+\delta-\varepsilon, c+\delta] \to \IR$ by
\begin{itemize}
	\item $\kappa(r) = 1$ locally near $r = c+\delta-\varepsilon$,
	\item $\kappa(r) = 0$ locally near $r = c+\delta$, and
	\item $\kappa' \leq 0$.
\end{itemize}
Define $\widetilde{J}$ via
\[ \wt{J} \coloneqq \begin{pmatrix} 0 & \frac{-(r-\kappa(r) \cdot \delta)}{r} & 0 \\ \frac{r}{r-\kappa(r) \cdot \delta} & 0 & 0 \\ 0 & 0 & J_F \\ \end{pmatrix}\]
for $r \geq c+\delta - \varepsilon$ and $\widetilde{J} = \varphi_* J$ for $r \leq c + \delta - \varepsilon$.  $\widetilde{J}$ is admissible for $(M_{c+\delta},\Omega,\lambda)$ in the sense of \cref{defn:AdmissibleJOnSymplectization} with collar taken sufficiently close to $\partial M_{c+\delta}$.  Moreover, $\widetilde{J}$ is weakly admissible for $(M_{c+\delta},\Omega,\lambda)$ in the sense of \cref{defn:WeaklyConvexSymplecticDomain} with the collar taken to be $r \geq c+\delta - \varepsilon$.  Indeed, one sets $f$ to be the anti-derivative of $r/(r-\delta \cdot \kappa(r))$.

We now discuss rescalings of Hamiltonians.

\begin{defn}\label{defn:PhiScalable}
A (monotonically) admissible family of Hamiltonians $H^\sigma$ is \emph{$\varphi$-scalable} if
\begin{enumerate}
	\item $H^\sigma_{\underline{s}} = m_{\underline{s}}^\sigma \cdot r + k_{\underline{s}}^\sigma$ for some family of constants $m_{\underline{s}}^\sigma$ and $k_{\underline{s}}^\sigma$ for $a \leq r \leq b$, and
	\item $H^\sigma_{\underline{s}}$ is radial with constant slope (independent of $\underline{s}$) for $r \geq c-\varepsilon$.
\end{enumerate}
\end{defn}

\begin{rem}
We could work with more general Hamiltonians, but it is hardly worth the effort.
\end{rem}

\begin{lem}\label{lem:Scaling}
If $H^\sigma$ is $\varphi$-scalable, then there exists a (monotonically) admissible family of Hamiltonians $H^{\widetilde{\sigma}}$ on $M_{c+\delta}$ such that $\varphi_* X_{H^\sigma} = X_{{H}^{\widetilde{\sigma}}}$.
\end{lem}

\begin{proof}
Define
\[ H^{\widetilde{\sigma}} = \begin{cases} H^\sigma, & r \leq a \\ m^\sigma \cdot r + k^\sigma, & a \leq r \leq b+\delta \\ H^\sigma \circ \varphi^{-1}(r,x) + \delta\cdot m^\sigma, & b+\delta \leq r \\ \end{cases}.\]
By \cref{lem:HorizontalLiftOfHamiltonianVectorField}, over $[a,b] \times \partial M$,
\[ X_{H^\sigma} = \partial_r H^\sigma \cdot \widetilde{\partial_\theta} = m^\sigma \cdot \widetilde{\partial_\theta}. \]
So $\varphi_* X_{H^\sigma} = m^\sigma \cdot \widetilde{\partial_\theta} = X_{H^{\widetilde{\sigma}}}$ for $a \leq r \leq b+\delta$.  Also,
\[ (\varphi_* X_{H^\sigma})(r,x) = X_{H^\sigma}(\varphi^{-1}(r,x)) = X_{H^{\widetilde{\sigma}}}(r,x) \]
for $r \leq a$ and $r \geq b+\delta$ since
\begin{align*}
\Omega(X_{H^{\widetilde{\sigma}}}(r,x),\cdot) & = -dH^{\widetilde{\sigma}}(r,x)(\cdot) = -dH^{\sigma} (\varphi^{-1}(r,x)) \circ (\varphi^{-1})_*(r,x)(\cdot) 
 = \Omega(X_{H^{\sigma}}(\varphi^{-1}(r,x)),\cdot).
\end{align*}
Finally, the (monotonic) admissibility of $H^{\widetilde{\sigma}}$ follows from its construction and the (monotonic) admissibility of $H^\sigma$.
\end{proof}

Suppose that $H^\sigma$ is $\varphi$-scalable.  Since $\varphi_* X_{H^\sigma} = X_{H^{\widetilde{\sigma}}}$, we obtain a bijection between the $1$-periodic orbits of $H^\sigma_{\underline{e_i}}$ and the $1$-periodic orbits of $H^{\widetilde{\sigma}}_{\underline{e_i}}$ for all $i$.  Denote this by $x \leftrightarrow \widetilde{x}$.

If $u$ is a Floer trajectory of type $((H^\sigma,J),x_0,x_\ell)$, then $\varphi \circ u$ is a Floer trajectory of type $((H^{\widetilde{\sigma}}, \widetilde{J}),\widetilde{x}_0,\widetilde{x}_\ell)$.  Moreover, every Floer trajectory of type $((H^{\widetilde{\sigma}}, \widetilde{J}),\widetilde{x}_0,\widetilde{x}_\ell)$ arises in this manner.  Indeed, since $H^{\widetilde{\sigma}}$ is radial with constant slope for $r \geq c+\delta-\varepsilon$, all of its orbits lie in the region where $r < c-\varepsilon$.  Since $\widetilde{J}$ is admissible in the sense of \cref{defn:WeaklyConvexSymplecticDomain} for $r \geq c+\delta-\varepsilon$, \cref{prop:MaximumPrincipleForWeaklyConvexDomains} implies that any Floer trajectory of type $((H^{\widetilde{\sigma}},\widetilde{J}), \widetilde{x}_0, \widetilde{x}_\ell)$ lies completely in $M_{c+\delta-\varepsilon}$.  Consequently, if $v$ is of type $((H^{\widetilde{\sigma}}, \widetilde{J}),\widetilde{x}_0,\widetilde{x}_\ell)$, then $\varphi^{-1} \circ v$ is of type $((H^\sigma,J),x_0,x_\ell)$.  We have used that $\varphi^{-1}$ is $(\widetilde{J},J)$-holomorphic over the image of $v$.  This gives a canonical identification of $\overline{\sM}((H^\sigma,J),x_0,x_\ell)$ and $\overline{\sM}(({H}^{\widetilde{\sigma}}, \widetilde{J}),\widetilde{x}_0,\widetilde{x}_\ell)$.  We denote it by $u \leftrightarrow \widetilde{u} = \varphi \circ u$.

\begin{lem}\label{lem:CFRescalingIso}
For every $\varphi$-scalable Hamiltonian $H$, there is an isomorphism over $\Lambda$,
\[ CF(H,J) \to CF(\widetilde{H},\widetilde{J}), \hspace{10pt} x \mapsto \widetilde{x} \cdot T^{C(x)}, \]
where
\[ C(x) = -\int_{S^1} x^*((\ell(r) - r) d\theta) + \int_{S^1} \widetilde{H}(\widetilde{x}(t)) - H(x(t)) \ dt.\]
Moreover, for each (monotonically) admissible family $H^\sigma$, there exists a strictly commutative diagram
\[ \xymatrix{CF(H_{\underline{e_1}}^\sigma, J) \ar[r] \ar[d]^{c(\sigma)} & CF({H}_{\underline{e_1}}^{\widetilde{\sigma}}, \widetilde{J}) \ar[d]^{c(\widetilde{\sigma})} \\ CF(H_{\underline{e_0}}^\sigma, J) \ar[r] & CF({H}_{\underline{e_0}}^{\widetilde{\sigma}},\widetilde{J}). \\} \]
\end{lem}

\begin{rem}
This type of rescaling isomorphisms for action completed symplectic cohomology groups is not new.  Similar ideas appeared in \cite{McLean_BirationalCalabiYauManifoldsHaveTheSameSmallQuantumProducts}, and were later systematically studied in \cite{tonkonog2020superrigidity}.  The novelty in our approach lies in that we establish the rescaling isomorphism for boundaries given by mapping tori of symplectomorphisms as opposed to contact-type boundaries.  In light of this, our analogue of an index-bounded type assumption that appears in the above papers is also different.
\end{rem}

\begin{proof}
Define
\[ \Phi: CF(H,J) \to CF(\widetilde{H},\widetilde{J}), \hspace{10pt} \Phi(x) = \widetilde{x} \cdot T^{C(x)},\]
where $C(x)$ is given in the statement of the lemma.  Clearly, $\Phi$ is an isomorphism over $\Lambda$.  It just remains to show that it is a chain map.  It suffices to show that
\[ E_{top}(\widetilde{u}) - E_{top}(u) = C(x_-) - C(x_+), \]
where $u$ is of type $((H^\sigma,J),x_-,x_+)$.
\begin{align*}
E_{top}&(\widetilde{u}) - E_{top}(u) 
\\ & = \int \widetilde{u}^* \Omega - \int \widetilde{H}(\widetilde{x}_+) + \int \widetilde{H}(\wt{x}_-) - \int u^*\Omega + \int H(x_+) - \int H(x_-)
\\  & = \int u^*(\varphi^* \Omega - \Omega) - \int \widetilde{H}(\widetilde{x}_+) + \int \widetilde{H}(\wt{x}_-) + \int H(x_+) - \int H(x_-)
\\  & = \int u^*(\eta + \ell' dr \wedge d\theta - \eta - dr \wedge d\theta) - \int \widetilde{H}(\widetilde{x}_+) + \int \widetilde{H}(\wt{x}_-) + \int H(x_+) - \int H(x_-)
\\  & = \int x_+^*((\ell(r)-r) d\theta) - \int x_-^*((\ell(r)-r)d\theta) - \int \widetilde{H}(\widetilde{x}_+) + \int \widetilde{H}(\wt{x}_-) + \int H(x_+) - \int H(x_-)
\\ & = C(x_-) - C(x_+).
\end{align*}
An analogous computation shows the diagram for continuation maps is also commutative.
\end{proof}

Now we can establish the rescaling isomorphism that we will need to prove \cref{lem:ModifiedParallelTransportCY}.  Suppose that $K$ is a compact, codimension zero submanifold of $M_c$ with connected boundary such that $b< r(\partial K) < c$.

\begin{lem}
Let $H_n$ be a cofinal sequence of Hamiltonians in $\sH^+(K \subset M_c;f,\tau)$.  Suppose that each $H_n$ is $\varphi$-scalable.  Then $\widetilde{H}_n$ is a cofinal sequence of Hamiltonians in $\sH^+(K_{\delta} \subset M_{c+\delta}; \widetilde{f},\tau)$, where $\widetilde{f}$ is some function on $M_{c+\delta}$.
\end{lem}

\begin{proof}
Since the $H_n$ are $\varphi$-scalable, this sequence must converge to a linear function when $a \leq r \leq b$.  Write $f = m \cdot r + k $ for $a \leq r \leq b$.  Define
\[ \widetilde{f} = \begin{cases} f, & r \leq a \\ m \cdot r +k, & a \leq r \leq b+\delta \\ f \circ \varphi^{-1}(r,x)+\delta \cdot m,& b +\delta \leq r. \\ \end{cases} \]
By definition of $\widetilde{H}_n$, it is a cofinal sequence of Hamiltonians in $\sH^+(K_{\delta} \subset M_{c+\delta}; \widetilde{f},\tau)$.
\end{proof}

\begin{lem}\label{lem:SHRescalingIso}
Suppose that $c_1(M) = 0$.  If there exists a cofinal sequence of $\varphi$-scalable Hamiltonians $H_n$ in $\sH^+(K \subset M_c; f,\tau)$ and constants $W_\ell > 0$ for each $\ell \in \IN$ such that all $1$-periodic orbits $x$ of the Hamiltonians $H_n$ satisfy
\[ |CZ(x)| \leq \ell \Longrightarrow \left| \int x^* d\theta \right| \leq W_\ell,\]
then $\widehat{SH}(K \subset M_c; f,\tau) \otimes \Lambda \cong \widehat{SH}(K_\delta \subset M_{c+\delta}; \widetilde{f},\tau) \otimes \Lambda$.
\end{lem}

\begin{proof}
By \cref{con:TheBump}, we can construct monotonically admissible families $H^{\sigma_n}$ such that $H^{\sigma_n}_{\underline{e_0}} = H_{n+1}$ and $H^{\sigma_n}_{\underline{e_1}} = H_n$.  We may assume that $H^{\sigma_n}$ is $\varphi$-scalable.  So by \cref{lem:CFRescalingIso}, we have a strictly commutative diagram
\[ \xymatrix{ \cdots \ar[r] & CF(H_n,J) \ar[r]^{c(\sigma_n)} \ar[d] & CF(H_{n+1},J) \ar[r] \ar[d] & \cdots \\ \cdots \ar[r] & CF(\widetilde{H}_n, \widetilde{J}) \ar[r]^{c(\wt{\sigma}_n)} & CF(\widetilde{H}_{n+1},\widetilde{J}) \ar[r] & \cdots. }\]
This induces an isomorphism of mapping telescopes
\[ \Tel(\sC\sF(\{H_n\})) \otimes \Lambda \to \Tel(\sC\sF(\{\widetilde{H}_n\}))\otimes \Lambda. \]
Since $c_1(M)=0$, we obtain a genuine $\IZ$-grading on our Floer chain complexes by the Conley-Zehnder index (not the $2$-periodic grading mentioned in \cref{subsec:CF}).  Since \cref{defn:CompletionFunctor} is a degree-wise completion, if the above isomorphism is degree-wise bounded, then it will give the desired isomorphism of action completed symplectic cohomology groups.

So we need to show that for each $\ell \in \IN$ there exists $C_\ell$ such that if $x$ is a $1$-periodic orbit of $H_n$ with $|CZ(x)| = \ell$, then $|C(x)| \leq C_\ell$.  Consider such an $x$.  By construction $r(x) \leq a$ or $r(x) \geq b$.  If $r(x) \leq a$, then $\widetilde{H}_n(\widetilde{x}) - H_n(x) = 0$.  If $r(x) \geq b$, then $\widetilde{H}_n(\widetilde{x}) - H_n(x) = \delta \cdot m_n \leq \delta \cdot m$, where $m$ is as in the construction of $\wt{f}$.  Setting $C_\ell = \delta \cdot W_\ell +\delta \cdot m$, we have that
\begin{align*}
|C(x)| & = \left| - \int x^*(\ell-r) d\theta + \int \widetilde{H}_n(\widetilde{x}) - \int H_n(x) \right| 
\leq \delta \left| \int x^* d\theta \right| + \delta \cdot m 
= C_\ell.
\end{align*}
\end{proof}


\section{Hamiltonians adapted to normal crossings divisors}\label{sec:DivisorModels}

We wish to use \cref{lem:SHRescalingIso} to establish \cref{lem:ModifiedParallelTransportCY}.  So we need to construct an appropriate sequence of cofinal Hamiltonians that satisfy the index-bounded assumptions of \cref{lem:SHRescalingIso}.  We consider the set up in \cref{sec:ProofOfTheorem} and note that \cite{Hironaka_ResolutionOfSingularities} gives a birational map $\wt{M} \to M$ that is an isomorphism away from $\pi^{-1}(0)$ such that its composition with $\pi$, $\wt{\pi}: \wt{M} \to \IC$, is smooth away from $\wt{\pi}^{-1}(0)$, and $\wt{\pi}^{-1}(0) \subset \wt{M}$ is a normal crossings divisor.  We will associate to a neighborhood of this normal crossings divisor a cofinal sequence of Hamiltonians whose Conley-Zehnder indices are controlled by the algebraic topology of the divisor.  These Hamiltonians will give rise to our desired sequence of index-bounded Hamiltonians.  To construct these Hamiltonians, we need to deform a neighborhood of the normal crossings divisor in such a manner that the neighborhoods of the open strata of the normal crossings divisor admit a system of compatible, symplectic tubular neighborhoods.  So the remainder of this subsection is divided up as follows.  First, we discuss a purely symplectic notion of normal crossings divisors, and the existence of a standard symplectic neighborhood of these divisors.  Our discussion closely follows that of \cite[Section 6.2]{McLean_BirationalCalabiYauManifoldsHaveTheSameSmallQuantumProducts}, and we claim no originality for this material.  Second, we construct the above mentioned Hamiltonians.

For now, we consider an arbitrary symplectic manifold $(M,\Omega)$.

\begin{defn}\cite[Definition 6.10]{McLean_BirationalCalabiYauManifoldsHaveTheSameSmallQuantumProducts}
A \emph{tubular neighborhood} of a smooth submanifold $Q \subset M$ is an open neighborhood $U$ of $Q$ in $M$ and a smooth fibration $\pi: U \to Q$ such that
\begin{enumerate}
	\item for a Riemannian metric $g$ on $M$, $U = \exp^{-1}(DQ)$, where
	\[ DQ \coloneqq \{ (x,v) \in TM \mid x \in Q, g(v,v) < 1, g(v,w) = 0 \mbox{ for all }w \in T_x Q\}, \]
	and
	\item $\pi = \pi_{DQ} \circ \exp^{-1}$, where $\pi_{DQ}: DQ \to Q$ is the natural projection.
\end{enumerate}
\end{defn}

\begin{notn}
Let $\sI$ be a finite indexing set.  Given $I \subset \sI$, set $U(1)^I \coloneqq \prod_{i \in I} U(1)^{\{i\}}$ and $\ID^I \coloneqq \prod_{i \in I} \ID^{\{i\}}$, where $\ID$ is the unit disk.
\end{notn}

\begin{defn}\cite[Section 6.2]{McLean_BirationalCalabiYauManifoldsHaveTheSameSmallQuantumProducts}
For $I \subset \sI$, a \emph{symplectic $U(1)^I$  neighborhood} of a symplectic submanifold $Q \subset (M,\Omega)$ is a tubular neighborhood $\pi: U \to Q$ of $Q$ such that
\begin{enumerate}
	\item $\pi^{-1}(x)$ is a symplectic submanifold, symplectomorphic to $\ID^I$,
	\item $\pi$ has structure group $U(1)^I \coloneqq \prod_{i \in I} U(1)$ given by acting on the $\ID^I$ diagonally, and
	\item the symplectic parallel transport map of $(U,\pi,\Omega|_U)$ is well-defined\footnote{This tuple defines a Hamiltonian fibration (with open fibre) and thus determines a horizontal distribution with respect to which we can parallel transport.  Given that the fibre is open, this parallel transport need not be well-defined.} and has holonomy lying in $U(1)^I$\footnote{that is, the symplectic parallel transport maps respect the $U^I$ structure groups}.
\end{enumerate}
\end{defn}

Let $\pi: U \to Q$ be a symplectic $U(1)^I$ neighborhood.  Given $J \subset I$, the action of $U(1)^J \subset U(1)^I$ on $U$ gives rise to a $U(1)^J$-bundle $\pi^J: U \to U^J$, where $U^J$ is the fixed locus of the $U(1)^J \subset U(1)^I$ action.  The restriction of $\pi^J$ to a fibre of $\pi$ is the projection $\ID^I \to \ID^J$.  This gives rise to a symplectic $U(1)^J$ neighborhood of $U^J$ inside of $U$.

\begin{defn}\cite[Section 6.2]{McLean_BirationalCalabiYauManifoldsHaveTheSameSmallQuantumProducts}
A \emph{symplectic crossings divisor} in $(M,\Omega)$ is a finite collection $(D_i)_{i \in \sI}$ of transversally intersection submanifolds of $(M,\Omega)$ such that
\begin{enumerate}
	\item $D_I = \cap_{i \in I} D_i$ is symplectic in $(M,\Omega)$,
	\item the orientation of $D_I$ from $\Omega|_{D_I}$ agrees with the orientation of $D_I$ from the orientation of $M$ from $\Omega$ and the orientation of the normal bundle of $D_I$, denoted $ND_I$, induced from the splitting $ND_I = \oplus_{i \in I} N D_i|_{D_I}$.
\end{enumerate}
\end{defn}

\begin{defn}\label{defn:StandardTubularNeighborhood}\cite[Section 6.2]{McLean_BirationalCalabiYauManifoldsHaveTheSameSmallQuantumProducts}
A \emph{standard tubular neighborhood} of a symplectic crossing divisor $(D_i)_{i \in \sI}$ is a collection of symplectic $U(1)^I$ neighborhoods $\pi_I: U_I \to D_I$ for each $I \subset \sI$ such that
\begin{enumerate}
	\item $U_I \cap U_J = U_{I \cup J}$, and
	\item $\pi_J(U_I) = U_I \cap D_J$ for $J \subset I \subset \sI$, and
	\item $\pi_I^J: U_I \to U_I^J$ is equal to $\pi_J|_{UI}: U_I \to U_I \cap D_J$ as $U(1)^J$ bundles for $J \subset I \subset \sI$.
\end{enumerate}
\end{defn}

\begin{notn}
\begin{itemize}
	\item Given a $U(1)^{\{i\}}$ trivialization of $\pi_i$, let $(r_i,\theta_i)$ denote the polar coordinates on the fibre $\ID^{\{i\}}$.	\item Let $\rho_i: U_i \to \IR$ denote the map whose restriction to a $U(1)^{\{i\}}$ trivialization of $\pi_i$ is $\rho_i = r_i^2/2$.
	\item Let $\alpha_i$ denote the $1$-form on $U_i$ whose restriction to a $U(1)^{\{i\}}$ trivialization of $\pi_i$ is $\alpha_i = d\theta_i$.
	\item Define $N^\delta = \cup_i \rho_i^{-1}((-\infty,\delta])$, and $V_I = U_I \smallsetminus \cup_{j \in \sI \smallsetminus I} U_j$.
\end{itemize}
\end{notn}

By \cite[Lemma 5.3 and Lemma 5.14]{McLean_GrowthRate} or by \cite[Theorem 2.12]{tehrani2017normal}), for $\delta>0$ sufficiently small, there exists an open neighborhood $N(\cup_i D_i)$ of $\cup_i D_i$ in $M$ and a symplectomorphism
\[ \psi: (N(\cup_i D_i), \Omega) \to (N^\varepsilon,\Omega_0) \]
such that $(N^\epsilon,\Omega_0)$ admits a standard tubular neighborhood as in \cref{defn:StandardTubularNeighborhood}.  In particular, 
\begin{enumerate}
	\item $\Omega_0|_{V_I} = \Omega|_{D_I} + \sum_i d(\rho_i \cdot \alpha_i)$, and
	\item for all $x \in \cup_i D_i$, $\psi(x) = x$.
\end{enumerate}
After rescaling, we can (and do) assume that $\varepsilon = 1$.

Now our discussion diverges from the discussion in \cite[Section 6.2]{McLean_BirationalCalabiYauManifoldsHaveTheSameSmallQuantumProducts}.  We now move on to construct a cofinal sequence of Hamiltonians that are adapted to our standard tubular neighborhood.  To be precise, we fix the following.

\begin{notn}\label{notn:SlicesOfDivisorNhood}
Let $f:  [0,+\infty) \to \IR$ be the $C^1$-function given by
\[ f(r) = \begin{cases} -(1-r)^2/2, & r \leq 1 \\  0, &  r \geq 1. \\ \end{cases} \]

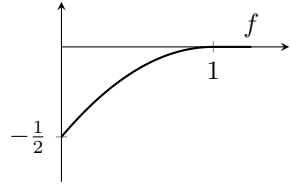
\begin{figure}
\begin{tikzpicture}[declare function={
     func(\x)= (\x<=1) * (-((1-\x)^2)/2)   +
      and(\x>1,\x<=2 ) * (0) ;}]
\begin{axis}[width=.33\linewidth, axis lines=center, xtick={1},xticklabels= {$1$}, xlabel={}, ytick={-1/2},yticklabels={$-\frac{1}{2}$}, xmin=0, xmax=1.5, ymin=-.75, ymax=.25]
  \addplot[color=black, domain=0:1.25, samples=100, smooth, thick]{func(x)} node[above,pos=1] {$f$};
\end{axis}
\end{tikzpicture}
\caption{Depiction of the function $f$ for \cref{notn:SlicesOfDivisorNhood}.}
\label{fig:Directednessgn}
\end{figure}

Let $H: N^1 \to \IR$ be given by $H|_{V_I} = \sum_i f(\rho_i)$.  Note, $dH|_{V_I} = \sum_{i \in I} f'(\rho_i) d \rho_i$.  So $H^{-1}(\delta)$ is a regular submanifold of codimension $1$ in $N^1$ for each $\delta < 0$.  For each $\delta > -1/2$, $N^\varepsilon \subset H^{-1}((-\infty,\delta])$ for some $\varepsilon>0$ sufficiently small.  Indeed, in $V_I \cap N^\varepsilon$, $H$ is bounded by $\sum_{i \in I} f(\varepsilon)$.  As $\varepsilon \to 0$, $\sum_{i \in I} f(\varepsilon)$ convergences to $-|I|/2 \leq -1/2 < \delta$.  So for $\varepsilon_I$ sufficiently small, $H|_{V_I \cap N^{\varepsilon_I}}< \delta$.  Take $\varepsilon \coloneqq \min_{I \subset \sI} \{\varepsilon_I\}$ to obtain the result.
\end{notn}

\begin{figure}
\centering
\includegraphics[width=.25\linewidth]{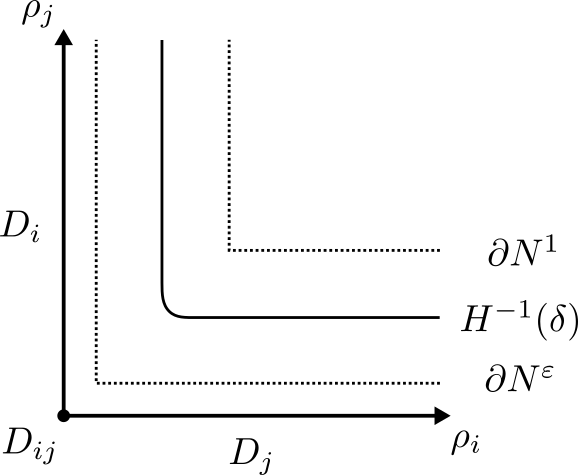}
\caption{A toric representation of our neighborhoods associated to a symplectic crossings divisor with two components}
\label{fig:DivisorNhoods}
\end{figure} 

\begin{lem}\label{rem:InitialCofinalProperties}
There exists a sequence of Hamiltonians $H_n: N^1 \to \IR$ that satisfy:
\begin{enumerate}
	\item If $y \in H^{-1}((\delta,0])$, then $\lim_{n \to \infty} H_n(y) = +\infty$.  
	\item If $x \in H^{-1}((-\infty,\delta])$, then $\lim_{n \to \infty} H_n(x) = 0$.
	\item $H_n(x) > H_m(x)$ for all $x$ and $n > m$.
\end{enumerate}
\end{lem}

\begin{con}\label{con:DivisorHamiltonians}
Let $\wt{\ell}_n: \IR \to \IR$ be the $C^2$-function given by
\[ \wt{\ell}_n(r) \coloneqq \begin{cases} 0, & r \leq \delta + 1/n \\ -6n^4 \left( \frac{(r-\delta-1/n)^3}{3} - \frac{(r-\delta-1/n)^2}{2n} \right), & \delta + 1/n \leq r \leq \delta +2/n\\ n, & r \geq \delta +2/n .\\ \end{cases}\]
Consider the $C^1$-function $\wt{H}_n \coloneqq \wt{\ell}_n \circ H$.  The $\Omega_0$-dual of $-d\wt{H}_n$ is
\[ X_{\wt{H}_n} = \sum_i \wt{\ell}_n'(H) \circ f'(\rho_i) \cdot \partial_{\theta_i}. \]
Since $f'(\rho_i)$ goes to zero as $\rho_i$ goes to $1$ and since $\wt{\ell}_n'(H)$ goes to zero as $H$ goes to either $\delta+1/n$ or $\delta+2/n$, the $1$-periodic orbits of $\wt{H}_n$ occur in the regions where $f$ and $\wt{\ell}_n(H)$ are smooth.  So we may perturb $f$ and $\wt{\ell}_n$ to smooth functions $f_n$, and $\ell_n$ such that they and the composition $H_n \coloneqq \ell_n \circ (\sum_i f_n(\rho_i)) - 1/n$ satisfy the following.  First, $f_n$ satisfies:\begin{enumerate}
	\item $|f-f_n|_{C^1} < \varepsilon_n$, and
	\item $f_n(r) \equiv 1$ for $r \geq 1$,
\end{enumerate}
where $\varepsilon_n>0$ is some small constant that depends on $n$.  In our estimates below, it suffices to take $\varepsilon_n = 1/n^5$.  Second, $\ell_n$ satisfies:
\begin{enumerate}
	\item $|\wt{\ell}_n - \ell_n|_{C^2} < \varepsilon_n$,
	\item $\ell_n(r) \equiv 0$ for $r \leq \delta+1/n$, and
	\item $\ell_n(r) \equiv n$ for $r \geq \delta+2/n$.
\end{enumerate}
Third $H_n$ satisfies:
\begin{enumerate}
	\item the collections of orbits and fixed points agree: $\Gamma_{H_n} = \Gamma_{\wt{H}_n}$ with $\Gamma_{H_n}(0) = \Gamma_{\wt{H}_n}(0)$, and
	\item in an open neighborhood of $\Gamma_{H_n}(0) = \Gamma_{\wt{H}_n}(0)$, the flows agree $\varphi_t^{\wt{H}_n} = \varphi_t^{{H_n}}$.
\end{enumerate}
\end{con}

We prove \cref{rem:InitialCofinalProperties}.

\begin{proof}
We handle each of the three statements individually.
\begin{enumerate}
	\item Suppose that $H(y) = \delta+\varepsilon$ for $\varepsilon>0$.  Then for $n\gg 0$,
	\[ \sum_i f_n(\rho_i(y)) > H(y) - \varepsilon_n \cdot |\sI| = \delta+\varepsilon - \varepsilon_n \cdot |\sI| > \delta + \frac{2}{n}. \]
	So
	\[ H_n(y) = \ell_n\left(\sum_i f_n(\rho_i(y))\right) - 1/n = n - 1/n.\]

	\item Suppose that $H(x) \leq \delta$.  Then
	\[ \sum_i f_n(\rho_i(x)) \leq \delta + \varepsilon_n \cdot |\sI| \leq \delta + 1/n.\]
	So
	\[ H_n(x) = \ell_n\left(\sum_i f_n(\rho_i(x))\right) - 1/n = -1/n.\]
	
	\item Suppose that $n>m$.  Set $F_n(x) = \sum_i f_n(\rho_i(x))$, and $F(x) = \sum_i f(\rho_i(x))$.  So $\left|F_n(x) - F(x)\right| \leq |\sI|  \cdot \varepsilon_n$.  Via the triangle inequality,
	\begin{align*}
	\left|\ell_n(F_n(x)) - \wt{\ell}_n(F(x)) \right| & \leq \left| \ell_n(F_n(x)) - \wt{\ell}_n(F_n(x)) \right| + \left| \wt{\ell}_n(F_n(x)) - \wt{\ell}_n(F(x)) \right|
	\\ & \leq \varepsilon_n + | \wt{\ell}_n'|_{\sup} \cdot |F_n(x) - F(x)|
	\\ & = \varepsilon_n + \frac{3n^2}{2} \cdot |F_n(x) - F(x)|
	\\ & \leq \varepsilon_n + \frac{3n^2}{2} \cdot |\sI| \cdot \varepsilon_n
	\\ & \leq \varepsilon_n \cdot (1+2n^2 \cdot |\sI|).
	\end{align*}
	Continuing to estimate coarsely,
	\begin{align*}
	\ell_n(F_n(x)) & - \ell_m(F_m(x)) 
	\\ & = \ell_n(F_n(x)) - \wt{\ell}_n(F(x))+\wt{\ell}_n(F(x)) - \wt{\ell}_m(F(x)) + \wt{\ell}_m(F_m(x)) - \ell_m(F_m(x))
	\\ & \geq - \varepsilon_n(1+2n^2  \cdot |\sI|) + \wt{\ell}_n(F(x)) - \wt{\ell}_m(F(x)) - \varepsilon_m(1+2m^2 \cdot|\sI|)
	\\ & \geq - \varepsilon_n \cdot (1+2n^2 \cdot |\sI|)-\varepsilon_m \cdot (1 + 2m^2 \cdot |\sI|)
	\\ & \geq - \varepsilon_n \cdot (4n^2 \cdot |\sI|) - \varepsilon_m \cdot (4m^2 \cdot |\sI|)
	\\ & \geq -8 |\sI| \cdot (1/m^3).
	\end{align*}
	So
	\begin{align*}
	H_n - H_m & = \ell_n(F_n(x)) - 1/n - \ell_m(F_m) + 1/m
	\\ &  \geq -8 |\sI| \cdot (1/m^3) + 1/m - 1/n
	\\ & \geq -8 |\sI| \cdot (1/m^3) + 1/m^2
	\\ & > 0
	\end{align*}
	with the last line holding when $m$ and $n$ are greater than $8 \cdot | \sI |$.
\end{enumerate}
This completes the proofs of all the items.
\end{proof}

We now study the orbits of the Hamiltonians $H_n$.  The associated Hamiltonian vector field of $H_n$ is
\[ X_{H_n} = \sum_i \ell_n'\left( \sum_i f_n(\rho_i) \right) \cdot f_n'(\rho_i) \cdot \partial_{\theta_i}. \]

\begin{notn}
\begin{enumerate}
	\item Let $\underline{d} = (d_1,\dots,d_{|\sI|}) \in \IZ^{|\sI|}_{\geq 0}$.
	\item Let $I(\underline{d}) = \{ i \in \sI \mid d_i \neq 0\}$.
	\item Let 
	\begin{align*}
	\Gamma^n_{\underline{d}}(0) 
	& = \left\{ x \in N^1 \mid \ell_n'\left(\sum_i f_n(\rho_i)\right) \cdot f_n'(\rho_i) = 2\pi \cdot d_i\right\} 
	\\ & = \left\{ x \in N^1 \mid \wt{\ell}_n'\left(\sum_i f(\rho_i)\right) \cdot f'(\rho_i) = 2\pi \cdot d_i\right\}.
	\end{align*}
	This is the fixed points of a collection of $1$-periodic orbits of $H_n$, which we denote by $\Gamma_{\underline{d}}^n$.  $\Gamma_{\underline{d}}^n$ lies completely in $V_{I(\underline{d})}$.  Geometrically, $\Gamma_{\underline{d}}^n$ is a $T^{|I(\underline{d})|}$ torus bundle over the $D_{I(\underline{d})} \smallsetminus \cup_{j \not \in I(\underline{d})} V_i$, which is a manifold with corners.  The flow $\varphi_t^{H_n}$ is given by rotation in the fibres.  As $n$ increases, the flow rotates the fibres by a greater and greater amount.
\end{enumerate}
\end{notn}

\begin{lem}\label{lem:CZEstimate1}
The (path components of) $\Gamma_{\underline{d}}^n(0)$ are pseudo Morse-Bott families of size
\[ \dim(\ker(d \varphi_1^{H_n} - \Ione)|_{\Gamma_{\underline{d}}^n(0)}) = 2n - |I(\underline{d})|.\]
Also, for each $\gamma \in \Gamma_{\underline{d}}^n$, there exists a capping $v$ in $N^1$ such that
\[ 2 \sum_i d_i - 2n \leq CZ(\gamma,v) \leq 2 \sum_i d_i + 2n.\]
\end{lem}

\begin{rem}
At this point in our construction, $c_1(N^1)$ does not necessarily vanish.  So $CZ(\Gamma_{\underline{d}}^n)$ is not well-defined, which is why \cref{lem:CZEstimate1} is stated in terms of cappings.
\end{rem}

\begin{proof}
Since the fixed points $\Gamma_{\underline{d}}^n(0)$ lie completely in $V_{I(\underline{d})}$, we work in $V_{I(\underline{d})}$.  Suppose that $I(\underline{d}) = \{1,\dots,m\}$.  In this region, $\varphi_t^{H_n}$ is given fibrewise over $\pi_{I(\underline{d})}$ by
\[ \left( \dots, \exp\left(\wt{\ell}_n'\left(\sum_i f(\rho_i)\right) \cdot f'(\rho_i) \cdot it\right) \cdot z_i,\dots\right) \]
where $(z_1,\dots,z_{|I(\underline{d})|})$ are the complex coordinates for the fibre.
The return map $d \varphi_t^{H_n}$ with respect to the basis $\partial_{\rho_i}$, $\partial_{\theta_i}$, and the horizontal distribution determined by $\pi_{I(\underline{d})}$ is
\[ \begin{pmatrix} \Ione & 0 & 0 \\ F_{ij}(t) & \Ione & 0 \\ 0 & 0 & \Ione \\ \end{pmatrix} \]
where
\[ F_{ij}(t) = t \cdot \left(\ell''\left(\sum_i f(\rho_i)\right) \cdot f'(\rho_j) \cdot f'(\rho_i) + \delta_{ij} \cdot \ell'\left(\sum_i f(\rho_i)\right) \cdot f''(\rho_i)\right).\]
To prove the first claim, we just need to show that the matrix $F_{ij}(1)$ is non-singular.  Notice that ${\ell}_n' \cdot f'(\rho_i) = 2\pi \cdot d_i$.  So $f'(\rho_i)/d_i = f'(\rho_j)/ d_j$ for all $i$ and $j$.  Since $f'$ is linear near our orbits, $\rho_i = a_{ij} \cdot \rho_j$ for some $a_{ij} \in \IQ$.  If $\det(F_{ij}(1)) =0$, then the values of the $\rho_i$ are algebraic for all $i$, which implies that $\pi$ is algebraic, a contradiction.

To estimate the Conley-Zehnder index, notice that if $\gamma$ is a $1$-periodic orbit in $\Gamma_{\underline{d}}^n$, then $\gamma$ is completely contained in a fibre $\pi_{I(\underline{d})}^{-1}(p)$ for some $p \in D_{I(\underline{d})}$.  So there exists a map $v: \ID \to \pi_{I(\underline{d})}^{-1}(p)$ such that $v(\exp(2\pi it)) = \gamma(t)$.  To show that $CZ(\gamma,v)$ satisfies the statement of the lemma, one argues as in \cite[Proof of Proposition 6.17 following equation (15)]{McLean_BirationalCalabiYauManifoldsHaveTheSameSmallQuantumProducts}.
\end{proof}

Now we consider $\wt{\pi}: \wt{M} \to \IC$ from the beginning of this section.  Let $\beta: \wt{M} \to M$ denote the resolution map.  For convenience, set $U = M \smallsetminus \pi^{-1}(0)$, and $\wt{U} = \wt{M} \smallsetminus \wt{\pi}^{-1}(0)$.  So $\beta$ is an isomorphism from $\wt{U}$ to $U$ and $\wt{M} \smallsetminus \wt{U}$ is a collection of transversally intersecting complex hypersurfaces $D_1,\dots,D_n$ that comprise our normal crossings divisor.  Suppose $c_1(M)=0$.

Let $\tau: \sK^*_{M} \to M \times \IC$ be a trivialization of the anti-canonical bundle of $M$, and let $\wt{\tau} = \tau \circ \beta$ be the induced trivialization of the anti-canonical bundle of $\wt{U}$.  Fix a generic section $s$ of $\sK^*_{\wt{M}}$ that satisfies $s(x) \coloneqq \wt{\tau}^{-1}(x,1)$ for all $x$ outside of a compact subset that contains an open neighborhood of $\cup_i D_i$.

$\wt{M}$ smoothly deformation retracts onto $\cup_i D_i$.  So $H_{2n-2}(\wt{M};\IZ) \cong \oplus_i [D_i]$.  Write $[s^{-1}(0)] = -\sum_i a_i \cdot [D_i]$.  These terms are the \emph{discrepancy} of the resolution $\beta$.  By \cite[Lemma 6.15]{McLean_BirationalCalabiYauManifoldsHaveTheSameSmallQuantumProducts}, each $a_i$ is non-negative.  (\cite{McLean_BirationalCalabiYauManifoldsHaveTheSameSmallQuantumProducts} works in the setting where $U$ and $\wt{U}$ are affine; however, this assumption is not used in the presented argument.)  Similarly, consider the $1$-form $d\theta$ that is defined on $\IC^\times$.  Consider a smooth bump function $\varphi$ that is one near the divisors and zero outside of a compact subset that contains an open neighborhood of the divisor.  Then $d(\varphi \cdot d \theta) \in H^2_{c}(\wt{M};\IZ)$.  So by Poincar\'{e} duality, it may be written as $\sum_i w_i \cdot [D_i]$.

The $H_n: \wt{M} \to \IR$ in \cref{con:DivisorHamiltonians} descend to Hamiltonians on $M$ since they are constant near the divisors and since $\beta$ is an isomorphism over the complement of $\pi^{-1}(0)$.  Abusively write $H_n: M \to \IR$ for the descended Hamiltonian.  Using the discrepancy and \cref{CZ2,CZ3,CZ5}, we obtain the following.

\begin{lem}\label{lem:CZEstimate2}
For each $\gamma \in \Gamma_{\underline{d}}^n$ with capping $v$ as in \cref{lem:CZEstimate1}, we have
\[ CZ(\gamma) = CZ(\gamma,v) + 2 \sum_i d_i a_i\]
where the left most Conley-Zehnder index is computed in $M$.  Consequently,
\[ 2\sum_i d_i(a_i+1) - 2n \leq CZ(\gamma) \leq 2 \sum_i d_i (a_i+1) +2n .\]\qed
\end{lem}

As a corollary of \cref{lem:CZEstimate2} and our construction, we have the following.

\begin{lem}\label{lem:CZEstimate2}
There exists perturbations $H'_n$ as in \cref{lem:PerturbWithCZ} of our Hamiltonians $H_n$ on $M$ with the following property: for each $n \in \IZ$, there exists a constant $C_n$ such that if $\gamma$ is a $1$-periodic orbit of any $H_n'$ and $CZ(\gamma) \leq n$, then
\[ \int \gamma^*(d\theta) \leq C_n. \]
\end{lem}

\begin{proof}
We begin with a warm-up.  Suppose that we do not perturb $H_n$.  In this case, an orbit in $\Gamma^n_{\underline{d}}$ wraps $\sum_i d_iw_i$ around the origin when projected to $\IC$, that is,
\[ \int \gamma^* d\theta = \sum_i d_i w_i. \]
By \cref{lem:CZEstimate2}, since the $a_i$ are all non-negative, any bound on $CZ(\gamma)$ implies a bound on the possible values of each $d_i$, and, thus, a bound on $\int \gamma^* d\theta$.  This proves the result when we do not perturb our Hamiltonians.

To prove the result for the perturbed Hamiltonians, it is easiest to fix a very explicit perturbation.  In particular, we perturb each $H_n$ about the collections of orbits $\Gamma^n_{\underline{d}}$, which are torus bundles over manifolds with corners.  First, one perturbs the Hamiltonian by adding a $C^2$-small Morse function near the manifold with corners along $D_{I(\underline{d})}$, in such a manner that the Hamiltonian vector field has no orbits point near the corners.  After this perturbation, the orbits of $\Gamma^n_{\underline{d}}$ break up as torus families worth of orbits, with a torus family over each critical point of the added Morse function.  Finally, one further perturbs these torus families to break them up into $2|I(\underline{d})|$ individual orbits.  This is analogous to the types of perturbations we discussed in \cref{rem:MorseBottBreakings}.  Each resulting orbit of these perturbed Hamiltonians is $C^0$ close to an orbit from the associated unperturbed Hamiltonian.  This perturbation scheme is carried out in \cite[Proof of Lemma 6.8.]{McLean_GrowthRate}.  In particular, the proof there works equally as well for our setting.  From this, \cref{lem:PerturbWithCZ}, and the warm-up above, we conclude the lemma.
\end{proof}


\section{Proof of \cref{lem:ModifiedParallelTransportCY}}

We prove \cref{lem:ModifiedParallelTransportCY}.  We assume the notation from the setup of \cref{sec:ProofOfTheorem}.

\begin{proof}
Let $\Omega_P$ be any integral K\"{a}hler form on $P$.  Let $\Omega_z \coloneqq \Omega_P|_{P_z}$.  For each $z \in \IC^\times$, we have a symplectic parallel transport map from $(P_z, \Omega_z)$ to $(P_1,\Omega_1)$.  By \cite[Lemma 5.16]{McLean_GrowthRate}, we may pre-compose this parallel transport map with a symplectomorphism such that the composition $\varphi_z: (P_z, \Omega_z) \to (P_1,\Omega_1)$ maps $h_z^{-1}(\infty)$ to $h_1^{-1}(\infty)$ and $h_z^{-1}(0)$ to $h_1^{-1}(0)$.  We may identify $(P_1,\Omega_1)$ and $h_1$ with $(\overline{M},\Omega_1)$ and $\pi$ respectively.  Applying \cref{prop:KahlerComplementEmbedding} to $\pi^{-1}(\infty)$ in $\overline{M}$, there exists a K\"{a}hler form $\Omega'$ on $M$ and a symplectic embedding
\[ \psi: (M,\Omega_1) \hookrightarrow (M,\Omega') \]
that is the identity near $\pi^{-1}(0)$ and maps $\psi(M)$ into a compact subset of $M$, which we can choose to be $M_{r_1}$.

Let $\beta: \wt{M} \to M$ and $\wt{\pi}: \wt{M} \to \IC$ denote the resolution data of $\pi$ as in \cref{sec:DivisorModels}.  As in \cref{sec:DivisorModels}, after deforming a small open neighborhood of the divisor $\wt{\pi}^{-1}(0)$ in $\wt{M}$, we obtain a compact, codimension zero submanifold of $\wt{M}$, say $\wt{K}$, such that $\wt{\pi}^{-1}(0) \subset \wt{K}$, and $\wt{K} \subset (\wt{M})_{r_1}$.  Here $\wt{K}$ is the space $H^{-1}((-\infty,\delta])$ given in \cref{notn:SlicesOfDivisorNhood}.  So $K \coloneqq \beta(\wt{K})$ is 
a compact, codimension zero submanifold of ${M}$ with ${\pi}^{-1}(0) \subset {K}$, and ${K} \subset {M}_{r_1}$.  By \cref{lem:ExtendingOpenMaps}, there exists $r_0>0$ such that $r_0 < r(\partial K) < r_1$.  By \cref{lem:DisplacingFibres}, after shrinking $K$, it is stably displaceable inside of $(M_{r_1},\Omega')$.

We obtain a sequence of Hamiltonians $H_n': M \to \IR$ from \cref{lem:CZEstimate2} (where we have extended by constancy), that satisfy $H_n' < H_{n+1}'$ and $H_n' \to (0)_K$.  Without loss of generality, assume that $H_n' \equiv -1/n$ for $r \leq r_0$.

Since $\Omega'$ is K\"{a}hler, it is compatible with the almost complex structure that makes the projection $\pi: M \to \IC$ holomorphic.  So the fibres of $\pi$ are symplectic submanifolds with respect to $\Omega'$, and $(M,\pi, \Omega')$ defines a Hamiltonian fibration over $\IC^\times$.  By \cref{prop:FlatteningOverCStarNonDeg} with the constants $a,b,c$ in the statement being chosen such that $0<a< b< c\ll r_0$, we obtain a symplectic form ${\Omega''}$ on $M$, and a symplectic embedding $\mu: (M,{\Omega'}) \hookrightarrow (M,{\Omega''})$.
We may push-forward (and extend by constancy) the $H_n'$ to $(M,\Omega'')$ via $H_n'' \coloneqq (\mu^{-1})^*H_n'$.  The conclusions of \cref{lem:CZEstimate2} still hold for these $H_n''$ since $\mu$ is a smooth isotopy that is the identity near $\pi^{-1}(0)$.

After relabeling, by our application of \cref{prop:FlatteningOverCStarNonDeg} above, we can fix $0 < r_0 < r_1 < r_2 < r_3$ such that
\begin{itemize}
	\item $\mu \circ \psi(M)$ and $\mu(K)$ are both properly contained in the interior of $M_{r_3}$,
	\item $M \smallsetminus M_{r_0}$ is symplectomorphic to a symplectic mapping cylinder, whose associated Reeb vector field is non-degenerate,
	\item $r_2 < r(\partial\mu(K)) < r_3$, and
	\item $H_n'' \equiv -1/n$ for $r \leq r_2$.
\end{itemize}

The $H_n''$ are degenerate Hamiltonians with families of constant orbits $\{ H_n'' \equiv -1/n\}$ and $\{H_n'' \equiv n\}$.  Both of these are pseudo Morse-Bott families as in \cref{defn:PseudoMorseBott}.  Moreover, by \cref{CZ1}, the Conley-Zehnder indices of these constant orbits are zero.  So by \cref{lem:PerturbWithCZ} and \cref{lem:CZEstimate2}, we may pick $C^2$-small perturbation of the $H_n''$, say to obtain $H_n'''$, such that
\begin{itemize}
	\item for $r_1 \leq r \leq r_2$, $H_n'''$ is a linear function in $r$ (and thus $\varphi$-scalable as in \cref{defn:PhiScalable}),
	\item near $r = r_3$, $H_n'''$ is a linear function in $r$ with fixed slope $\tau$ that is independent of $n$, 
	\item the $H_n'''$ satisfy the hypotheses of \cref{lem:SHRescalingIso} (with $a = r_1$, $b = r_2$, and $c = r_3$),
	\item $H_n''' < H_{n+1}'''$, and
	\item the sequence $H_n'''$ is cofinal in $\sH^+(\mu(K) \subset M_{r_3}, \Omega''; 0,\tau)$.
\end{itemize}

Since $K$ was stably displaceable, $\mu(K)$ is stably displaceable inside of $M_{r_3}$.  So by \cref{lem:SHRescalingIso} and \cref{prop:PropertiesOfSH} \cref{prop:VanishingForStDisp},
\[ 0 = \widehat{SH}(\mu(K) \subset M_{r_3};0,\tau) \otimes \Lambda \cong \widehat{SH}((\mu(K))_\delta \subset M_{r_3+\delta}; 0,\tau) \otimes \Lambda\]
for all $\delta>0$.  For $\delta\gg 0$ sufficiently large, $\mu \circ \psi(M)$ is contained in $M_{r_3} \subset M_{r_3+1} \subset (\mu(K))_{\delta+1} \subset M_{r_3+1+\delta}$.  Now by \cite[Proposition 2.5]{tonkonog2020superrigidity} and \cref{prop:PropertiesOfSH} \cref{prop:InvarianceOfSlope},
\[ \widehat{SH}(M_{r_3+1} \subset M_{r_3+1+\delta};0) \otimes \Lambda \cong 0. \]
Setting $\psi_z = \mu \circ \psi \circ \varphi_z$, $a = r_3$, $b = r_3+1$, and $c = r_3+1+\delta$, we almost have completed the proof of the lemma.  It just remains to explain why $\Omega$ is integral.  To see this, one can either notice that $\Omega$ is cohomologous to the original K\"{a}hler form, which is integral (as all of the above embeddings are produced via Moser's argument), or that $\psi_z$ preserves the K\"{a}hler form near $\pi^{-1}(0)$ and so any sphere in $M$ is homotopy equivalent to a sphere in $\pi^{-1}(0)$, and thus, its integral of $\Omega$ will be an integer.
\end{proof}

\begin{rem}
Notice that in the proof above, the Reeb orbits of the stable Hamiltonian structure arising from the symplectic mapping cylinder like end do not need to satisfy any index bounded assumption.  Consequently, if we tried to construct our isomorphism by using radial Hamiltonians whose slopes go to infinity as $n$ increases, then a priori, we would not know if we could apply \cref{lem:SHRescalingIso}.  It was the divisor model that explicitly gave us the index bounded assumption.  Typically, one has a neighborhood of a divisor whose boundary is a contact manifold and one deforms the contact form to have the nice dynamical properties that our Hamiltonians have above.  In our case, there is no contact form to be found.  So we take the novel approach of working directly with Hamiltonians that are adapted to divisors.
\end{rem}


\part{Hamiltonian fibrations and symplectic deformations}\label{part:SymplecticDeformations}

In this part, we establish requisite symplectic embedding/deformation results that are used to construct the ``modified parallel transport maps" in the proof of our main result, see \cref{lem:ModifiedParallelTransport} and \cref{lem:ModifiedParallelTransportCY}.  We also discuss Hamiltonian fibrations and their relationships with symplectic mapping cylinders and convex symplectic domains.


\section{Symplectic self-embeddings of divisor complements}\label{sec:ComplementEmbeddings}

We discuss a K\"{a}hler embedding result, which allows one to symplectically push divisor complements away from divisors via a symplectic deformation that lies in a fixed K\"{a}hler class.

\begin{prop}\label{prop:KahlerComplementEmbedding}
Let $(X,\Omega,J)$ be a K\"{a}hler manifold and let $D$ be an effective normal crossings divisor in $X$.  There exists a K\"{a}hler form $\widetilde{\Omega}$ for $(X,J)$ and a symplectic embedding
\[ \psi: (X\smallsetminus D, \Omega|_{X \smallsetminus D}) \hookrightarrow (X\smallsetminus D, \widetilde{\Omega}|_{X \smallsetminus D}) \]
that satisfy:
\begin{enumerate}
	\item $\psi(X \smallsetminus D) \subset X \smallsetminus N(D)$ for an open neighborhood $N(D)$ of $D$ in $X$, and
	\item $\psi$ can be made the identity away from any fixed open neighborhood of $D$.
\end{enumerate}
\end{prop}

We begin with a local model for the embedding in \cref{subsec:LocalKahlerEmb}.  Using this local model, we prove \cref{prop:KahlerComplementEmbedding} in \cref{subsec:ProofKahlerEmb}.

\subsection{Computing in a local model}\label{subsec:LocalKahlerEmb}

Here we understand the local model for the symplectic embedding introduced above.

\begin{notn}
\begin{enumerate}
	\item Let $\Omega_{std}$ denote the standard symplectic form on $\IC^n$ with respect to the holomorphic coordinates $z_1,\dots,z_n$,
	\[ \Omega_{std} \coloneqq \sum_i \frac{i}{2} dz_i \wedge d \overline{z_i} = \sum_i r_i d r_i \wedge d \theta_i,\]
	where $z_i = r_i e^{\theta_i}$ in polar coordinates.  
	\item Let $\sigma: \IC^n \to \IC$ be given by
	\[ \sigma(z_1,\dots,z_n) = z_1^{\alpha_1} \cdots z_n^{\alpha_n},\]
	where $\alpha = (\alpha_1,\dots,\alpha_n)$ is a tuple of non-negative integers.  
	\item Let $D \coloneqq \sigma^{-1}(0)$ be the associated normal crossings divisor.
	\item Let $V_{std}$ denote the vector field on $\IC^n \smallsetminus D$ given by
	\[ \Omega_{std}( V_{std},\cdot) = \sum_i \alpha_i d \theta_i = -d^c \log|\sigma|(\cdot),\]
	where $|\cdot|$ is the standard metric on $\IC$.  Equivalently,
	\[ \Omega_{std}(V_{std},J \cdot) = \sum_i \frac{\alpha_i}{r_i} d r_i  = \sum_i d(\log(r_i^{\alpha_i})) = d(\log|\sigma|)(\cdot).\]
	Explicitly,
	\[ V_{std} = \sum_i \frac{\alpha_i}{r_i} \cdot \partial_{r_i}. \]
	\item Let $\rho: \IC^n \to \IR$ be a smooth function whose gradient vector field with respect to $\Omega_{std}( \cdot, J \cdot)$ is denoted by $\nabla \rho$,
	\[ \Omega_{std}(\nabla \rho, J \cdot) = d \rho(\cdot).\]
	\item For $0 \leq s \leq 1$, let $\Omega_s$ be a family of K\"{a}hler forms on $\IC^n$.
	\item Let $V_s$ denote the family of vector fields on $\IC^n \smallsetminus D$ given by
	\[ \Omega_s(V_s,J\cdot) = d(\log|\sigma|)(\cdot). \]
	\item Let $(\nabla \rho)_s$ denote the family of vector fields on $\IC^n$ given by
	\[ \Omega_s((\nabla \rho)_s,J \cdot) = d \rho(\cdot) .\]
\end{enumerate}
\end{notn}

We want the following estimate.

\begin{lem}\label{lem:LocalKahlerEmb3}
There exists $\varepsilon>0$ and $a>0$ so that
\[ d(\log|\sigma| + \rho)(V_s + (\nabla \rho)_s) \geq a \cdot (\log|\sigma|+\rho)^2\]
on $ \ID_\varepsilon^n \subset \IC^n$.
\end{lem}

Our strategy is to relate $V_{s}$ to $V_{std}$ via simple matrix algebra and bootstrap off of the analogue of \cref{lem:LocalKahlerEmb3} for $V_{std}$, which follows from a computation:

\begin{lem}\label{lem:LocalKahlerEmb1}
There exists $\varepsilon>0$ so that
\[ d(\log|\sigma|+\rho)(V_{std} + \nabla \rho) \geq (\log|\sigma|+\rho)^2\]
on $\ID_\varepsilon^n \subset \IC^n$.
\end{lem}

We need the following estimate.

\begin{lem}\label{lem:LocalKahlerEmb2}
Given  non-negative constants $b_1$, $b_2$, $c_1$, and $c_2$, there exists $\varepsilon>0$ such that if $0 \leq x \leq \varepsilon$ and $0 \leq y \leq \varepsilon$, then
\[ \log(c_1\cdot x) \cdot \log(c_2 \cdot y) \leq \frac{b_1}{x^2} + \frac{b_2}{y^2}.\]
\end{lem}

\begin{proof}
Without loss of generality, we may assume that
\[ 0 \leq c_1 \cdot x \leq c_2 \cdot y \leq 1.\]
By L'Hopital's rule, there exists $\varepsilon >0$ such that if $c_1 \cdot x \leq \varepsilon$, then
\[ \log(c_1 \cdot x)^2 \leq \frac{b_1}{x^2}. \]
So
\[ \log(c_1 \cdot x) \log(c_2 \cdot y) \leq \log(c_1 \cdot x_1)^2 \leq \frac{b_1}{x^2} \leq \frac{b_1}{x^2} + \frac{b_2}{y^2}, \]
as desired.
\end{proof}

We prove \cref{lem:LocalKahlerEmb1}.

\begin{proof}
\begin{align*}
d(\log|\sigma| + \rho) (V_{std}+\nabla \rho) & = \left( \frac{d |\sigma|}{|\sigma|} (V_{std}) \ + \frac{d |\sigma|}{|\sigma|}(\nabla \rho) + d\rho(V_{std}) + d \rho(\nabla \rho)	 \right)
\\ & = \sum_{i} \frac{\alpha_i^2}{r_i^2} + 2 \sum_{i} \frac{\alpha_i}{r_i} d r_i(\nabla \rho) + | \nabla \rho|^2.
\end{align*}
For each $\varepsilon >0$, $dr_i(\nabla \rho)$ and $|\nabla \rho|^2$ are bounded on $\ID_\varepsilon^n \subset \IC^n$.  So for $\varepsilon>0$ sufficiently small, there exists $c>0$ such that
\[ d(\log|\sigma| + \rho) (V_{std} + \nabla \rho) \geq c \cdot \sum_i \frac{\alpha_i^2}{r_i^2}\]
on $\ID_\varepsilon^n$.  Since
\[ \log(e^\rho \cdot |\sigma|) = \log|\sigma| + \rho,\]
\cref{lem:LocalKahlerEmb2} implies that (after shrinking $\varepsilon$)
\[ d (\log|\sigma| + \rho)(V_{std} + \nabla \rho) \geq c \cdot \sum_i \frac{\alpha_i^2}{r_i^2}\geq ( \log|\sigma| + \rho)^2\]
on $\ID_\varepsilon^n$, as desired.
\end{proof}

We prove \cref{lem:LocalKahlerEmb3}.

\begin{proof}
Relate $V_{s}$ to $V_{std}$ via some simple matrix algebra:  Consider the standard basis of $T \IC^n \cong T \IR^{2n}$ and use the standard metric $\Omega_{std}(\cdot, J \cdot)$ to extend it to a dual basis of $T^* \IC^n \cong T ^* \IR^{2n}$.  With respect to these bases, we may write
\begin{enumerate}
	\item $d( \log|\sigma| + \rho) = \beta$, a covector, and
	\item $\Omega_s(\cdot, J \cdot) = A_s$, a family of symmetric, positive definite matrices.
\end{enumerate}
Note that $A_s$ and $\beta$ both vary with respect to points in $\IC^n$.  With these identifications,
\[ V_{std} + (\nabla \rho) = \beta^T\]
and
\[ V_s + (\nabla \rho)_s = A_s^{-1} \cdot \beta^T.\]
Since $A_s^{-1}$ is also symmetric and positive definite,
\[v \cdot A_s^{-1}\cdot v^T \geq a \cdot v \cdot v^T\] 
for all $v \in T^* \ID_1^n$ and some $a>0$ (which is independent of $v$).  So
\begin{align*}
d(\log|\sigma| + \rho)(V_s + (\nabla \rho)_s) & = (V_s+(\nabla \rho)_s)^T \cdot A_s \cdot (V_s + (\nabla \rho)_s) 
\\ & = \beta \cdot (A_s^{-1})^T \cdot A_s \cdot A_s^{-1} \cdot \beta^T 
\\ & = \beta \cdot (A_s^{-1})^T \cdot \beta^T 
\\ & \geq a \cdot \beta \cdot \beta^T 
\\ & = a \cdot d(\log|\sigma|+\rho)(V_{std}+\nabla \rho)
\end{align*}
on $\ID_1^n$.  The lemma follows from this inequality and \cref{lem:LocalKahlerEmb1}.
\end{proof}

\subsection{Proof of proposition}\label{subsec:ProofKahlerEmb}
Here we apply the local computations from \cref{subsec:LocalKahlerEmb} to prove \cref{prop:KahlerComplementEmbedding}.

\begin{proof}
Consider the line bundle $\sO(D)$, a Hermitian metric $\|\cdot\|$ on $\sO(D)$, and a holomorphic section $\sigma$ with $\sigma^{-1}(0) = D$.  Fix a smooth function $\ell: \IR_{\geq 0} \to \IR_{\geq 0}$ that satisfies:
\begin{enumerate}
	\item $\ell(r) = r$ for $0 \leq r \leq \delta/2$,
	\item $\ell(r)$ is constant for $\delta \leq r$, and
	\item $\ell'(r) \geq 0$.
\end{enumerate}
Consider $f: X \smallsetminus D \to \IR$ given by $f = \log(\ell(\|\sigma\|))$.  This gives rise to a $(1,1)$-form on $X$, $dd^c f$.  Fix $c>0$ sufficiently small so that
\[ \widetilde{\Omega} \coloneqq \Omega + c \cdot dd^cf\]
is non-degenerate (and consequently K\"{a}hler).

$d^c f$ is not globally defined.  So $dd^cf$ is not globally exact on $X$ and $\widetilde{\Omega}$ is not cohomologous to $\Omega$.  However, on the complement of $D$, $d^c f$ is well-defined and $dd^c f$ is exact.  So we have a family of cohomologous K\"{a}hler forms on $X \smallsetminus D$ given by
\[ \Omega_s \coloneqq \Omega + s \cdot c \cdot dd^c f.\]
We would like to use Moser's argument, to determine a symplectic embedding
\[ \psi: (X\smallsetminus D,  \Omega) \hookrightarrow (X \smallsetminus D, \widetilde{\Omega}) \]
obtained by taking the time $1$ flow along $V_s$ determined by
\[ \Omega_s(V_s,\cdot) = -c \cdot d^c f.\]
The goal of the rest of this proof is to bound $V_s$ in order to show that its time $s$ flow, $\psi_s$, is well-defined and satisfies the desired properties.

We claim there exists a neighborhood $N(D)$ of $D$ in $X$ and $a>0$ such that
\[ d(\log\|\sigma\|)(V_s) \geq a \cdot (\log\|\sigma\|)^2 \]
on $N(D)$.  To see this, fix $p \in D$ and local holomorphic coordinates $z_1,\dots,z_n$ such that
\[ \sigma = z_1^{\alpha_1}\cdots z_n^{\alpha_n}.\]
In this trivialization, the Hermitian metric $\|\cdot\|$ may be written as $e^\rho\cdot|\cdot|$ where $|\cdot|$ is the standard metric on $\IC$.  So sufficiently close to $p$,
\[ \log(\ell(\|\sigma\|)) = \log \|\sigma\| = \log|\sigma|+\rho. \]
Moreover, in this local chart, $V_s$ is determined by
\[ \Omega_s(V_s,J \cdot) = c \cdot d(\log|\sigma|+\rho).\]
So our setup in this local chart (or rather a rescaling of it by $c$) agrees with our local model from \cref{subsec:LocalKahlerEmb}.  By \cref{lem:LocalKahlerEmb3}, we have that there exists $a_p>0$ and a neighborhood $N(p)$ of $p$ in $X$ such that
\[ d(\log\|\sigma\|)(V_s) \geq  a_p \cdot (\log\|\sigma\|)^2 \]
on $N(p)$.  Since $D$ is compact, we may find a neighborhood $N(D)$ of $D$ in $X$ and a single $a>0$ such that
\[ d(\log\|\sigma\|)(V_s) \geq a \cdot (\log\|\sigma\|)^2 \]
on $N(D)$, as desired.

We now show that the flow $\psi_s$ of $V_s$ is well-defined for all $s$ and address the claims in the proposition.  Let $\gamma$ be a flow line of $V_s$.  Fix $1/2> \varepsilon > 0$ such that $\|\sigma\|^{-1}([0,2\varepsilon]) \subset N(D)$.  By the above inequality,
\[ \frac{d}{ds} \left( \log\|\sigma\| \circ \gamma(s) \right) = d(\log\|\sigma\|)_{\gamma(s)} \circ (V_s)_{\gamma(s)} \geq a \cdot (\log\|\sigma\|)^2 \circ \gamma(s)\]
when $\|\sigma \| \circ \gamma(s) \leq 2 \varepsilon$.  So when $\|\sigma \| \circ \gamma(s_0) \leq 2 \varepsilon$, we may solve this differential inequality to get that
\[ \log \|\sigma\| \circ \gamma(s) \geq \frac{\log \|\sigma\| \circ \gamma(s_0)}{1 - a \cdot (s-s_0) \cdot \log \|\sigma\| \circ \gamma(s_0)} \]
when $\|\sigma\| \circ \gamma(s) \leq 2 \varepsilon$ and $s \geq s_0$.  From this, we see when $\|\sigma\| \circ \gamma(s_0) \leq 2\varepsilon$ that
\begin{enumerate}
	\item $\|\sigma \| \circ \gamma(s) \geq \|\sigma\| \circ \gamma(s_0)$, and
	\item $\|\sigma \| \circ \gamma(s) \geq \exp\left( \frac{-1}{a \cdot (s-s_0)} \right)$
\end{enumerate}
for $s_0 \leq s \leq s_1$, where
\[ s_1 = \min\{s \geq s_0 \mid \gamma([s_0,s]) \subset \|\sigma\|^{-1}([0,2\varepsilon]) \}. \]

To show that $\psi_s$ is well-defined, it suffices to show that, first, no flow line runs into $D$ in finite time, and, second, $\psi_s$ is the identity outside of a compact neighborhood of $D$.  To see that no flow line runs into $D$ in finite time, suppose by way of contradiction that $\gamma$ is not bounded away from $D$ for all $s$.  This implies that there exists an increasing sequence $s_\nu$ such that 
\begin{enumerate}
	\item $\|\sigma\| \circ \gamma(s_\nu) \leq 2 \varepsilon$, and
	\item $\lim_{\nu \to \infty} \gamma(s_\nu) \in D$.  
\end{enumerate}
The second item implies that
\[\lim_{\nu \to \infty} \|\sigma\| \circ \gamma(s_\nu) = 0.\]
The first item combined with the fact that
\[\|\sigma \| \circ \gamma(s_\nu) \geq \|\sigma\| \circ \gamma(s_0)\]
implies that
\[\lim_{\nu \to \infty} \|\sigma\| \circ \gamma(s_\nu)>0,\]
a contradiction.  To see that $\psi_s$ is the identity outside of a compact neighborhood of $D$ notice that $\ell$ is constant on $\|\sigma\|^{-1}([\delta,+\infty)$.  So $d^cf = 0$ on $\|\sigma\|^{-1}([\delta,\infty))$.  This implies that $V_s = 0$ on $\|\sigma\|^{-1}([\delta,\infty))$ and thus that $\psi_s$ is the identity on $\|\sigma\|^{-1}([\delta,\infty))$.

To show that $\psi(X \smallsetminus D)$ is contained in the complement of an open neighborhood of $D$ in $X$, it suffices to show that $\|\sigma\|^{-1} \circ \gamma(1) > \varepsilon$.  The above work implies that if $\|\sigma\| \circ \gamma(s_0) > \varepsilon$, then $\|\sigma\| \circ \gamma(s) > \varepsilon$ for all $s \geq s_0$.  So it suffices to show that if $\|\sigma\| \circ \gamma(0) \leq \varepsilon$, then for some $s_0$, $\|\sigma\| \circ \gamma(s_0) > \varepsilon$.  However, when $\varepsilon < \exp(-1/a)$, the above work implies that such an $s_0$ must exist.  So after shrinking $\varepsilon$, we obtain the desired inclusion.

Finally, to see that we can make the support of $\psi$ be contained in any arbitrarily small neighborhood of $D$, notice that (when we showed that $\psi_s$ was well-defined) we showed that $\psi_s$ is the identity on $\|\sigma\|^{-1}([\delta,\infty))$.  So taking $\delta$ to be arbitrarily small forces the support of $\psi$ to be contained in any arbitrarily small neighborhood of $D$, as desired.
\end{proof}


\section{Hamiltonian Fibrations}\label{sec:HamFibs}

\subsection{Definitions}\label{subsec:HamFibs_Definitions}

We give basic definitions that pertain to Hamiltonian fibrations.

\begin{defn}\label{defn:HamiltonianFibration}
Let $(F,\omega_F)$ be a closed symplectic manifold.  A \emph{Hamiltonian fibration} over (a possibly open) Riemann surface $\Sigma$ with fibre $(F,\omega_F)$ is a tuple $(M,\pi,\Omega)$ where
\begin{enumerate}
	\item $M$ is a manifold,
	\item $\pi: M \to \Sigma$ is a proper submersion, and
	\item $\Omega$ is a closed $2$-form on $M$ whose restriction to each fibre $F_z = \pi^{-1}(z)$ is symplectic and agrees with $\omega_F$, that is, $(F_z,\Omega|_{F_z}) \cong (F,\omega_F)$.
\end{enumerate}
$\Omega$ is the \emph{coupling form}.  If the coupling form $\Omega$ is non-degenerate, then $(M,\pi,\Omega)$ is called a \emph{non-degenerate Hamiltonian fibration}.
\end{defn}

\begin{defn}\label{defn:VerticalHorizontalDistributions}
The \emph{vertical distribution} of a Hamiltonian fibration $(M,\pi,\Omega)$ is
\[ \Ver_p = \ker(d\pi(p)). \]
If a vector field is contained in $\Ver$, then it is called \emph{vertical}.  The \emph{horizontal distribution} of a Hamiltonian fibration $(M,\pi,\Omega)$ is
\[ \Hor_p = \left\{ v \in T_pM \mid \Omega(v,\cdot)|_{\Ver_p} \equiv 0 \right\}. \]
If a vector field is contained in $\Hor$, then it is called \emph{horizontal}.
\end{defn}

There is a splitting $T_pM \cong \Hor_p \oplus \Ver_p.$

\begin{defn}\label{defn:HorizontalLift}
The \emph{horizontal lift} of a vector field $W$ on $\Sigma$ is the horizontal vector field $\widetilde{W}$ that satisfies
\[d\pi(p) \circ \widetilde{W}(p) = W(\pi(p)).\]
\end{defn}

\begin{defn}\label{defn:Curvature}
The \emph{curvature} of a Hamiltonian fibration $(M,\pi,\Omega)$ with respect to the symplectic form $\omega_{\Sigma}$ on $\Sigma$ is the function $R: M \to \IR$ determined by
\[R(p) \cdot \omega_{\Sigma}(v,w) = \Omega(\widetilde{v},\widetilde{w}), \]
where $v$ and $w$ are vectors in $T_{\pi(p)}\Sigma$ and $\wt{v}$ and $\wt{w}$ are their horizontal lifts in $T_pM$ respectively.
\end{defn}

A Hamiltonian fibration is non-degenerate if and only if the curvature with respect to any symplectic form on the base $\Sigma$ is everywhere positive.  So the non-degenerate Hamiltonian fibrations are precisely the positively curved Hamiltonian fibrations.  Any Hamiltonian fibration which is degenerate can be made non-degenerate by adding $\pi^*(h \cdot \omega_{\Sigma})$ to $\Omega$ for any smooth function $h: \Sigma \to \IR$ that satisfies $R+\pi^{*}h>0$, where $R$ is the curvature with respect to $\omega_{\Sigma}$.  This produces a new coupling form that is non-degenerate.

It will be useful to consider the following situation.  Let $(M,\Omega,\pi)$ be a non-degenerate Hamiltonian fibration and let $f: M \to \IR$ be a smooth function on $M$.  Given a $1$-form $\alpha$ on $\Sigma$, consider the $\omega_\Sigma$-dual of $\alpha$, denoted by $X_\alpha$.  We wish to compute the $\Omega$-dual of $f \cdot \pi^* \alpha$ in terms of $X_\alpha$ and the curvature $R$ of this Hamiltonian fibration with respect to $\omega_\Sigma$.

\begin{lem}\label{lem:HorizontalLiftOfHamiltonianVectorField}
The $\Omega$-dual of the $1$-form $f \cdot \pi^* \alpha$ is the vector field
\[ X = \frac{f}{R} \cdot \widetilde{X_{\alpha}}.\]
\end{lem}

\begin{proof}
Let $X$ denote the $\Omega$-dual of $f \cdot \alpha$.  Notice that $X$ is horizontal since
\[ \Omega\left(X,\cdot\right)|_{\Ver} = f \cdot \alpha \circ d\pi(\cdot)|_{\Ver} = 0. \]
So we need to show that
\[ \Omega\left(f \cdot \widetilde{X_\alpha}/R,\widetilde{W}\right) = f \cdot \alpha(W),\]
where $W$ is an arbitrary vector field on $\Sigma$.  But this is easy:
\[ \Omega\left(f \cdot \widetilde{X_\alpha}/R, \widetilde{W}\right) = \frac{f}{R} \cdot \Omega\left(\widetilde{X_\alpha},\widetilde{W}\right) = f \cdot \omega_{\Sigma}(X_\alpha,W) = f \cdot \alpha(W).\]
\end{proof}

\subsection{Symplectic mapping cylinders}\label{subsec:HamFibs_MappingTori}
We discuss Hamiltonian fibrations given by symplectizations of stable Hamiltonian structures associated to mapping tori of symplectomorphisms.  These Hamiltonian fibrations naturally arise in the context of our main result.

Let $(F,\omega_F)$ be a closed symplectic manifold of dimension $2n-2$ and let $\psi: (F,\omega_F) \to (F,\omega_F)$ be a symplectomorphism.  The \emph{mapping torus of the symplectomorphism $\psi$} is
\[ M_\psi \coloneqq [0,2\pi] \times F / ((0,x) \sim (2\pi,\psi(x))).\]
Equivalently, $M_\psi$ is the quotient of $\IR \times F$ by the action $(t,x) \mapsto (t+2\pi, \psi(x))$.  The projection $[0,2\pi] \times F \to [0,2\pi]$ gives a projection $M_\psi \to S^1$.  So the pull-back of the angular $1$-form $d \theta$ on $S^1$ is a non-vanishing $1$-form on $M_\psi$, which we will also denote by $d\theta$.

\begin{lem}\label{lem:MappingTorusSHS}
Let $M_\psi$ be the mapping torus of the symplectomorphism $\psi$ with a closed $2$-form $\eta$ that restricts to $\omega_F$ on each fibre.  The pair $(\eta,d \theta)$ defines a stable Hamiltonian structure on $M_\psi$.\qed
\end{lem}

The hyperplane distribution of $(\eta,d \theta)$ is $\ker(d\theta)$, which is the tangent spaces of the fibres of the projection to $S^1$.  Similarly, the Reeb vector field is the horizontal lift of the angular vector field $\partial_\theta$ on $S^1$.

The stable Hamiltonian structure obtained in \cref{lem:MappingTorusSHS} could have a degenerate Reeb vector field.  However, this can be rectified by adding a generic perturbation to $\eta$.

\begin{lem}\label{lem:PerturbMappingTori}
Consider a stable Hamiltonian structure $(\eta,d\theta)$ on $M_\psi$ as in \cref{lem:MappingTorusSHS}.  For a generic choice of Hamiltonian, $H: M_\psi \to \IR$, the pair $(\eta + d(H \cdot d\theta), d \theta)$ defines a stable Hamiltonian structure on $M_\psi$ with non-degenerate Reeb vector field.
\end{lem}

\begin{proof}
View $M_\psi$ as the quotient of $\IR \times F$ by the action $(\theta,x) \mapsto (\theta+2\pi, \psi(x))$, where $\theta$ denotes both the angular coordinate on $S^1$ and the Cartesian coordinate on $\IR$.  Fix a Hamiltonian $H: M_{\psi} \to \IR$.  By \cite[Example 4.2.3]{Oh_SymplecticTopology}, pulling-back $\eta + d(H \cdot d\theta)$ to $\IR \times F$ gives
\[\eta + d(H \cdot d\theta) = \omega_F + dG_\theta \wedge d{\theta} + d{H}_{{\theta}} \wedge d {\theta}, \] 
where $dG_\theta$ is an $\IR$-family of exact $1$-forms on $F$ that satisfy $\psi^* dG_{\theta+2\pi} = d(G_\theta)$ and $H_\theta$ is the pull-back of $H$ (so $\psi^* {H}_{{\theta}+2\pi} = {H}_{{\theta}}$).
In this trivialization, the Reeb vector field is $\partial_{\theta} + V_{\theta}$, where
\[ \omega_F(V_{{\theta}},\cdot) = dG_\theta(\cdot) + d{H}_{{\theta}}( \cdot) .\]
If $\psi_{\theta}: F \to F$ is the time $\theta$ flow along $V_{\theta}$, then the Reeb flow is given by $(\cdot + \theta,\psi_{\theta}(\cdot))$.  So the Reeb orbits have periods in $2\pi \cdot \IZ$ and correspond to fixed points of $\psi \circ \psi_{\theta}$.  So Reeb vector field will be non-degenerate if the derivative of $\psi \circ \psi_{\theta}$ at any fixed point has no eigenvalues equal to $1$; however, $\psi_\theta$ is a Hamiltonian diffeomorphism.  So for generic $H$, at the fixed points, the derivative of $\psi \circ \psi_\theta$ will have no eigenvalues equal to $1$, as desired.
\end{proof}

We now discuss the relationship between mapping tori of symplectomorphisms and Hamiltonian fibrations.  Considering the product $(0,\infty) \times M_\psi$, we obtain a proper submersion $\pi: (0,\infty) \times M_\psi \to (0,\infty) \times S^1$.  Fix polar coordinates $(r,\theta)$ for $(0,\infty)\times S^1$. 

\begin{defn}\label{defn:SymplectizationMappingTorus}
Given a closed $2$-form $\eta$ on $M_\psi$ that restricts to $\omega_F$ on each fibre, the \emph{symplectization of $M_\psi$ with respect to $\eta$} is the Hamiltonian fibration
\[((0,\infty) \times M_\psi,\pi,\eta + \pi^* (dr \wedge d\theta)).\]
Such a Hamiltonian fibration is a \emph{symplectic mapping cylinder}.
\end{defn}

The following characterizes the relationship between symplectic mapping cylinders and Hamiltonian fibrations over $(0,\infty) \times S^1$.

\begin{lem}\label{lem:1FormDeterminesMappingTorus}
The symplectic mapping cylinder $((0,\infty) \times M_\psi,\pi,\eta + \pi^* (dr \wedge d\theta))$ defines a non-degenerate Hamiltonian fibration whose curvature with respect to $\pi^*(dr \wedge d\theta)$ is constant and equal to $1$.  Conversely, if a Hamiltonian fibration $(M,\pi, \eta + \pi^{*}(dr \wedge d\theta))$ satisfies $\pi: M \to (0,\infty)\times S^1$, $\eta(\partial_r,\cdot)|_{\Ver} \equiv 0$, and $\partial_r(\eta) \equiv 0$, then it is given by a symplectic mapping cylinder.
\end{lem}

\begin{proof}
To prove the first claim, it suffices to check that the curvature is always equal to one (the other elements of the first claim being clear).  Let  $\partial_r$ and $\partial_\theta$ be the vector fields associated to coordinates $(r,\theta)$.  The horizontal lift of $\partial_r$ is $\partial_r$:
\[ (\eta+\pi^*(dr \wedge d\theta))\left({\partial_r},\cdot\right)|_{\Ver} = \left(\eta({\partial_r},\cdot\right) + (dr \wedge d\theta)(\partial_r,\cdot))|_{\Ver} = \eta(\partial_r,\cdot)|_{\Ver} \equiv 0.\]
So the curvature is always one:
\[ R \cdot (dr \wedge d\theta)(\partial_r,\partial_\theta)=(\eta+\pi^*(dr \wedge d\theta))(\widetilde{\partial_r},\widetilde{\partial_\theta}) = \eta({\partial_r},\widetilde{\partial_\theta}) + (dr \wedge d\theta)(\partial_r,\partial_\theta) = (dr \wedge d\theta)(\partial_r,\partial_{\theta}).\]
Consequently, $R \equiv 1$.

The converse follows from \cite[Theorem 4.2.4]{Oh_SymplecticTopology}.  Briefly, one considers the family of symplectomorphisms
\[\psi_\theta: (\pi^{-1}(1,\theta), \eta|_{\pi^{-1}(1,\theta)}) \to (\pi^{-1}(1,\theta), \eta|_{\pi^{-1}(1,\theta)})\]
obtained from parallel transporting along $\gamma(\theta) = (1,\theta) \in (0,\infty) \times S^1$.  These parallel transport maps give trivializations of $M$ that agree with trivializations of $(0,\infty)\times M_{\psi_{2\pi}}$.  So the two spaces are fibrewise diffeomorphic, giving the result.
\end{proof}

\begin{rem}
It should be clear from \cref{defn:SymplectizationStableHamStructure} and \cref{defn:SymplectizationMappingTorus} that these two definitions of symplectizations agree for mapping tori of symplectomorphisms.
\end{rem}

We conclude by discussing admissible almost complex structures.

\begin{lem}\label{lem:MappingTorusAdmissibleJ}
The space of admissible almost complex structures of a convex symplectic domain with boundary given by a mapping torus of a symplectomorphism is non-empty and contractible.
\end{lem}

\begin{proof}
Fix a collar: $((0,1] \times M_\psi, \eta + dr \wedge d\theta)$.  The condition that $d(d\theta)(\cdot, J \cdot)$ is non-negative definite on all $J$-complex lines in $T(M_\psi)$ is vacuously satisfied.  The conditions that $J(\ker(d\theta)) = \ker(d\theta)$ and $-dr \circ J = rd\theta$ imply that $J$ (with respect to the basis $(r \cdot \partial_r) \oplus \widetilde{\partial_\theta} \otimes T F$) is
\[ \begin{pmatrix} 0 & -1 & 0 \\ 1 & 0 & 0 \\ 0 & 0 & J_F \end{pmatrix} \]
where $J_F$ is an $(0,1] \times S^1$ family of almost complex structures on the fibre.  The condition that $J$ be $(\eta + dr \wedge d\theta))$-compatible implies that $J_F$ is $\omega_F$-compatible.  So a choice of admissible almost complex structure (in the collar) corresponds to a choice of $\omega_F$-compatible almost complex structures on the vertical distribution.  The space of such choices is non-empty and contractible.  This gives the desired result.
\end{proof}

\begin{rem}\label{rem:MappingToriIntegrated}
As an upshot of the proof of \cref{lem:MappingTorusAdmissibleJ}, any admissible almost complex structure makes the projection $(0,1] \times M_\psi \to (0,1] \times S^1$ holomorphic.  Moreover, for a radially admissible choice of Hamiltonian, Floer trajectories in the collar will project to Floer trajectories in $(0,1] \times S^1$.
\end{rem}

\subsection{Flattening Hamiltonian fibrations}\label{subsec:HamFibs_FlatteningHamFib}

We investigate the relationship between symplectic mapping cylinders and (more general) non-degenerate Hamiltonian fibrations over $(0,\infty)\times S^1$.  Let $(M,\Omega)$ be a symplectic manifold.  Let $\pi: M \to \IC$ be a smooth, proper map that is a submersion over $\IC^\times$.  Assume that $(\pi^{-1}(\IC^\times), \Omega, \pi)$ is a non-degenerate Hamiltonian fibration.  For example, this holds when $M$ admits an $\Omega$-compatible almost complex structure that makes $\pi$ holomorphic.  Finally, let $\omega \coloneqq d(f d\theta)$ be a symplectic form on $(0,\infty) \times S^1$, where $f$ is a smooth function in $r$ with non-negative second derivative.  Similarly as in \cref{notn:rInverse}, let $M_a \coloneqq \pi^{-1}(\ID_a)$, where $\ID_a$ denotes the disk of radius $a$ in $\IC$.

\begin{prop}\label{prop:FlatteningOverCStarNonDeg}
Given $c>b>a>0$, there exists a symplectic form $\overline{\Omega}$ on $M$ and a symplectic embedding
\[\varphi: (M,\Omega) \hookrightarrow (M,\overline{\Omega})\]
that satisfy:
\begin{enumerate}
	\item $\varphi$ is the identity on $M_a$,
	\item $M_b \subset \varphi(\intrr(M_c))$, and 
	\item on $M \smallsetminus M_b$, $\overline{\Omega} = \overline{\eta} + \pi^{*}\omega$, 
\end{enumerate}
where $\overline{\eta}$ is a closed $2$-form that satisfies:
\begin{itemize}
	\item the restriction to each fibre agrees with $\omega_F$,
	\item $\overline{\eta}(\partial_r,\cdot)|_{\Ver} \equiv 0$, 
	\item $\partial_r(\overline{\eta}) \equiv 0$, and
	\item the Reeb vector field of the stable Hamiltonian structure $(\overline{\eta},d\theta)$ on $\partial M_b$ is non-degenerate.
\end{itemize}
\end{prop}

\begin{rem}
When $\omega = dr \wedge d\theta$, \cref{lem:1FormDeterminesMappingTorus} combined with \cref{prop:FlatteningOverCStarNonDeg} implies that $(M,\Omega)$ may be symplectically embedded inside a symplectic manifold whose end is modeled after a symplectic mapping cylinder with non-degenerate Reeb vector field.  Condition (i) implies that one can arrange for this embedding to be supported in the end of $M$.  Condition (ii) implies that one can further arrange for the image of a fixed slice of the end of $M$ under $\varphi$ to be properly contained inside the symplectization region.
\end{rem}

To prove \cref{prop:FlatteningOverCStarNonDeg}, we study \cref{prop:FlatteningOverCStar} and combine it with \cref{lem:PerturbMappingTori}.

\begin{prop}\label{prop:FlatteningOverCStar}
Given $c>b>a>0$, there exists a symplectic form $\widetilde{\Omega}$ on $M$ and a symplectic embedding
\[\psi: (M,\Omega) \hookrightarrow (M,\widetilde{\Omega})\]
that satisfy:
\begin{enumerate}
	\item $\psi$ is the identity on $M_a$,
	\item $M_b \subset \psi(\intrr(M_c))$, and 
	\item on $M \smallsetminus M_b$, $\widetilde{\Omega} = \eta + \pi^{*}\omega$, where $\eta = \Omega|_{M_b}$.
\end{enumerate}
\end{prop}

The proof of \cref{prop:FlatteningOverCStar} is technical and quite involved.  The idea is the following.  Consider the obvious radial deformation retraction $\varphi: M \to M_b$.  The pull-back $\varphi^*\Omega$ is cohomologous to $\Omega$; however, it is not symplectic.  Nevertheless, $\varphi^*\Omega + \omega$ is symplectic and still cohomologous to $\Omega$.  So by Moser's argument, we expect a symplectic embedding
\[ (M, \Omega) \to (M, \varphi^*\Omega+\omega). \]
Given this, condition (iii) in \cref{prop:FlatteningOverCStar} would be satisfied by construction.  However, this map is not the identity over $M_a$ and, a priori, need not satisfy condition (ii) in \cref{prop:FlatteningOverCStar}.  So to achieve conditions (i) and (ii), above, one could replace $\varphi^*\Omega+\omega$ with $\varphi^*\Omega+\ell \cdot \omega$, where $\ell$ is an appropriately defined radial cut-off function on $M$ that vanishes on $M_a$.  

There are issues that arise when trying to make this rigorous.  First, the symplectic embedding obtain from Moser's argument need not be well-defined.  The flow that gives this symplectic embedding may have solutions that exit all compact subsets of $M$ in finite time.  Second, the retraction $\varphi$ is never smooth and thus $\varphi^*\Omega$ is not a differential $2$-form.  To deal with these complications, we break up the desired symplectic embedding into three separate symplectic deformations, each of which is explicit enough for us to control the solutions to the flow equation given by Moser's argument and, thus, obtain well-defined symplectic mappings as well as our desired properties.  But first, we fix coordinates and a description of $\Omega$.

\begin{notn}\label{notn:Flattening}
Using polar coordinates $(r,\theta)$ for $(a,\infty) \times S^1 \subset \IC^\times$, the composition $r \circ \pi$ is a proper submersion (and thus a fibre bundle)
\[ \pi^{-1}((a,\infty) \times S^1) \to (a,\infty). \]
Trivialize $\pi^{-1}((a,\infty) \times S^1)$ over $(a,\infty)$:
\[ \xymatrix{  (a,\infty) \times \pi^{-1}( b \times S^1) \ar[r] \ar[d] & \pi^{-1}((a,\infty) \times S^1) \ar[d] \\ (a,\infty) \ar[r] & (a,\infty).} \]
Define $S \coloneqq \pi^{-1}(b \times S^1)$.  So $\pi^{-1}((a,\infty) \times S^1) \cong (a,\infty) \times S$ with the first factoring agreeing with $r \circ \pi$.  We will refer to this coordinate as $r$.  The projection $(a,\infty) \times S \to S$ is a (smooth) homotopy equivalence with (smooth) homotopy inverse given by the inclusion of $S$.  So over $(a,\infty)$, $\Omega$ is cohomologous to the pull-back of a closed 2-form on $S$.  More precisely, $\Omega =  \eta + d(\xi)$, where $\eta$ is the closed $2$-form given by pulling-back $\Omega$ along the composition
\[\pi^{-1}((a,\infty) \times S^1) \cong (a,\infty) \times S \twoheadrightarrow  S \hookrightarrow \pi^{-1}((a,\infty) \times S^1).\]
By this description, $\eta(\partial_r,\cdot)$ and $\partial_r(\eta)$ are both zero and $\xi$ may be chosen to be a $1$-form whose restriction to $S$ is zero.  Indeed, if $\xi$ did not vanish along $S$, then restricting $\eta + d(\xi)$ to $S$ gives that $\xi$ restricted to $S$ is closed.  So we can replace $\xi$ by $\xi - (\xi|_{S})$ and the resulting form will now vanish along $S$.
\end{notn}

\begin{lem}\label{lem:Flattening1}(The first symplectic deformation)
Given $c>b>a>0$ as in \cref{prop:FlatteningOverCStar}, for $\delta>0$ sufficiently small, there exists a symplectic form $\Omega'$ on $(a,\infty) \times S$ and a symplectic embedding
\[\psi': ((a,\infty)\times S,\Omega) \hookrightarrow ((a,\infty)\times S,\Omega')\]
that satisfy:
\begin{enumerate}
	\item $\psi'$ is the identity for $r \leq a-(b-a)/2$,
	\item $\Omega' = \eta + d(\xi')$ with
	\[\xi' = \left(f+c_\delta\right) \pi^* d\theta\]
	for $b-\delta \leq r \leq b+\delta$, where $c_\delta$ is a fixed constant, and
	\item $r^{-1}((a,b+\delta]) \subset \psi'(r^{-1}((a,c)))$.
\end{enumerate}
\end{lem}

The idea for \cref{lem:Flattening1} is to pull-back $\Omega$ along a radial self-map of $(a,\infty) \times S$ that smoothly collapses a small annular neighborhood of $\pi^{-1}(b \times S^1)$ onto $\pi^{-1}(b \times S^1)$.  One then add a small multiple of $\omega$ (multiplied by an appropriate radial cut-off function) to this pull-back to obtain a symplectic form that is cohomologous to the original $\Omega$.  One then applies Moser's argument to obtain the desired symplectic embedding.  If this self-map is the identity near $a \times S$ and one collapses a sufficiently small annular neighborhood of $\pi^{-1}(b \times S^1)$ onto $\pi^{-1}(b \times S^1)$, then by adding a sufficiently small multiple of $\omega$ as above, we can obtain the additional conditions of the lemma.

\begin{proof}
For each $\delta>0$ sufficiently small, fix a smooth function $\ell_\delta: \IR_{\geq 0} \to \IR_{\geq 0}$ that satisfies:
\begin{itemize}
	\item $\ell_\delta(r) = r$ for $0 \leq r \leq a + (b-a)/2$ and $c-(c-b)/2 \leq r$,
	\item $\ell_\delta(r) = b$ for $b-\delta \leq r \leq b+\delta$, and
	\item $\ell_\delta'(r) \geq 0$ for all $r$ with $\ell'(r) = 0$ if and only if $b-\delta \leq r \leq b + \delta$.
\end{itemize}
For each $\delta$, we have a smooth family of maps
\[ \phi^\delta_s: (a,\infty) \times S \to (a,\infty) \times S, \hspace{10pt} \phi^\delta_s(r,x) = ( (1-s)\cdot r + s \cdot \ell_\delta(r), x) \]
for $0 \leq s \leq 1$.  This gives the family $(\phi^\delta_s)^* \Omega$ of cohomologous $2$-forms on $(a,\infty) \times S$.  Since $d \phi^\delta_s|_{\Ver} = \Ione$ for each $s$, each $(\phi^\delta_s)^*\Omega$ is a coupling form for $(a,\infty) \times S$ with respect to $\pi$.  However, each form is not necessarily non-degenerate and thus is not necessarily symplectic.  Note, $\det(d\phi^\delta_s) \geq 0$ and $\det(d \phi^\delta_s) = 0$ if and only if $b-\delta \leq r \leq b+\delta$ and $s = 1$.  So $(\phi^\delta_s)^* \Omega$ is degenerate if and only if $b-\delta \leq r \leq b+\delta$ and $s = 1$.  Let $R_s$ denote the curvature of $\Omega_s$ with respect to $\omega$.  We have $R_s \geq 0$ with $R_s = 0$ if and only if $b-\delta \leq r \leq b+\delta$ and $s = 1$.  

To rectify the degeneracy, for each $\varepsilon>2b\delta$ sufficient small, fix a smooth function $H_\varepsilon: \IR_{\geq 0} \to \IR_{\geq 0}$ that satisfies:
\begin{itemize}
	\item $H_\varepsilon(r) = 0$ for $0 \leq r \leq a +(b-a)/2$,
	\item $H_\varepsilon(r) = f+c_\varepsilon$ for $b-\delta \leq r \leq b + \delta$, where $c_\varepsilon$ is some constant,
	\item $H_\varepsilon(r) = \varepsilon$ for $c - (c-b)/2 \leq r$, and
	\item $H_\varepsilon'(r) \geq 0$ for all $r$.
\end{itemize}
Define a smooth family of $2$-forms on $(a,\infty) \times S$,
\[ \Omega_s \coloneqq (\phi^\delta_s)^* \Omega + s \cdot d(H_\varepsilon \cdot \pi^*d\theta).\]
By construction, $\Omega_s$ is a smooth family of cohomologous symplectic forms on $(a,\infty)\times S$.

We wish to apply Moser's argument to this symplectic deformation to obtain our desired symplectic embedding.  Using our description of $\Omega = \eta + d(\xi)$ and the fact that $(\phi_s^\delta)^*\eta = \eta$ ($\eta$ is independent of $r$),
\[ \frac{d}{ds}\left(\Omega_s\right) = d \left( \frac{d}{ds} \left( (\phi_s^\delta)^* \xi \right) + H_\varepsilon \cdot \pi^*d\theta \right) \eqqcolon d(\xi_s).\]
Define $V_s$ by $\Omega_s(V_s,\cdot) = -\xi_s$.  Let $\psi_s$ denote the time $s$ flow of $V_s$.  By Moser's argument, if $\psi_s$ is well-defined, then it is a symplectic embedding.  We claim that $\psi' \coloneqq \psi_1$ and $\Omega' = \Omega_1$ define the desired pieces of data for $\delta$ and $\varepsilon$ sufficiently small.

First, we show that $\psi_s$ is well-defined for $0\leq s \leq 1$ and is the identity for $r \leq a-(b-a)/2$.  Since $\phi_s^\delta$ is the identity for $r \leq a - (b-a)/2$ and $c - (c-b)/2 \leq r$, we have for said values of $r$ that
\[\frac{d}{ds} \left( (\phi_s^\delta)^* \xi \right) \equiv 0. \]
Also for $r \leq a - (b-a)/2$, $H_\varepsilon = 0$ and for $c - (c-b)/2 \leq r$, $H_\varepsilon$ is constant and equal to $\varepsilon$.  So for $r \leq a - (b-a)/2$, $V_s = 0$ and, consequently, $\psi_s$ is the identity and well-defined for $0\leq s \leq 1$ on this end.  On the other end, by \cref{lem:HorizontalLiftOfHamiltonianVectorField}, for $c - (c-b)/2 \leq r$,
\[ V_s = \frac{-\varepsilon}{f' \cdot R_s} \cdot \widetilde{\partial_r}.\]
So $V_s$ is always pointing inward with respect to $r$ and thus $\psi_s$ is well-defined for $0\leq s \leq 1$ on this end.  It follows that $\psi_s$ is well-defined for $0\leq s \leq 1$. 

Second, we show $\Omega' = \eta + d(\xi')$ with $\xi' = \left(f+c_\varepsilon\right) \pi^* d\theta$ for $b-\delta \leq r \leq b+\delta$.  Notice that for $b-\delta \leq r \leq b+\delta$, the map $\phi^\delta_1$ agrees with the projection $(a,\infty) \times S \to S$.  So $(\phi_1^\delta)^* \Omega = \eta$ and $(\phi_1^\delta)^*\xi = 0$ for $b-\delta \leq r \leq b+\delta$.  Setting
\[ \xi' = \xi_1 = d\left((\phi^\delta_1)^*\xi+ H_\varepsilon \cdot \pi^* d\theta\right),\]
the desired result follows.

Finally, we explain the inclusion $r^{-1}((a,b+\delta]) \subset \psi_s(r^{-1}((a,c)))$.  For $c - (c-b)/2 \leq r$, our description of $V_s$ (and compactness) implies that
\[ \lim_{\varepsilon \to 0} |r(\psi_s) - r| = 0 \]
on all compact subsets.  So the inclusion holds by taking $\varepsilon$ and $\delta$ sufficiently small.
\end{proof}

\begin{lem}\label{lem:Flattening2}(The second symplectic deformation)
Assuming the notation of \cref{lem:Flattening1}, there exists a symplectic form $\Omega''$ on $(a,\infty) \times S$ and a symplectic embedding
\[\psi'': ((a,\infty)\times S,\Omega') \hookrightarrow ((a,\infty)\times S,\Omega'')\]
that satisfy:
\begin{enumerate}
	\item $\psi''$ is the identity for $r \leq b+\delta/3$, and
	\item $\Omega'' = \eta + d(\xi'')$ with
	\[\xi'' = F \cdot \pi^*d\theta\] 
	for $b-\delta \leq r$, where $F$ is a smooth function in $r$ that satisfies $F(r) =  f+C$ for $b-\delta \leq r \leq b+\delta/3$ with $C$ being some fixed constant.
\end{enumerate}
\end{lem}

The idea of \cref{lem:Flattening2} is to pull-back $\Omega$ along a radial self-map of $(a,\infty) \times S$ that smoothly collapses the entire positive end of $(a,\infty) \times S$ onto $\pi^{-1}((b+\delta/2) \times S^1)$ for some small $\delta$.  One then add a sufficiently large, radially varying, symplectic form from the base $(a,\infty) \times S^1$ (again, multiplied by an appropriate radial cut-off function) to this pull-back to obtain a symplectic form that is cohomologous to the original $\Omega$.  When the symplectic area of this base form is sufficiently large, one can show that the flow obtained from Moser's argument is well-defined.  The additional conditions of the lemma will follow as in \cref{lem:Flattening2}.

\begin{proof}
Fix a smooth function $\ell: \IR_{\geq 0} \to \IR_{\geq 0}$ that satisfies
\begin{itemize}
	\item $\ell(r) = r$ for $0 \leq r \leq b+\delta/3$,
	\item $\ell(r) = b+\delta/2$ for $b+\delta/2 \leq r$, and
	\item $\ell'(r) \geq 0$ for all $r$ with $\ell'(r) = 0$ if and only if $b+\delta/2 \leq r$.
\end{itemize}
As in \cref{lem:Flattening1}, we have a smooth family of maps
\[ \phi'_s: (a,\infty) \times S \to (a,\infty) \times S, \hspace{10pt} \phi_s'(r,x) = ( (1-s)\cdot r + s \cdot \ell(r), x) \]
for $0 \leq s \leq 1$, which gives rise to the family $(\phi_s')^* \Omega'$ of cohomologous coupling forms (with respect to $\pi$) on $(a,\infty) \times S$.  Let $R_s'$ denote the curvature of the coupling form $(\phi_s')^* \Omega'$ with respect to $\omega$.  As in \cref{lem:Flattening1}, $R_s' \geq 0$ with $R_s' = 0$ if and only if $b+\delta/2 \leq r$ and $s = 1$.  To run Moser's argument, we again need to rectify this degeneracy; however, there is the complication that the support of $\phi_s'$ is not compact.  So we need a more delicate construction. 

Let $r_n$ be a strictly increasing sequence of real numbers that diverges to infinity and satisfies $r_n > c$ for each $n$.  Let $U_n$ be a collection of connected, pairwise disjoint open intervals in $\IR_{>c}$ such that $r_n \in U_n$.  For each increasing sequence of positive real numbers $c_n$, fix a smooth function $H_{\{c_n\}}: \IR_{\geq 0} \to \IR_{\geq 0}$ that satisfies:
\begin{itemize}
	\item $H_{\{c_n\}}(r) = 0$ for $0 \leq r \leq b + \delta/3$,
	\item $H_{\{c_n\}}'(r) \geq 0$ with $H_{\{c_n\}}'(r) > 0$ for $b+\delta/2 \leq r$, and
	\item $H_{\{c_n\}}(r) = r+c_n$ for $r \in U_n$.
\end{itemize}
Define a smooth family of $2$-forms on $(a,\infty) \times S$ by
\[ \Omega_s' \coloneqq (\phi_s')^* \Omega' + s \cdot d\left(H_{\{c_n\}} \cdot \pi^*d\theta\right).\]
$\Omega_s'$ is a smooth family of cohomologous symplectic forms on $(a,\infty)\times S$.

As in \cref{lem:Flattening1}, we apply Moser's argument to obtain our symplectic embedding.  Like before,
\[ \frac{d}{ds}\left(\Omega_s'\right) = d \left( \frac{d}{ds} \left( (\phi_s')^* \xi' \right) + H_{\{c_n\}} \cdot \pi^*d\theta \right) \eqqcolon d(\xi_s').\]
Define $V_s'$ by $\Omega_s'(V_s',\cdot) = -\xi_s'$.  Let $\psi_s'$ denote the time $s$ flow of $V_s'$.  We claim that for a particular choice of sequence $c_n$, the flow $\psi_s'$ is well-defined and is the identity for $r \leq b+\delta/3$.

We first show that $\psi_s'$ is the identity (and thus well-defined) for $r \leq b+\delta/3$.  Notice that $\phi_s'$ is the identity for $r \leq b+\delta/3$.  So for said values of $r$,
\[\frac{d}{ds} \left( (\phi_s')^* \xi' \right) \equiv 0. \]
Also for $r \leq b+\delta/3$, $H_{\{c_n\}} = 0$.  So for $r \leq b + \delta/3$, $V_s' = 0$ and, consequently, $\psi_s'$ is the identity on this end.

We now show that $\psi_s'$ is well-defined for $c\leq r$.  By \cref{lem:HorizontalLiftOfHamiltonianVectorField}, for $b+\delta/2 \leq r$, the $\Omega_s'$-dual of $-H_{\{c_n\}} \cdot \pi^* d\theta$ is the vector field
\[ W_s = \frac{-H_{\{c_n\}}}{\left(f' \cdot R_s' + s \cdot H_{\{c_n\}}'\right)} \cdot \widetilde{\partial_r}.\]
Let $Z_s$ denote the vector field determined by
\[ \Omega_s'(Z_s,\cdot) = -\frac{d}{ds} \left( (\phi_s')^* \xi' \right).\]
Using the splitting of the tangent space into vertical and horizontal components, write $Z_s = f_s \cdot \widetilde{\partial_r} + Y_s$, where $Y_s$ is orthogonal to $\widetilde{\partial_r}$.  Notice that if
\[ f_s + \frac{-H_{\{c_n\}}}{\left(f'\cdot R_s' + s \cdot H_{\{c_n\}}'\right)} \leq 0,\]
then the flow of $V_s'$ is inward pointing.  Notice that over each $U_n$,
\[ \Omega_s' = (\phi_s')^* \Omega' + s \cdot dr \wedge d\theta.\]
So $f_s$ is independent of $c_n$ on $U_n$.  On $r^{-1}(U_n)$, the above equation is
\[ f_s(r,x) - \frac{r+c_n}{(f' \cdot R_s'(r,x) + s)} \leq 0.\]
By compactness, pick $c_n$ sufficiently large so that this inequality holds over $U_n$.  So the flow $\psi_s'$ is inward pointing over each $U_n$.  It follows that the flow of $\psi_s'$ is well-defined for $0\leq s \leq 1$ on every compact subset of $(a,\infty)\times S$ and thus is globally well-defined for $0\leq s \leq 1$, as desired.

Fix $c_n$ so that the above flow $\psi_s'$ is well-defined.  Let $\psi'' \coloneqq \psi_1'$ and let $\Omega'' \coloneqq \Omega_1'$.  To complete the proof, we need to show that $\Omega'' = \eta + d(\xi'')$ with $\xi'' = F \cdot \pi^*d\theta$ for $b-\delta \leq r$, where $F$ is a smooth function in $r$ that satisfies $F(r) = f+C$ for $b-\delta \leq r \leq b+\delta/3$ with $C$ being some fixed constant.  As before, we have 
\[ \xi'' = (\phi_1')^* \xi' + H_{\{c_n\}} \cdot \pi^*d\theta = (\phi_1')^*(\phi_1^*\xi) + (\phi_1')^*H_\varepsilon \cdot \pi^*d \theta + H_{\{c_n\}} \cdot \pi^* d \theta. \]
Notice that for $b-\delta \leq r$, the composition, $\phi_1 \circ \phi_1'$, is given by the projection onto $S$.  So $(\phi_1')^*(\phi_1^*\xi) = 0$.  Since $\phi_1'$ only depends on the variable $r$, $F \coloneqq (\phi_1')^*H_\varepsilon+H_{\{c_n\}}$ is a smooth function only depending on $r$ and $\xi'' = F \cdot \pi^*d\theta$ for $b-\delta \leq r$.  Finally, $F$ agrees with $f+C$ for $b-\delta \leq r \leq b+\delta/3$.  Indeed, $\phi_1'$ is the identity for $r \leq b+\delta/3$, $H_{\{c_n\}} \equiv 0$ for $r \leq b+\delta/3$, and $H_\varepsilon = f+C$ for $b-\delta \leq r \leq b+\delta$.  So for $b-\delta \leq r \leq b+\delta/3$,
\[ F(r) = (\phi_1'^*H_\varepsilon+H_{\{c_n\}})(r) = H_\varepsilon(r)+H_{\{c_n\}}(r) = f+C.\]
\end{proof}

\begin{lem}\label{lem:Flattening3}(The third symplectic deformation)
Assuming the notation of \cref{lem:Flattening1} and \cref{lem:Flattening2}, there exists a symplectic form $\Omega'''$ on $(b-\delta,\infty) \times S$ and a symplectic embedding
\[\psi''': ((b-\delta,\infty)\times S,\Omega'') \hookrightarrow ((b-\delta,\infty)\times S,\Omega''')\]
that satisfy:
\begin{enumerate}
	\item $\psi'''$ is the identity for $r \leq b+\delta/3$, and
	\item $\Omega''' = \eta + d(\xi''')$ with $\xi''' = (f+C) \cdot \pi^*d\theta$, where $C$ is some constant.
\end{enumerate}
\end{lem}

We indicate the idea.  The form $\Omega''$ in \cref{lem:Flattening2} almost satisfies the conditions of \cref{lem:Flattening3}.  For $r \geq b + \delta$, $\Omega'' = \eta + d(F \cdot \pi^* d\theta)$.  So we need to alter the form so that for $r \geq b + \delta$, $\Omega''' = \eta + d(f \cdot \pi^* d\theta)$.  In the case where $\eta = 0$ (that is, when our total space is $\IC$), we can always construct (the obvious) symplectic deformation between these two forms.  The associated flow from Moser's argument can be explicitly computed using that it is radially determined and area preserving.  When one explicitly solves, one sees that the flow is well-defined when the symplectic areas of $d(f \cdot \pi^* d\theta)$ and $d(F \cdot \pi^* d\theta)$ are infinite, which is ensured by our construction.  So to produce the result for $\eta \neq 0$, we relate its associated flow from Moser's argument to the flow in case where $\eta =0$.

\begin{proof}
Consider the smooth function $G: \IR_{\geq b-\delta} \to \IR_{\geq 0}$, $G(r) = f+C$, where $C$ is the constant in \cref{lem:Flattening2}.  We have a smooth family of symplectic forms
\[ \Omega_s'' = \eta + d( ( (1-s)\cdot F + s \cdot G) \cdot \pi^*d\theta)\]
for $0 \leq s \leq 1$.  By \cref{lem:HorizontalLiftOfHamiltonianVectorField} and the discussion in the proof of \cref{lem:1FormDeterminesMappingTorus}, the vector field $V_s''$,
\[\Omega_s''(V_s'',\cdot) = (F-G) \cdot d\theta,\]
is
\[ V_s'' = \frac{F-G}{(1-s) \cdot F' + s \cdot G'} \cdot \partial_r.\]
Let $\psi_s''$ denote the time $s$ flow along $V_s''$.  We claim that $\psi_s''$ is well-defined and the identity for $r \leq b+\delta/3$.

Notice that the flow of $V_s''$ is well-defined if and only if the flow of $W_s$ is well-defined, where $W_s$ is the vector field on $(b-\delta,\infty)\times S^1$ determined by
\[ d( ( (1-s)\cdot F + s \cdot G) \cdot d\theta)(W_s,\cdot) = (F-G) \cdot d\theta,\]
that is,
\[ W_s = \frac{F-G}{(1-s) \cdot F' + s \cdot G'} \cdot \partial_r.\]
By Moser's argument, the flow along $W_s$ is a symplectomorphism and thus an area preserving map.  The time $s$ flow of $W_s$ may be written as $(r,\theta) \mapsto (\varphi_s(r),\theta)$, where $\varphi_s: \IR_{ \geq 0} \to \IR_{\geq 0}$.  By the area preserving property, we have that
\begin{align*}
F(r) - F(b) & = \int_{b}^{r} F'\, dr 
\\ & = \int_{b}^{\varphi_s(r)} ((1-s) \cdot F' + s \cdot G')\, dr 
\\ & = (1-s)\cdot (F(\varphi_s(r)) - F(b)) + s \cdot (G(\varphi_s(r)) - G(b)).
\end{align*}
Canceling terms (recall that $F(b) = G(b)$),
\[ F(r) = (1-s) \cdot F(\varphi_s(r)) + s \cdot G(\varphi_s(r)).\]
The derivative of the right-hand-side is positive for $0\leq s \leq 1$.  So for $0 \leq s \leq 1$, the right-hand-side is invertible.  So we can solve for $\psi_s''$ when the range of the left-hand-side is contained in the range of the right-hand-side.  Notice that when $s > 0$, the range of the right-hand-side is $[f(b)+C,\infty)$, because $G(r) = f(r)+C$ is exhaustive (since $f'$ and $f''$ are both positive) and $F(b) = G(b)$.  Since $F'$ is positive and $F(b) = f(b)+C$, the range of the left-hand-side is contained in $[f(b)+C,\infty)$.  So we can solve for $\psi_s''$ and the flow is well-defined.  

Finally, since $F = G$ for $r \leq b+\delta/3$, $V_s'' = 0$ and thus $\psi_s''$ is the identity for said values of $r$.  Setting $\psi''' \coloneqq \psi_1''$ and
\[\Omega''' \coloneqq \Omega_1'' = \eta + d( G \cdot \pi^* d\theta) = \eta + d\left( \left(f+C \right) \cdot \pi^* d\theta\right)\] 
yields the desired data.
\end{proof}

We now chain together the above deformations and prove \cref{prop:FlatteningOverCStar}.

\begin{proof}
First, the symplectic embeddings and symplectic forms constructed in the above lemmas all extend to symplectic (self) embeddings and symplectic forms on all of $M$, because $\psi'$, $\psi''$, and $\psi'''$ are all the identity near $a$, $a$, and $b-\delta$ respectively.  We do not distinguish between the original maps and forms and their respective extensions to $M$.  Set $\psi = \psi''' \circ \psi'' \circ \psi'$ and $\widetilde{\Omega} = \Omega'''$.  Notice that for $b-\delta \leq r$, $\Omega''' = \eta + \pi^*\omega$. $\psi$ is the identity over $\ID_a \subset \IC$ because $\psi'$, $\psi''$, and $\psi'''$ are each the identity in this region.  Finally, the inclusion, $M_b \subset \psi(\intrr(M_c))$, follows from the inclusion property of $\psi'$ from \cref{lem:Flattening1} and the fact that $\psi''$ and $\psi'''$ are both the identity for $r \leq b+\delta/3$.  This finally completes the proof.
\end{proof}

Now we produce one final deformation to prove \cref{prop:FlatteningOverCStarNonDeg}.  The idea is to introduce a small perturbation as in \cref{lem:PerturbMappingTori} to the symplectic form $\widetilde{\Omega}$ from \cref{prop:FlatteningOverCStar}.  If we make the perturbation sufficiently small, then applying Moser's argument will produce a symplectomorphism from $\widetilde{\Omega}$ to the perturbed form that still preserves the additional conditions of the propositions.

\begin{proof}
Assume the notation and conclusions of \cref{prop:FlatteningOverCStar}.  For $\delta>0$ sufficiently small, consider a generic Hamiltonian $H: M \to \IR$ that satisfies:
\begin{enumerate}
	\item $H \equiv 0$ for $r \leq a+\delta$,
	\item $H$ is independent of $r$ for $r \geq b-\delta$, and
	\item $H \geq 0$.
\end{enumerate}
Define
\[ \overline{\Omega}_s \coloneqq \widetilde{\Omega} + s \cdot d(H \cdot d\theta). \]
For $|H|$ sufficiently small, $\overline{\Omega}_s$ will be symplectic.
Consider the vector field $V_s$,
\[ \overline{\Omega}_s(V_s,\cdot) = -H \cdot d \theta. \]
By \cref{lem:HorizontalLiftOfHamiltonianVectorField}, for $r >0$,
\[ V_s = \frac{-H}{f' \cdot R_s} \cdot \partial_r, \]
where $R_s$ is the curvature of $\overline{\Omega}_s$ with respect to $\omega$.  Let $\psi_s$ be the time $s$ flow along $V_s$.  We claim that the symplectic form $\overline{\Omega} \coloneqq \overline{\Omega}_1$ and the symplectic mapping
\[ \psi_1 \circ \psi: (M,\Omega) \to (M,\overline{\Omega}) \]
define the desired data.  

We argue that $\psi_s $ is well-defined for all $0 \leq s \leq 1$.  For $r \leq a+\delta$, $H \equiv 0$.  So for $r \leq a +\delta$, $V_s$ vanishes and $\psi_s$ is the identity in this region.  For $r \geq b-\delta$, $H$ is independent of $r$.  So by \cref{lem:1FormDeterminesMappingTorus}, $R_s \equiv 1$.  Since $H \geq 0$, for $r \geq b-\delta$, $V_s$ is inward pointing along $\partial_r$.  So the flow $\psi_s$ is globally well-defined for $0 \leq s \leq 1$.

Next, by \cref{prop:FlatteningOverCStar} and the discussion above, $\psi_1 \circ \psi$ is the identity on $M_a$.  To obtain the inclusion $M_b \subset (\psi_1 \circ \psi)(\intrr(M_c))$, we argue as follows.  Over the support of $\psi_s$, the curvature $R_s$ is bounded below independently of $s$.  Consequently, as $|H| \to 0$, $\psi_s$ converges to the identity map.  So using that $M_b \subset \psi(\intrr(M_c))$, we may take $H$ to be sufficiently small so that $\psi_1$ preserves this inclusion.  Finally, on $M \smallsetminus M_b$, we have that
\[ \overline{\Omega} = \eta+d(H \cdot d \theta) + \pi^* \omega. \]
Define $\overline{\eta} = \eta+d(H \cdot d \theta)$.  By \cref{lem:PerturbMappingTori}, for a generic choice of $H$, $\overline{\eta}$ will satisfy the conditions of the lemma.
\end{proof}

\subsection{Flattening more Hamiltonian fibrations}

We show that every Hamiltonian fibration over $\IC$ can be embedded into the product Hamiltonian fibration over $\IC$.  Let $(M,\Omega)$ be a symplectic manifold.  Let $\pi: M \to \IC$ be a smooth, proper map that is a submersion over all of $\IC$, which we assume gives a Hamiltonian fibration.  Let $F = \pi^{-1}(0)$ denote the fibre of $\pi$ over the origin.  Let $\Omega_F$ denote the restriction of $\Omega$ to $F$.  Let $\omega \coloneqq d(f d\theta)$ be a symplectic form on $(0,\infty) \times S^1$, where $f$ is a smooth function in $r$ that vanishes at the origin and has non-negative second derivative. 

\begin{prop}\label{prop:FlatteningOverC}
There exists a symplectic embedding
\[\varphi: (M,\Omega) \hookrightarrow (F \times \IC , \Omega_F + \omega ),\]
which is a smooth homotopy equivalence and is the identity from $\pi^{-1}(0)$ to $F \times \{0\}$.
\end{prop}

The construction is very similar to the construction of the embedding in \cref{prop:FlatteningOverCStarNonDeg}.  So when elements of the construction are entirely analogous to elements of the construction in \cref{prop:FlatteningOverCStarNonDeg}, we will leave it to the reader to check that the previous arguments carry over.  We now fix a nice system of coordinates and a nice description of $\Omega$.

\begin{notn}\label{notn:FlatteningC}
Fix polar coordinates $(r,\theta)$ for $\IC$.  Since $\pi$ is everywhere a proper submersion, fix a trivialization of $\pi$.  Namely, identify $M$ with  $F \times \IC$ and $\pi$ with the projection to the second component, $\pi: F \times \IC \to \IC$.  Let $\pi_F: F \times \IC \to F$ denote the projection to the first component.  This is a smooth homotopy equivalence with smooth homotopy inverse given by the inclusion of $F$ as $F \times \{0\}$.  So the closed $2$-form $\Omega$ is cohomologous to the pull-back of a closed 2-form $\Omega_F$ on $F$.  We have $\Omega =  \Omega_F + d(\xi)$.  Moreover, $\xi$ may be chosen to be a $1$-form whose restriction to $F$ is zero.  Indeed, if $\xi$ does not vanish along $F$, then restricting $\Omega_F + d(\xi)$ to $F$ gives that $\xi$ restricted to $F$ is closed.  Consequently, we can replace $\xi$ by $\xi - (\xi|_{F})$ and the resulting form will now vanish along $F$, as desired.
\end{notn}

\begin{lem}\label{lem:FlatteningC1}
There exists a function $H: \IR \to \IR$ such that $d(H(r) d \theta)$ is symplectic on $\IC$ and there exists a symplectic embedding
\[ \psi': (F\times \IC, \Omega_F + d(\xi)) \hookrightarrow (F \times \IC, \Omega_F + d(H(r)d \theta)), \]
which is a smooth homotopy equivalence and is the identity from $\pi^{-1}(0)$ to $F \times \{0\}$.
\end{lem}

\begin{proof}
Consider the smooth family of maps $\psi_s: F \times \IC \to F \times \IC$ given by $\psi_s(x,z) = (x,s \cdot z)$.  Notice that $\psi_1$ is the projection to $F$ and $\psi_s(x,0) = (x,0)$ for all $x \in F$.  This gives rise to the family $(\phi_s)^* \Omega$ of cohomologous coupling forms (with respect to $\pi$).  Let $R_s$ denote the curvature of the coupling form $(\phi_s)^* \Omega$ with respect to $\omega$.  Analogous to the argument in \cref{lem:Flattening1}, $R_s \geq 0$ with $R_s = 0$ if and only if $s = 1$.  We consider the family of cohomologous symplectic forms
\[ \Omega_s \coloneqq \phi_s^*\Omega + s \cdot d(H(r)d\theta) \]
for some radial function $H$ to be determined.  The curvature of the coupling forms $\Omega_s$ is
\[ R_s + s \cdot \frac{H'(r)}{f'(r)}. \]
As in \cref{lem:Flattening1}, we apply Moser's argument to obtain our symplectic embedding.  We have
\[ \frac{d}{ds}\left(\Omega_s\right) = d \left( \frac{d}{ds} \left( (\phi_s)^* \xi \right) + H \cdot d\theta \right) \eqqcolon d(\xi_s).\]
Define the vector field $V_s$ by $\Omega_s(V_s,\cdot) = -\xi_s$.
Let $\psi_s$ denote the time $s$ flow of $V_s$.  For a particular choice of function $H$, the flow $\psi_s$ is well-defined and is the identity on $F \times \{0\}$.  To see this, one constructs a sufficiently large function $H$ as in the proof of \cref{lem:Flattening2} that vanishes at the origin.  As in \cref{lem:Flattening2}, this will ensure that the flow $\psi_s$ is well-defined.  To see that $\psi_s$ fixes $F \times \{0\}$ for all $s$, observe that $\xi_s|_{F \times \{0\}} = \xi|_{F \times \{0\}} = 0$.  So $V_s$ vanishes along $F \times \{0\}$ and $\psi_s$ is the identity along $F \times \{0\}$.  Setting $\psi' \coloneqq \psi_1$ proves the lemma.
\end{proof}

\begin{lem}\label{lem:FlatteningC2}
Let $f$ be as above.  Given $H: \IR \to \IR$ such that $d(H(r)d\theta)$ is symplectic on $\IC$ and $H(0) = 0$, there exists a symplectic embedding
\[ \psi'': (F\times \IC, \Omega_F+d(H(r)d\theta)) \hookrightarrow (F \times \IC, \Omega_F + d(f d\theta)) \]
such that $\psi''$ is a smooth homotopy equivalence that is the identity from $\pi^{-1}(0)$ to $F \times \{0\}$.
\end{lem}

\begin{proof}
We have a smooth family of symplectic forms
\[ \Omega_s = \Omega_F + d( ( (1-s) \cdot H + s \cdot f) \cdot d\theta)\]
for $0 \leq s \leq 1$.  By \cref{lem:HorizontalLiftOfHamiltonianVectorField} and the discussion in the proof of \cref{lem:1FormDeterminesMappingTorus}, the vector field $V_s$ determined by $\Omega_s(V_s,\cdot) = (H-f) \cdot d\theta$ is
\[ V_s = \frac{H-f}{(1-s) \cdot H' + s \cdot f'} \cdot \partial_r.\]
Let $\psi_s$ denote the time $s$ flow along $V_s$.  To see that the flow $\psi_s$ is well-defined, one argues almost identically to the proof of \cref{lem:Flattening3}.  To see that $\psi_s$ is the identity on $F \times \{0\}$, one observes that $V_s$ vanishes along $F\times \{0\}$ and thus the flow is the identity along $F \times \{0\}$, as desired.  Setting $\psi'' \coloneqq \psi_1$ proves the lemma.
\end{proof}

\cref{prop:FlatteningOverC} now follows immediately from composing the embeddings constructed in \cref{lem:FlatteningC1} and \cref{lem:FlatteningC2}.


\part{Gromov-type compactness results}

In this part, we discuss Gromov-type compactness results that are necessary for our main results.  We with proving a compactness result for sequences of holomorphic curves with boundaries that arise in the proof of our main result; namely, the sequences of curves that degenerate to the center fibre of our degeneration to the normal cone.  We also prove a Gromov-Floer-type compactness result for sequences of Floer trajectories associated to sequences of Hamiltonians that uniformly converge to the zero Hamiltonian on a compact subset.  This is the key ingredient that allows us to produce holomorphic curves from Floer trajectories.


\section{Compactness for degenerations \`{a} la Fish}\label{sec:CompactnessForDegenerationsALaFish}

We establish a compactness result (specific to our setting) for holomorphic curves with boundaries.

\begin{notn}
Let $(Q,\Omega,J)$ be an almost K\"{a}hler, symplectic manifold with boundary such that $\intrr(Q)$ is an open complex analytic manifold.  Let $\pi_Q: Q \to \IC$ be a proper, surjective map with $\pi_Q^{-1}(0) = F \cup E$, where $E$ is a compact subset of $Q$, and $F$ is a (possibly singular) proper, complex analytic subscheme of the interior of $Q$ that is not strictly contained in $E$.

Let $M$ be a smooth quasi-projective variety with a smooth, projective compactification $\overline{M}$.  We assume that there exists a birational map $\varphi: F \to \overline{M}$ that gives an isomorphism $\varphi: (F \smallsetminus E \cap F) \to M$.
\end{notn}

To orient the reader, $Q$ will be a ``trimming'' of the degeneration of the normal cone of $\pi^{-1}(\infty)$ to $\overline{M}$ in \cref{lem:SequenceOfCurves}.  The below lemma is inspired by a result of McLean \cite[Lemma 4.6]{McLean_DukeMathJournal}.  As with McLean's result, \cref{lem:FishCompactness} relies on a compactness result due to Fish \cite[Theorem A]{Fish_TargetLocalGromovCompactness}.

\begin{lem}\label{lem:FishCompactness}
Suppose that there is a sequence of (possibly disconnected) genus zero, compact, holomorphic curves $u_\nu: \Sigma_\nu \to Q$ with non-empty boundaries\footnote{We do not place any constraints on the topological type of the boundaries.} that satisfy:
	\begin{itemize}
		\item $\pi_Q \circ u_\nu \equiv z_\nu \in \IC$ with $z_\nu \to 0$,
		\item $u_\nu(x_\nu) \to p \in F \smallsetminus E \cap F$ for some $x_\nu \in \Sigma_\nu$,
		\item $u_\nu(\partial \Sigma_\nu) \subset \partial Q$,
		\item $u_\nu(\Sigma_\nu)$ is connected, and
		\item the energies of the $u_\nu$ (with respect to $\Omega$ and $J$) are uniformly bounded in $\nu$.
	\end{itemize}
There exists a genus zero, compact, holomorphic curve $u: \Sigma \to \overline{M}$ with empty boundary such that $u(\Sigma)$ is connected, and the image of $u$ intersects $\varphi(p)$ and $\overline{M} \smallsetminus M$.
\end{lem}

\begin{proof}
Fix a background metric on $Q$ with distance function $\dist$.  After passing to a subsequence of the $u_\nu$, assume the images of the $u_\nu$ lie in a fixed compact subset of $Q$.  So by \cite[Theorem A]{Fish_TargetLocalGromovCompactness}, there exists a subsequence of the $u_\nu$, $\varepsilon>0$, and a dense open subset $\sI \subset [0,\varepsilon)$ such that for each $\delta \in \sI$, the sequence of curves $u_\nu^\delta: \Sigma_\nu^\delta \to Q^\delta$, where
 \[ Q^\delta = \left\{ q \in Q \mid \dist(p,\partial Q) \geq \delta \right\} \]
 and
 \[ \Sigma_\nu^\delta \coloneqq  \left\{ x \in \Sigma_\nu \mid \dist(u_\nu(z), \partial Q) \geq \delta \right\},\]
 Gromov converges (see \cite[Definition 2.11]{Fish_TargetLocalGromovCompactness}).  For each $\delta \in \sI$ and each $\nu$, $\Sigma_\nu^\delta$ is a compact surface with boundary.  Now fix $\delta \in \sI$ such that the $u_{\nu}^\delta$ Gromov converges and such that $F$ is a proper subset of $\intrr(Q^\delta)$.  Since $u_\nu(\Sigma_\nu)$ is connected for each $\nu$, there exists a (possibly non-unique) subset of connected components ${\Sigma_\nu^\delta}' \subset \Sigma_\nu^\delta$ such that $u_\nu^\delta({\Sigma_\nu^\delta}')$ is connected and intersects both $\widetilde{p}$ and $\partial Q^\delta$.  After passing to a subsequence, assume that the subsequence $u_\nu^\delta|_{{\Sigma_\nu^\delta}'}: \Sigma_\nu^{\delta'} \to Q^\delta$ Gromov converges.  Denote the limiting stable holomorphic curve by $v: C \to Q^\delta$.

The map $v$ will not satisfy the conclusions of the lemma.  However, restricting $v$ to a subdomain will yield the desired curve.  To this end, denote the associated map with ``smoothed'' domain by $\widetilde{v}^r: \widetilde{C}^r \to Q^\delta$ (see \cite[Discussion before Definition 2.11]{Fish_TargetLocalGromovCompactness}).  By Gromov convergence, there exist diffeomorphisms $\phi_\nu: {\Sigma_\nu^\delta}' \to \widetilde{C}^r$ such that the maps $v_\nu \coloneqq u_\nu^\delta \circ \phi_\nu^{-1}: \widetilde{C}^r \to Q^\delta$ converge to $\widetilde{v}^r$ in $C^0(\widetilde{\Sigma}^r)$.  Using this and our hypotheses, we deduce that
\begin{enumerate}
 	\item $v(C)$ is contained in $E \cup F$,
	\item $v(\partial C) \subset E \smallsetminus (E \cap F)$, 
	\item $v(C) \cap E \neq \varnothing$,
	\item there exists $x \in C \smallsetminus \partial C$ such that $v(x) = p$, and 
	\item $v(C)$ is connected.
\end{enumerate}

Now we define our desired curve.  By items (i), (iii), (iv), and (v) above, there exists a sequence of connected components $C_0,\dots,C_k \subset C$ such that $x \in C_0$, $v(C_i) \cap v(C_{i+1}) \neq \varnothing$, and $v(C_k) \cap E \neq \varnothing$.

Since $F$ is a proper, complex analytic subscheme of $\intrr(Q)$ and $v: C \to F \cup E$ is a holomorphic map, if $v(C_i) \cap (F \smallsetminus (E \cap F)) \neq \varnothing$, then $v(C_i) \subset F$.  Let $0 \leq \ell \leq k$ be the smallest integer such that $v(C_\ell) \cap (F \smallsetminus (E \cap F)) \neq \varnothing$ and $v(C_\ell) \cap (F \cap E) \neq \varnothing$.  So $C_\ell$ is the first component in the sequence that is completely contained in $F$ and non-trivially intersects $E$.  Given this $\ell$, if $i \leq \ell$, then $v(C_i) \subset F$.  Set $\Sigma \coloneqq \sqcup_{i=0}^\ell C_i$ and define $u: \Sigma \to \overline{M}$ by composing $v|_{\Sigma}$ with $\varphi: F \to \overline{M}$.  

It remains to show that $u: \Sigma \to \overline{M}$ satisfies the conclusions of the lemma.  To see that $u$ defines a genus zero, compact, holomorphic curve with empty boundary, we proceed as follows.  First, $\Sigma$ is compact by construction.  Second, the genus is zero by the classification of surfaces with boundaries.  Third, to see that the boundary of $\Sigma$ is empty, argue as follows:  Since $\partial \Sigma \subset \partial C$, $v|_{\Sigma}(\partial \Sigma) \subset v(\partial C) \subset E \smallsetminus (E \cup F)$ (using item (ii) above).  But $v|_{\Sigma}(\Sigma) \subset F$.  So $v|_{\Sigma}(\partial \Sigma) = \varnothing$, that is, $\partial \Sigma = \varnothing$.  Finally, that $u(\Sigma)$ is connected and intersects $\varphi(p)$ and $\overline{M} \smallsetminus M$ is immediate from our construction.  This completes the proof.
\end{proof}


\section{Gromov-Floer compactness with turning off Hamiltonians}\label{sec:Pearls}

Let $(M,\Omega,\lambda)$ be a convex symplectic domain with admissible almost complex structure $J$.  We also assume that the boundary of $M$ has non-degenerate Reeb vector field.  We consider how to produce holomorphic curves from sequences of Floer trajectories associated to Hamiltonians that become zero on the interiors of domains via a Gromov-Floer-type compactness result.  Consider a degenerating sequence of Hamiltonians.

\begin{notn}
Let $H_\nu$ be a sequence of $(a-\delta+\varepsilon)$-radially admissible Hamiltonians on $M$ with
\begin{itemize}
	\item $H_\nu = H_{\nu+1}$ for $a-\delta+\varepsilon \geq r(x)$ , and
	\item $H_\nu = \varepsilon_\nu \cdot g$ on $M_{a-\delta}$, where $g$ is a negative $C^2$-small Morse function that is Morse-Smale with respect to $J$ and $\lim_{\nu \to \infty} \varepsilon_\nu = 0$.
\end{itemize}
Let $H_\infty$ be the limiting Hamiltonian, which we assume to be smooth.  There is an obvious identification of the orbits of $H_\nu$ with the orbits of $H_{\nu+1}$.  So we do not distinguish between the orbits of these different Hamiltonians.
\end{notn}

Given $H_\infty$, we define a notion of a Morse-Bott broken Floer trajectory, which is an appropriate gluing of Floer trajectories associated to $(H_\infty,J)$ with segments of negative gradient trajectories of $g$.  Similar definitions appear in \cite{DiogoLisi_MorseBottSplitSymplecticHomology} and \cite{BourgeoisOancea_Autonomous}.  The main difference is that our Hamiltonian $H_\infty$ is non-degenerate in the collar neighborhood of the boundary.  We added in perturbations near our $S^1$-families of orbits (see \cref{rem:MorseBottBreakings}).  Consequently, we do not need to consider ``cascades'' that arise when one ``turns off'' the Hamiltonian perturbation term about these orbits.

\begin{defn}
A \emph{Morse-Bott broken Floer trajectory} of type $((H_\infty,J,g),x_-,x_+)$ is a tuple $(\gamma^0,u^1,\gamma^1,\dots,u^k, \gamma^k, u^{k+1},u^{k+2},\dots,u^m)$ where
\begin{enumerate}
	\item $u^i: C_i \to M$ is a smooth map with each $C_i$ being a nodal curve of type $(0,2)$ with punctures $p_i^{\pm}$ both contained in the same irreducible component, denoted $C_i^0$, and
	\item $\gamma^i: I_i \to M$ is a smooth map with $I_i$ being a closed, connected subset of $\IR$\footnote{In other words, $I_i$ is either of the form $[a,b]$, $[a,+\infty)$, or $(-\infty,b]$ for some constants $a$ and $b$.}.
\end{enumerate}
These maps satisfy
\begin{enumerate}
	\item $\lim_{s \to \pm \infty} u^i|_{C_i^0} = x_{\pm}^i$ are orbits of $H_\infty$,
	\item $\lim_{s \to -\infty} \gamma^0(s) = x_-$,
	\item $\lim_{s \to +\infty} u^m(s,t) = x_+$,
	\item $\underline{ev}(\gamma^i) = x_+^i$ for $1 \leq i \leq k$,
	\item $\overline{ev}(\gamma^i) = x_-^{i+1}$ for $0 \leq i \leq k$, and
	\item $x_+^i = x_{-}^{i+1}$ for $k+1 \leq i \leq m-1$,
\end{enumerate}
and
\begin{enumerate}
	\item $\overline{\partial}_{(H_\infty,J)} u^i|_{C_i^0} \equiv 0$,
	\item $u^i|_{C_i \smallsetminus C_i^0}$ is holomorphic, and
	\item $-\nabla g \circ \gamma^i = (\gamma^i)'$.
\end{enumerate}
\end{defn}

\begin{figure}[h]
\centering
\includegraphics[width=.8\linewidth]{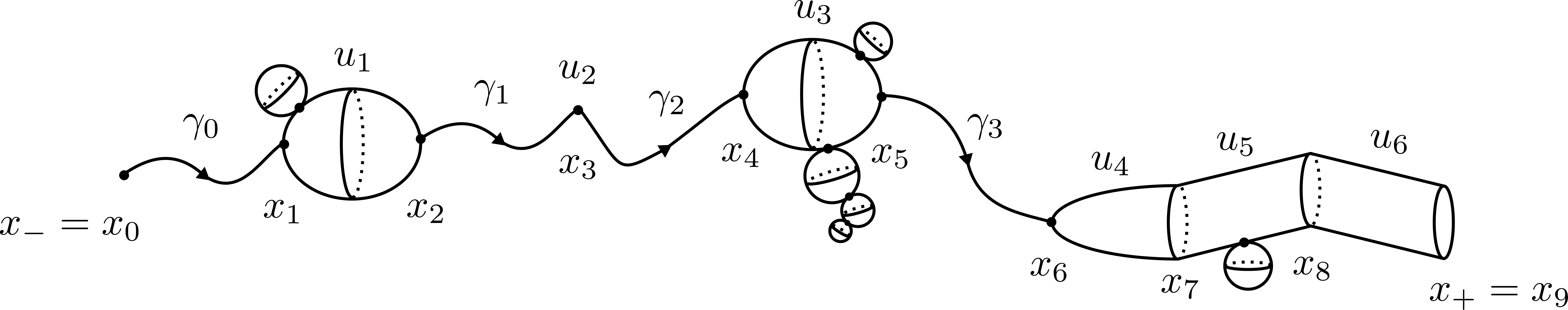}
\caption{A heuristic picture of a Morse-Bott broken Floer trajectory.}
\label{fig:MorseBottTrajectory}
\end{figure}

We want the following compactness result.

\begin{prop}\label{prop:ConvergenceToPearls}
Let $u_\nu \in \overline{\sM}((H_\nu,J),x_-,x_+)$ be a sequence of Floer trajectories with energies uniformly bounded by a constant $E$, and suppose that each $u_\nu$ has a single irreducible component.  There exists a Morse-Bott broken Floer trajectory of type $((H_\infty,J), x_-,x_+)$, $(\gamma^0,u^1,\gamma^1,\dots,u^k, \gamma^k, u^{k+1},u^{k+2},\dots,u^m)$, and shifts $s_\nu^i \in \IR$ so that $u_\nu(\cdot + s_\nu^i,\cdot)$ converges to $u^i$ in the Gromov-Floer sense, and $\sum_{i=0}^m E(u^i) = E$.
\end{prop}

To drop the assumption that $u_\nu$ has a single irreducible component and still obtain a Morse-Bott broken Floer trajectory, one simply inducts, keeping track of bubbles.  The proof of \cref{prop:ConvergenceToPearls} is long and slightly tedious.  We give it as a sequence of claims.  Now we fix notation.

\begin{notn}
Fix a background metric on $M$ with injectivity radius less than $1$.  Let $dist$ denote its associated distance function.  Let
\[ B_c(x) = \{ y \in M \mid dist(x,y) < c \} \]
denote the open metric ball of radius $c$ about a $1$-periodic orbit $x$ of $H_\nu$.  Without loss of generality, we may assume that
\begin{itemize}
	\item the $B_1(x)$ are pairwise disjoint among $1$-periodic orbits of $H_\nu$, and
	\item if $r(x) \geq a- \delta +\varepsilon$, then $r(B_1(x)) \geq a - \delta +\varepsilon$.
\end{itemize}
\end{notn}

We construct the Morse-Bott broken Floer trajectory inductively.  We will start at the orbit $x_+$ and construct our way down to the orbit $x_-$.

\begin{lem}\label{lem:ConvergenceToPearls1}
If $r(x_+) \geq a - \delta+\varepsilon$, then there exists sequences of shifts $s_\nu^i$ for $1 \leq i \leq \ell$ (for some $\ell$) such that (after extracting a subsequence) $u_\nu(\cdot+ s_\nu^i,\cdot)$ converges in the Gromov topology to an element $u^i$ of $\overline{\sM}((H_\infty,J),x_-^i,x_+^i)$, where $x_-^i$ and $x_+^i$ are $1$-periodic orbits of $H_\infty$ that satisfy:
\begin{enumerate}
	\item $x_+^1 = x_+$,
	\item $x_-^{i} = x_+^{i+1}$ for $1 \leq i \leq \ell-1$,
	\item $r(x_+^i) \geq a - \delta+\varepsilon$ for $1 \leq i \leq \ell$,
	\item $r(x_-^i) \geq a - \delta+\varepsilon$ for $1 \leq i \leq \ell-1$, and
	\item $r(x_-^{\ell}) \leq a - \delta$.
\end{enumerate}
\end{lem}

The proof of \cref{lem:ConvergenceToPearls1} follows the classical proof of Floer's original global compactness result \cite[Proposition 3b]{Floer_FixedPoints}.  We closely follow the proof in \cite[Theorem 9.1.7]{AudinDamian_MorseTheoryAndFloerHomology} in order to fix notation and establish some properties for our sequences of shifts, both of which will be used later.

\begin{proof}
Define
\[ s_\nu^1 \coloneqq \sup \{ s \mid u_\nu(s, \cdot) \in M \smallsetminus B_1(x_+) \}. \]
Consider $u_\nu(\cdot + s_\nu^1,\cdot)$.  After extracting a subsequence, it converges in the Gromov topology to an element $u^1$ of $\overline{\sM}((H_\infty,J),x_-^1,x_+^1)$ by Floer's original compactness theorem \cite[Proposition 3c]{Floer_FixedPoints}.  By construction,
\[u_\nu(s+s_\nu^1, \cdot) \cap B_1(x_+) \neq \varnothing\]
for all $s > 0$.  Also, $u_\nu(s_\nu^1,t) \in \partial B_1(x_+)$ for some $t$.  So $u^1$ must satisfy:
\begin{itemize}
	\item $u^1(s,\cdot) \cap B_1(x_+) \neq \varnothing$ for $s>0$, and
	\item $u^1(0,\cdot) \cap \partial B_1(x_+) \neq \varnothing$.
\end{itemize}
Since the $1$-periodic orbits of $H_\infty$ are isolated for $r \geq a - \delta + \varepsilon$,
\[ x_+^1 = \lim_{s \to +\infty} u^1(s,t) = x_+.\]
If $r(x_-^1) \leq r - \delta$, then we are done.  If not, then $r(x_-^1) \geq a - \delta + \varepsilon$, and we iterate the following argument below until we do obtain a $1$-periodic orbit $x_-^\ell$ with $r(x_-^\ell) \leq a - \delta$.

Since $\lim_{s \to -\infty} u^1(s,t) = x_-^1$, there exists $s^-$ so that $u^1(s,\cdot) \in B_1(x_-^1)$ for all $s \leq s^-$.  Since $u_\nu(s^- +s_\nu^1,\cdot)$ converges to $u^1(s^-,\cdot)$, we have that
\[ u_\nu(s^-+s^1_\nu,\cdot) \in B_1(x_-^1),\]
for $\nu \gg 0$.  Since $x_-^1 \neq x_-$, $u_\nu$ must exit $B_1(x_-^1)$ for $s\ll s^-+s_\nu^1$.  So define
\[ s_\nu^2 \coloneqq \sup \{ s \leq s^- + s_\nu^1 \mid u_\nu(s,\cdot) \in M \smallsetminus B_1(x_-^1)\}. \]
Consider $u_\nu(\cdot + s_\nu^2, \cdot)$.  After extracting a subsequence, this sequence converges in the Gromov topology to an element $u^2$ of $\overline{\sM}((H_\infty,J),x_-^2,x_+^2)$.  We claim that $x_+^2 = x_-^1$, and $x_-^2 \neq x_-^1$.

We first observe that $s_\nu^1 - s_\nu^2 \to + \infty$.  Indeed, suppose by way of contradiction that $s_\nu^1 - s_\nu^2 \leq C$ for some constant $C$.  Then $[s_\nu^2 - s_\nu^1,s^-] \subset [-C,s^-]$.  Since $u_\nu(\cdot + s_\nu^1,\cdot) \to u^1$ uniformly on all compact subsets modulo bubbling, it does so on $[-C,s^-]\times S^1$.  By definition of $s^-$, $u^1([-C,s^-],\cdot) \subset B_1(x_-^1)$.  So for $\nu \gg 0$,
\[ u_\nu( [s_\nu^2,s^-+s_\nu^1], \cdot) \subset u_\nu([-C,s^-] + s_\nu^1,\cdot)  \subset B_1(x_-^1).   \]
In particular, $u(s_\nu^2,\cdot) \subset B_1(x_-^1)$, which is a contradiction since $u_\nu(s_\nu^2,\cdot) \cap \partial B_{1}(x_-^1) \neq \varnothing$.

Now for fixed $s>0$, there exists $\nu \gg 0$ such that
\[ s_\nu^2 \leq s_\nu^2+s \leq s^-+s_\nu^1.\]
So for this fixed $s>0$ and $\nu\gg0$,
\[ u_\nu(s_\nu^2+s,\cdot) \cap B_1(x_-^1) \neq \varnothing. \]
So $u^2(s,\cdot) \cap B_1(x_-^1) \neq \varnothing$ for all $s>0$.  Since $x_-^1$ is an isolated orbit,
\[ x_+^2 = \lim_{s \to +\infty} u^2(s,t) = x_-^1. \]
Also, $u^2(0,\cdot) \cap \partial B_1(x_+^2) \neq \varnothing$.  So $u^2$ must exit this ball, and, consequently, $x_-^2 \neq x_+^2$.

We iterate this argument to obtain shifts $s_\nu^i$ for $1 \leq i \leq \ell$ and Floer trajectories $u^i \in \overline{\sM}((H_\infty,J),x_-^i,x_+^i)$ that satisfy the conclusions of the lemma.
\end{proof}

Given \cref{lem:ConvergenceToPearls1}, either $x_-^\ell = x_-$, or $r(x_-^\ell) \leq a - \delta$.  In the former case, we are done.  So we assume that $r(x_-^\ell) \leq a - \delta$ and continue our argument.  Now our argument starts to differ from the standard argument presented in \cite{AudinDamian_MorseTheoryAndFloerHomology}.

\begin{lem}\label{lem:ConvergenceToPearls2}
There exists a sequence of shifts $s_\nu^{\ell+1}$ such that (after extracting a subsequence) $u_\nu(\cdot+ s_\nu^{\ell+1},\cdot)$ convergences in the Gromov topology to an element $u^{\ell+1}$ of $\overline{\sM}((H_\infty,J),x_-^{\ell+1},x_+^{\ell+1})$, where $x_-^{\ell+1}$ and $x_+^{\ell+1}$ are $1$-periodic orbits of $H_\infty$ with
\begin{enumerate}
	\item $s_\nu^{\ell} - s_\nu^{\ell+1} \to +\infty$,
	\item $x_+^{\ell+1} \in \overline{B_1(x_-^{\ell})}$, and
	\item $r(u^{\ell+1}) \leq a - \delta$.
\end{enumerate}
\end{lem}

\begin{proof}
As in \cref{lem:ConvergenceToPearls1}, since $\lim_{s \to -\infty} u^{\ell}(s,t) = x_-^\ell$, there exists $s_n^-$ such that $u^\ell(s,t) \in B_{1/n}(x_-^\ell)$ for all $s \leq s_n^-$.  Since $u_\nu(s_n^-+s_\nu^\ell,t)$ converges to $u^\ell(s_n^-,t)$,
\[ u_\nu(s_n^-+s_\nu^\ell,t) \in B_{1/n}(x_-^\ell)\]
for $\nu\gg0$.  Since $\lim_{s \to -\infty} u_\nu(s,t) = x_- \neq x_-^\ell$, $u_\nu$ must exit $B_{1/n}(x_-^\ell)$ for $s \ll  s_n^-+s_\nu^\ell$.  So define
\begin{equation}\label{eqn:snun}
s_{\nu,n}^{\ell+1} = \sup\{s \leq s^-+s_\nu^\ell \mid u_\nu(s,\cdot) \in B_{1/1}(x_-^\ell)\}
\end{equation}
and set $s_\nu^{\ell+1} \coloneqq s_{\nu,1}^{\ell+1}$.  Consider the sequence of Floer trajectories $u_\nu(\cdot+s_\nu^{\ell+1},\cdot)$.  After extracting a subsequence, this sequence converges in the Gromov topology to an element $u^{\ell+1}$ of $\overline{\sM}((H_\infty,J),x_-^{\ell+1},x_+^{\ell+1})$.  To wrap up the proof, we make remarks.
\begin{enumerate}
	\item As in \cref{lem:ConvergenceToPearls1}, $s_\nu^\ell - s_\nu^{\ell+1} \to +\infty$.
	\item As in \cref{lem:ConvergenceToPearls1}, $u^{\ell+1}(s) \in \overline{B_1(x_-^\ell)}$ for $s \geq 0$.  In particular, $x_+^\ell \in \overline{B_1(x_-^\ell)}$.  However, since the orbits of $H_\infty$ are non-isolated in $r \leq a -\delta$, we can not conclude that $x_+^{\ell+1}$ is $x_-^\ell$.
	\item A priori, $u^{\ell+1}$ could be constant.  However, if it is constant, then $x_-^{\ell+1} \neq x_-^\ell$ because $u^{\ell+1}(0,\cdot)$ meets $\partial B_1(x_-^\ell)$.
	\item Since the positive end of $u^{\ell+1}$ is contained in $\{r \leq a - \delta\}$, \cref{lem:TopologicalFiltration} implies that all of $u^{\ell+1}$ is contained in $\{r \leq a - \delta\}$, as desired.
\end{enumerate}
\end{proof}

We now construct a possibly broken negative gradient fragment of $g$ that starts at $x_+^{\ell+1}$ and ends at $x_-^\ell$.

\begin{lem}\label{lem:ConvergenceToPearls3}
There exists a sequence of shifts $s_\nu^{top}$ such that (after extracting a subsequence) $u_\nu(\cdot+ s_\nu^{\ell+1},\cdot)$ converges in the Gromov topology to the constant curve at $x_+^{\ell+1}$.
\end{lem}

\begin{proof}
We will need to extract a subsequence and relabel our sequence in order to construct the $s_\nu^{top}$.  First, if $u^{\ell+1}$ is constant, then we are done and set $s_\nu^{top} = s_\nu^{\ell+1}$.  Now assume that $u^{\ell+1}$ is non-constant.  

As before, there exists constants $s_n^+>0$ such that $u^{\ell+1}(s,t) \in B_{1/n}(x_+^{\ell+1})$ for all $s \geq s_n^+$, and there exists $\nu_n$ such that if $\nu \geq \nu_n$, then
\[ u_\nu(s_\nu^{\ell+1}+s_n^+,t) \in B_{1/n}(x_+^{\ell+1}). \]
Since $s_\nu^{\ell} - s_\nu^{\ell+1} \to \infty$, we may assume that if $\nu \geq {\nu_n}$, then $s_\nu^\ell - s_\nu^{\ell+1} > n+s_n^+$.  Now diagonalize and relabel.  Set $\widetilde{u}_\nu \coloneqq u_{\nu_\nu}$ and $\widetilde{s}_\nu^{i} \coloneqq s_{\nu_\nu}^i$ for all $i$.  Notice that 
\begin{itemize}
	\item $\widetilde{u}_\nu(\cdot+\widetilde{s}_{\nu}^i, \cdot)$ converges to $u^i$,
	\item $\widetilde{u}_\nu(\widetilde{s}_\nu^{\ell+1}+s_\nu^+,\cdot) \in B_{1/\nu}(x_+^{\ell+1})$,
	\item $\widetilde{s}_\nu^\ell - (\widetilde{s}_\nu^{\ell+1}+{s}_\nu^+) \to \infty$, and
	\item ${s}_\nu^+ - \widetilde{s}_\nu^{\ell+1} \to \infty$\footnote{Indeed, if $s_\nu^+ - s_{\nu_\nu}^{\ell+1}$ is bounded, then $u^{\ell+1}$ would have to be constant, which we have assumed not to be the case. }.
\end{itemize}
Write $u_\nu$ for $\widetilde{u}_\nu$ and $s_\nu^i$ for $\widetilde{s}_\nu^i$.  We define $s_\nu^{top} \coloneqq s_\nu^{\ell+1}+s_\nu^{+}$.  Consider $u_\nu(\cdot+s_\nu^{top},\cdot)$.  After extracting a subsequence, this converges in the Gromov topology to some $u^{top}$.  For each $s$ fixed,
\[ s^-+s_\nu^\ell > s+s_\nu^{top} \geq s_\nu^{\ell+1}\]
for $\nu \gg 0$, where $s^-$ is as in \cref{lem:ConvergenceToPearls2}.  It follows that $u^{top}(s) \in B_1(x_-^\ell)$ for all $s$ and, consequently, $u^{top}$ is constant.  In particular, $u^{top}$ converges to the constant curve at $x_+^{\ell+1}$.
\end{proof}

\begin{lem}\label{lem:ConvergenceToPearls4}
For each $n$ sufficiently large, there exists a sequence of shifts $s_{\nu,n}^{\ell+1}$ such that (after extracting a subsequence) $u_\nu(\cdot+ s_{\nu,n}^{\ell+1},\cdot)$ converges in the Gromov topology to a constant curve $v^n$ at a point $x_n \in \partial B_{1/n}(x_-^{\ell})$.  Moreover, $s_{\nu,n}^{\ell+1} - s_\nu^{top} \to \infty$, and $s_\nu^\ell - s_{\nu,n}^{\ell+1} \to \infty$.
\end{lem}

\begin{proof}
Recall the shifts $s_{\nu,n}^{\ell+1}$ from \cref{eqn:snun}.  Fix $N\gg0$ such that $dist(x_+^{\ell+1},x_-^\ell) > 1/N$.  Define $v^n$ to be the Gromov limit of the curves $u_\nu(s_{\nu,n}^{\ell+1}+\cdot, \cdot)$ for $n \geq N$.  We first show that for each $n$, $s_{\nu,n}^{\ell+1} - s_\nu^{top} \to \infty$, and $s_\nu^\ell - s_{\nu,n}^{\ell+1} \to \infty$.  
The second item follows from the same argument used to show that $s_\nu^\ell - s_\nu^{\ell+1} \to \infty$.  For the first item, suppose by way of contradiction that $0 \leq s_{\nu,n}^{\ell+1} - s_\nu^{top} \leq C$.  Since $u_\nu(\cdot+ s_\nu^{top},\cdot)$ converges to $x_+^{\ell+1}$ on all compact subsets, it does so on $[0,C]$.  So
\[v^n(0,\cdot) = \lim_\nu u_\nu(s_{\nu,n}^{\ell+1}-s_\nu^{top}+s_\nu^{top},\cdot) = x_+^{\ell+1}.\]
But $v^n(0,\cdot)$ meets the boundary of $B_{1/n}(x_-^\ell)$, a contradiction.

As before, $v^n(s,t)$ is contained in $B_{1/n}(x_-^\ell)$ for $s>0$; however, we further claim that $v^n(s)$ is contained in $B_1(x_-^\ell)$ for $s < 0$.  For fixed $s<0$ and $\nu\gg0$,
\[ s_\nu^\ell+s_n^- \geq s_{\nu,n}^{\ell+1} \geq s_{\nu,n}^{\ell+1}+s \geq s_\nu^{\ell+1} + s_\nu^+ \geq s_\nu^{\ell+1}. \]
It follows that $u_\nu(s+s_{\nu,n}^{\ell},t) \in B_1(x_-^{\ell})$.  Consequently, $v^n$ must be constant and contained in the boundary of $B_{1/n}(x_-^\ell)$.
\end{proof}

\begin{lem}\label{lem:ConvergenceToPearls5}
There exists a possibly broken negative gradient fragment of $g$ that starts at $x_+^{\ell+1}$ and ends at $x_-^\ell$ with length $\lim_n \lim_{\nu} \varepsilon_\nu \cdot (s_\nu^{top} - s_{\nu,n}^{\ell+1})$.
\end{lem}

\begin{proof}
We will prove that there exists a possibly broken negative gradient fragment of $g$ from $x_+^{\ell+1}$ to $x_n$.  As the $x_n$ converge to $x_-^\ell$, a diagonalization produces the desired trajectory from $x_+^{\ell+1}$ to $x_-^{\ell}$.

The energies of the $u_\nu(\cdot+s_{\nu,n}^{\ell+1},\cdot)$ over the intervals $[s_\nu^{top} - s_{\nu,n}^{\ell+1},0]$ converge to zero since their Gromov limits over these intervals are constant.  So by the work in \cite{Oh_SpectralInvariants} (see \cite[Theorem 4.3]{OhZue_ThickThin} for a statement), there exists a possibly broken negative gradient fragment of $g$ from $x_+^{\ell+1}$ to $x_n$ of length $\lim_{\nu} \varepsilon_\nu \cdot (s_\nu^{top} - s_{\nu,n}^{\ell+1})$.
\end{proof}

One inducts upon the arguments in \cref{lem:ConvergenceToPearls2}, \cref{lem:ConvergenceToPearls3}, \cref{lem:ConvergenceToPearls4}, and \cref{lem:ConvergenceToPearls5} to obtain the desired Morse-Bott broken Floer trajectory in \cref{prop:ConvergenceToPearls}.


\part{Appendix}


\section{Algebraic Preliminaries}\label{sec:AppendixHA}

\subsection{Chain complexes and conventions}\label{subsec:Conventions}

We fix conventions for chain complexes and mapping telescopes.

\begin{defn}\label{defn:ChainComplex}
Let $\IK$ be a unital, commutative ring of characteristic zero.  Let $\Ch(\IK)$ denote the category whose objects are $\IZ$-graded chain complexes over $\IK$ and whose morphisms are chain maps.  We use cohomological grading conventions for our chain complexes.  So the differentials increase the grading by $+1$.
\end{defn}

\begin{defn}\label{defn:Rays}
A \emph{ray of chain complexes} (a \emph{ray}) is a sequence of chain maps
\[ \sA = \left\{ \xymatrix{ A_0 \ar[r] ^{\alpha_0} & A_1 \ar[r]^{\alpha_1} & A_2 \ar[r]^{\alpha_2} & \cdots} \right\}.\]
\end{defn}

\begin{defn}\label{defn:MappingTelescope}
The \emph{mapping telescope} of a ray $\sA$
is the chain complex with terms
\[ \Tel(\sA) = \bigoplus_{i \in \IN} A_{i}[1] \oplus A_{i} \]
and differential
\begin{align*}
d(\dots, (a_i',a_i),\dots) & = (\dots,d_{A_i[1]}(a_i'),\Ione[1](a_i')+d_{A_i}(a_i)+\alpha_{i-1}[1](a_{i-1}'),\dots)
\\ & = (\dots,-d_{A_i}(a_i'), a_i'+d_{A_i}(a_i)+ \alpha_{i-1}(a_{i-1}'),\dots).
\end{align*}
\end{defn}

\begin{rem}\label{rem:PicDifferentialMappingTelescope}
Pictorially, the differential of a mapping telescope is represented by the following diagram.  Notice that all maps have degree $+1$.
\[ \xymatrix{ A_0 \ar@(ul,ur)^{d_{A_0}} & & A_1 \ar@(ul,ur)^{d_{A_1}} & & A_2 \ar@(ul,ur)^{d_{A_2}} & & \cdots \\
A_0[1] \ar[u]^{\Ione[1]} \ar[urr]^{\alpha_0[1]} \ar@(dl,dr)_{d_{A_0[1]}} & & A_1[1] \ar[u]^{\Ione[1]} \ar[urr]^{\alpha_1[1]} \ar@(dl,dr)_{d_{A_1[1]}} & & A_2[1] \ar[u]^{\Ione[1]} \ar[urr]^{\alpha_2[1]} \ar@(dl,dr)_{d_{A_2[1]}} && \cdots. \\}\]
\end{rem}

\begin{rem}
The mapping telescope of a ray gives a representative for the homotopy colimit of the associated diagram.
\end{rem}

There is a filtration
\[ F^n \Tel(\sA) = \left( \bigoplus_{i=0}^{n-1} A_i[1] \oplus A_i \right) \oplus A_n, \gap \Tel(\sA) = \colim_n F^n \Tel(\sA). \]
We have induced maps $F^n \Tel(\sA) \to A_n$,
\[ \left( \bigoplus_{i=0}^{n-1} \begin{pmatrix} 0 & \alpha_n \circ \cdots \circ \alpha_{i} \\ \end{pmatrix} \right) \oplus \Ione_{A_n}. \]
The maps $F^n \Tel(\sA) \to A_n$ are quasi-isomorphisms and fit into a strictly commutative diagram
\[ \xymatrix{  F^0 \Tel(\sA) \ar[r] \ar[d] & F^1 \Tel(\sA) \ar[r] \ar[d] & F^2 \Tel(\sA) \ar[r] \ar[d] & \cdots \\ A_0 \ar[r]^{\alpha_0} & A_1 \ar[r]^{\alpha_1} & A_2 \ar[r]^{\alpha_2} & \cdots. \\}\]
This induces a map \[ \xymatrix{ \Tel(\sA) \cong \colim_n F^n \Tel(\sA) \ar[r] & \colim_n A_n }.\]
Since colimits commute with cohomology, we can deduced the following.

\begin{lem}\label{lem:TelToColimQI}
The canonical map $\Tel(\sA) \to \colim_n A_n$ is a quasi-isomorphism.\qed
\end{lem}

Using the filtration, we describe how homotopy commutative diagrams of rays induce morphisms of the associated mapping telescopes.  Consider a homotopy commutative diagram of rays
\[ \xymatrix{ A_0 \ar[r] \ar[d]^{\varphi_0} \ar[rd]^{\psi_0} & A_1 \ar[d]^{\varphi_1} \ar[r] \ar[rd]^{\psi_1} & \cdots \\ B_0 \ar[r] & B_1 \ar[r] & \cdots. \\ } \]
This gives rise to a strictly commutative diagram of the filtered complexes
\[ \xymatrix{ F^0 \Tel(\sA) \ar[r] \ar[d]^{F^0 \Tel(\varphi_0)} & F^1 \Tel(\sA) \ar[d]^{F^1 \Tel(\varphi_1)} \ar[r] & \cdots \\ F^0 \Tel(\sB) \ar[r] & F^1 \Tel(\sB) \ar[r] & \cdots \\ } \]
where
\[ F^n \Tel(\varphi_n)(\cdots, a_i',a_i,\cdots) = (\cdots, \varphi_i[1](a_i'), \psi_{i-1}(a_{i-1}')+\varphi_i(a_i),\cdots). \]
Passing to colimits, one obtains the desired morphism of the associated mapping telescopes.

\subsection{Completions of chain complexes over the Novikov Ring}\label{subsec:NovikovCompletions}

We discuss completions of chain complexes over the universal Novikov ring and how mapping telescopes and colimits of rays behave with respect to such completions.

\begin{defn}\label{defn:UniNovikovField}
The \emph{universal Novikov field} over\footnote{For the discussion in this section, one could replace $\IQ$ with any unitial, commutiative ring.  For the discussion of the total exposition, one could replace $\IQ$ with any field of characteristic zero.} $\IQ$ is
\[ \Lambda = \left\{ \sum_{i=0}^\infty q_i T^{R_i} \mid q_i \in \IQ, R_i \in \IR, \mbox{ and for each }R \in \IR \mbox{ the set } \{ i \mid q_i \neq 0, R_i < R\} \mbox{ is finite.}\right\}\]
with addition and multiplication defined as they are for power series.
\end{defn}

\begin{defn}\label{defn:valuationmap}
The \emph{valuation map} $\val: \Lambda \to \IR \cup \{\infty\}$ is
\[ \val \left( \sum_{i=0}^\infty q_i T^{R_i} \right) = \begin{cases} \min_{i} \{ R_i \mid q_i \neq 0 \}, & \sum_{i=0}^\infty q_i T^{R_i} \neq 0 \\ \infty, & \mbox{else}. \\ \end{cases} \]
\end{defn}

\begin{defn}\label{defn:UniNovikovRing}
The \emph{universal Novikov ring} over $\IQ$ is $\Lambda_{\geq 0} \coloneqq \val^{-1}([0,\infty))$.  Define $\Lambda_{\geq R} \coloneqq \val^{-1}([R,\infty))$ and $\Lambda_{> R} \coloneqq \val^{-1}((R,\infty))$, which are ideals in $\Lambda_{\geq 0}$.
\end{defn}

Given a chain complex $A$ over $\Lambda_{\geq 0}$ that is a free $\Lambda_{\geq 0}$-module, the valuation map extends to a valuation map $\val: A \to \IR \cup \{\infty\}$ by
\[ \val(a) = \begin{cases} \min \left\{ R \mid 0 \neq a \in A \otimes_{\Lambda_{\geq 0}} (\Lambda_{\geq 0} / \Lambda_{\geq R} ) \right\}, & a \neq 0 \\ \infty, & \mbox{else}. \\ \end{cases} \]
If $A$ is finitely generated with basis $a_0,\dots,a_k$, the valuation of $a = \sum_i \lambda_i \cdot a_i$ is $\min\{ \val(\lambda_i)\}.$

\begin{defn}\label{defn:CompletionFunctor}
The \emph{completion functor} $\widehat{\cdot}: \Ch(\Lambda_{\geq 0}) \to \Ch(\Lambda_{\geq 0})$ sends chain complexes to their degree-wise completions,
\[ \widehat{A}^\bullet = \lim_{R} A^\bullet \otimes_{\Lambda_{\geq 0}} \Lambda / \Lambda_{\geq R}, \]
with degree-wise completed differentials $\widehat{d}$.
\end{defn}

\begin{rem}\label{rem:CompletionsAndCauchySequences}
It is useful to think of completions in terms of Cauchy sequences.  A \emph{Cauchy sequence} in $A$ is a sequence $(a_1,a_2,\dots)$ of elements of $A$ such that for every $R \geq 0$ there exists $N \in \IN$ so that if $n,m \geq N$, then $\val(a_n-a_m) \geq R$.  Two Cauchy sequences $(a_1,a_2,\dots,)$ and $(a_1',a_2',\dots,)$ are \emph{equivalent} if and only if for every $R \geq 0$ there exists $N \in \IN$ so that if $n,m \geq N$, then $\val(a_n-a_m') \geq R$.  The completion $\widehat{A}$ is given by the set of Cauchy sequences in $A$.  The completion of a map $\varphi: A \to B$ is given by the induced map on Cauchy sequences,
\[\widehat{\varphi}(a_1,a_2,\dots) = (\varphi(a_1),\varphi(a_2),\dots).\]
There is a canonical map $A \to \widehat{A}$ given by the inclusion of constant Cauchy sequences, $a \mapsto \widehat{a} \coloneqq (a,a,a,\dots)$. 

For the case where $A$ is a free $\Lambda_{\geq 0}$-module, we have the following description.
\[ \widehat{A}^\bullet = \left\{ \sum_{j = 0}^\infty \lambda_j \cdot a_j \mid \lambda_i \in \Lambda_{\geq 0},\, a_j \in A^\bullet,\, \mbox{and for each }R \in \IR_{\geq 0} \,\mbox{ the set }\{ i \mid \val(\lambda_i) < R\} \mbox{ is finite} \right\}\]
and
\[ \widehat{d}\left(\sum_{j \in \IN} \lambda_j \cdot a_j\right) = \sum_{j \in \IN} \lambda_j \cdot d(a_j). \]
From this description, it becomes clear that if $A$ is a degree-wise finitely generated, free $\Lambda_{\geq 0}$-module, then the induced map $A \to \widehat{A}$ is an isomorphism, that is, $A$ is complete.
\end{rem}

In general, the functor $\widehat{\cdot}: \Ch(\Lambda_{\geq 0}) \to \Ch(\Lambda_{\geq 0})$ is not exact.  It is a limit and thus the obstruction to exactness is given by the non-vanishing of the associated $\lim^1$ functor.  Nevertheless, we still have the following analogue of \cref{lem:TelToColimQI} for completions, for example, see \cite[Lemma 2.3.7]{Varolgunes_MayerVietorisPropertyForRelativeSymplecticCohomology}.

\begin{lem}\label{lem:TelToColimCompleteQI}
The canonical map $\Tel(\sA) \to \colim \sA$ induces a quasi-isomorphism $\widehat{\Tel}(\sA) \to \widehat{\colim}\sA$. \qed
\end{lem}

\subsection{From completed differentials to uncompleted differentials}\label{subsec:BE}
We discuss how non-trivial differentials in completed complexes give rise to non-trivial differentials in the original (uncompleted) complexes.  This is spelled out in the following two lemmas.

\begin{lem}\label{lem:KillingLemma}
Let $A$ be a chain complex over $\Lambda_{\geq 0}$ with $x \in A$ closed.  If $[\widehat{x}] \in H(\widehat{A}) \otimes \Lambda$ is null-homologous, then there exists $\lambda \in \Lambda_{\geq 0}$ such that for each $R\geq 0$ there exists $y \in A$ with
\[ \val(\lambda \cdot x - d(y)) > R. \]
So if $\val(\lambda \cdot x) \leq R$, then $y$ is non-zero.
\end{lem}

\begin{proof}
Using \cref{rem:CompletionsAndCauchySequences}, there exists a Cauchy sequence $(y_0,y_1,\dots)$ such that the Cauchy sequence $(d(y_0),d(y_1),\dots)$ is equivalent to the Cauchy sequence $\lambda \cdot (x,x,\dots)$ for some $\lambda \in \Lambda_{\geq 0}$.  So for each $R \in \IR_{\geq 0}$, there exists $N_R$ such that for all $n \geq N_R$, $\val(\lambda \cdot x - d(y_n)) > R$.  We take $y = y_{N_R}$.
\end{proof}

\begin{lem}\label{lem:TelKillingLemma}
Let $\sA$ be a ray of chain complexes over $\Lambda_{\geq 0}$ with $x \in A_0$ closed.  If $[\widehat{x}]$ maps to zero under
\[H(\widehat{A_0}) \otimes \Lambda \to H(\widehat{\Tel}(\varphi)) \otimes \Lambda,\]
then there exists $\lambda \in \Lambda_{\geq 0}$ such that for each $R\geq 0$ there exists $y \in A_n$ (for some $n$ dependent on $R$) with
\[\val(\lambda \cdot \alpha_{n-1} \circ \cdots \circ \alpha_0(x) -d(y)) > R.\]
So if $\val(\lambda \cdot \alpha_{n-1} \circ \cdots \circ \alpha_0(x)) \leq R$, then $y$ is non-zero.
\end{lem}

\begin{proof}
By \cref{lem:KillingLemma}, there exists ${\lambda} \in \Lambda_{\geq 0}$ such that for each $R \geq 0$, there exists $\widetilde{y} \in \Tel(\sA)$ with
\[ \val({\lambda} \cdot x - d(\widetilde{y})) > R,\]
where we have used $x$ to denote the image of $x$ in $\Tel(\sA)$.  This inequality is realized in, say, $F^n \Tel(\sA)$.  Write
\[ \widetilde{y} = (y_0,y_0',\dots,y_{n-1},y_{n-1}',y_n) \in F^n \Tel(\sA).\]
Applying the map $F^n \Tel(\sA) \to A_n$,
\[ \val \left( {\lambda} \cdot \alpha_{n-1} \circ \cdots \circ \alpha_0(x) - d \left( \sum_{i=0}^{n} \alpha_{n-1} \circ \cdots \circ \alpha_i(y_i) \right) \right) > R.\]
So we set
\[y = \sum_{i=0}^{n} \alpha_{n-1} \circ \cdots \circ \alpha_i(y_i).\]
\end{proof}


\section{An integrated maximum principle for convex symplectic domains}\label{sec:WeaklyConvexDomains}

We establish a maximum principle for compact symplectic manifolds $(M,\Omega)$ with boundaries of a slightly weaker form than those of convex symplectic domains.

\begin{defn}\label{defn:WeaklyConvexSymplecticDomain}
A compact symplectic manifold with boundary $(M,\Omega)$ is a \emph{weakly convex symplectic domain} if
\begin{enumerate}
	\item there exists a collar neighborhood $N(\partial M)$ of the boundary that is symplectomorphic to the (bottom half of the) symplectization of a stable Hamiltonian structure $(\Omega|_{\partial M}, \alpha)$ on $\partial M$ for some $1$-form $\alpha$ on $\partial M$, and
	\item there exists an $\Omega$-compatible almost complex structure $J$ on $M$ such that on $N(\partial M)$,
	\begin{enumerate}
	\item $-df \circ J = \lambda$ for some function $f$ (dependent on $J$) in $r$ with $f'(r) > 0$, and
	\item $d\alpha|_\xi(\cdot, J \cdot)$ is semi-positive definite on $\xi$.
	\end{enumerate}
\end{enumerate}
Using the identification with the symplectization, the \emph{radial/collar coordinate} is 
\[r: N(\partial M) \cong (0,1] \times \partial M \to (0,1].\]
The associated \emph{Liouville $1$-form} on $M$ is $\lambda = r \cdot \alpha$.  Denote a weakly convex symplectic domain by $(M,\Omega,\lambda)$.  Finally, an almost complex structure $J$ as above is \emph{weakly admissible} for the tuple $(M,\Omega,\lambda)$.
\end{defn} 

\begin{rem}
A convex symplectic domain is a weakly convex symplectic domain.  We could have phrased \emph{all} of our Floer theoretic constructions, arguments, and properties in terms of weakly convex symplectic domains.  In particular, we could define action completed symplectic cohomology for these domains.  The additional assumptions for convex symplectic domains are not (vitally) used in any of our constructions or arguments. We introduce these weakly convex symplectic domains now because we need this weaker notion of convexity and its associated integrated maximum principle when we deform our rescaled almost complex structure in our rescaling argument, see \cref{sec:RescalingIso}.
\end{rem}

The usage of some flavor of the convexity is standard in pseudo-holomorphic curve theory since it gives maximum principles for pseudo-holomorphic curves.  Our's is no exception.

\begin{notn}\label{notn:MaximumPrincipleSetup}
\begin{enumerate}
	\item Let $(M,\Omega,\lambda)$ be a weakly convex symplectic manifold.
	\item Let $J_s$ be an $\IR$-family of admissible almost complex structures for this tuple.  We assume that $-df \circ J_s = \lambda$ for a single function $f$ (as opposed to some family $f_s$).
	\item Let $H: \IR \times S^1 \times M \to \IR$ be a family of Hamiltonians such that $H_s|_{N(\partial M)} = h_s$, where $h_s$ is a family of radial functions in $r$ that satisfies $\partial_s \partial_r h_s \leq 0$.
	\item Let $u: \IR \times S^1 \to M$ be a smooth map that satisfies
	\[0 = (du-X_{H} \otimes dt)^{0,1}.\]
\end{enumerate}
\end{notn}

We now give our integrated maximum principle for weakly convex symplectic domains.  It is inspired by the integrated maximum principle in \cite{AbouzaidSeidel_OpenStringViterbo}.  In our case, the exterior derivative of the Liouville $1$-form need not be the symplectic form and so some variant of the argument in \cite{AbouzaidSeidel_OpenStringViterbo} is required.

\begin{prop}\label{prop:MaximumPrincipleForWeaklyConvexDomains}
If
\[ \lim_{s \to \pm \infty} r \circ u(s,t) = c_{\pm} < 1, \]
then $r \circ u \leq \max(c_\pm)$.  Moreover, if $\partial_r \partial_r h_s \geq 0$, then $r \circ u \leq c_+$, and if $c_- = c_+$, then $r \circ u$ is constant.
\end{prop}

\begin{proof}
For the first statement, since $f$ satisfies $f'(r) > 0$, it suffices to show if $\lim_{s \to \pm \infty} f \circ u(s,t) \leq f(c_\pm)$,
then $f \circ u \leq \max(f(c_\pm))$.  As in \cite{Seidel_ABiasedViewOfSymplecticCohomology}, one finds 
\begin{align*}
\Delta(f \circ u) & = -d(df \circ du \circ j) \geq \frac{-r\cdot \partial_r \partial_r h_s}{f'(r)} \cdot \partial_s(f \circ u).
\end{align*}

So $f \circ u$ satisfies a maximum principle, giving the first claim.

For the second part, assume by way of contradiction that $c_- > c_+$.  So $f \circ u \leq f(c_-)$.  For each fixed $s$,
\begin{align*}
\int_{0}^1 \partial_s(f \circ u)(s) \ dt & = \int_0^1 \lambda(\partial_t u - X_H)(s) \ dt
\\ & = \int_0^1 \lambda(\partial_t u)(s) - (\partial_rh_s(u) \cdot r(u))(s) \ dt
\\ & \geq \int_0^1 r \circ u(s) \cdot \partial_rh_s(c_-) - (\partial_r h_s(u) \cdot r(u))(s) \ dt
\\ & \geq 0
\end{align*}
The first inequality follows from Stokes' Theorem, and the positivity of $u^* d \alpha$.  The second inequality follows from the assumptions that $\partial_s \partial_r h_s \leq 0$, $\partial_r \partial_r h_s \geq 0$, and the fact that $\partial_s(r \circ u)$ must be non-positive near $\lim_{s \to -\infty} u(s,t)$.  So the average value of $\partial_s(f \circ u)(s)$ on each $s$-slice of the domain of $u$ is positive.  Thus, $f \circ u$ must be constant, which implies that $c_- = c_+$, a contradiction.

For the final statement, if $c_- = c_+$, then the above argument (does not lead to a contradiction but instead) implies that $r\circ u$ is constant, as desired.
\end{proof}


\section{Stable displaceability of neighborhoods of divisors}\label{sec:DisDiv}

The reader should recall the definition of stably displaceability from \cref{defn:StablyDisplaceable}.  Here we will discuss and prove the following.

\begin{lem}\label{lem:StDispFibres}
Let $\pi: M \to \IC$ be a proper surjective morphism of smooth, quasi-projective varieties.  Given a K\"{a}hler form $\Omega$ on $M$, any fibre of $\pi$ is stably displaceable in $(M,\Omega)$.
\end{lem}

\begin{lem}\label{lem:FibreToNhood}
Let $(M,\Omega)$ be a symplectic manifold with a proper surjective map $\pi: M \to \IC$ that is submersive on a punctured neighborhood of the origin.  If $\pi^{-1}(0)$ is stably displaceable in $(M,\Omega)$, then there exists $\varepsilon>0$ such that $\pi^{-1}(\ID_\varepsilon)$ is stably displaceable in $(M,\Omega)$.
\end{lem}

Combining \cref{lem:StDispFibres} and \cref{lem:FibreToNhood}, we obtain the following.

\begin{cor}\label{lem:DisplacingFibres}
Let $\pi: {M} \to \IC$ be a proper surjective morphism of smooth, quasi-projective varieties.  Given a K\"{a}hler form $\Omega$ on $M$, there exists $\varepsilon>0$ such that $\pi^{-1}(\ID_\varepsilon)$ is stably displaceable in $(M,\Omega)$.
\end{cor}

\begin{proof}
By \cref{lem:StDispFibres}, $\pi^{-1}(0)$ is stably displaceable in $(M,\Omega)$.  Since $\pi$ is surjective, the critical values of $\pi$ are isolated.  So $\pi$ is submersive in a punctured neighborhood of the origin.  Applying \cref{lem:FibreToNhood} gives the desired result.
\end{proof}

\cref{lem:StDispFibres} follows immediately from a result of McLean:

\begin{thm}\label{thm:SubvarietiesOfKahlerStablyDisplaceable}\cite[Corollary 6.21]{McLean_BirationalCalabiYauManifoldsHaveTheSameSmallQuantumProducts}
Every proper, positive codimension subvariety of a K\"{a}hler manifold is stably displaceable.\qed
\end{thm}

\cref{lem:FibreToNhood} follows from point-set topology.

\begin{lem}\label{lem:ExtendingOpenMaps}
Let $f: Y \to X$ be a surjective proper map of metric spaces.  Let $U \subset Y$ be an open subset and suppose $D \subset U$ with $f^{-1}(f(D)) = D$.  If $f(U \smallsetminus D)$ is open, then $f(U)$ is open.
\end{lem}

\begin{proof}
First, if $x \in f(U) \smallsetminus f(D) \subset f(U \smallsetminus D)$, then since $f(U \smallsetminus D)$ is open there exists an open subset $V \subset f(U \smallsetminus D) \subset f(U)$ that contains $x$.  Consequently, $x$ is an interior point of $f(U)$.

Second, suppose that $x \in f(D)$.  Suppose by way of contradiction that it is not an interior point of $f(U)$.  There must exist $x_n \in B_{1/n}(x) \smallsetminus f(U)$ for each $n \in \IN$.  Since $f$ is surjective, there exists $y_n \in f^{-1}(B_{1/n}(x)) \smallsetminus U$.  Since $f$ is proper, the $y_n$ are contained in the compact subset $f^{-1}(\overline{B_1(x))}$.  Passing to a subsequence and relabeling, the sequence $y_n$ converges to some point $y$.  By continuity, $f(y) = x$ and thus $y \in D$.  Since $U$ is an open subset containing $D$, we must have that $y_n \in U$ for $n\gg 0$, a contradiction.  Consequently, each $x \in f(D)$ is an interior point of $f(U)$.  So all the point of $f(U)$ are interior points.
\end{proof}

We prove \cref{lem:FibreToNhood}.

\begin{proof}
By assumption, $\pi^{-1}(0)$ is stably displaceable.  So  $\pi^{-1}(0) \times S^1 \subset M \times T^* S^1$ is Hamiltonian displaceable.  Since $\pi^{-1}(0) \times S^1 \subset M \times T^* S^1$ is a compact subset, it admits an open neighborhood which is Hamiltonian displaceable.  After shrinking said open neighborhood, assume it is a product $U_1 \times U_2 \subset M \times T^* S^1$.  By \cref{lem:ExtendingOpenMaps}, $\pi(U_1)$ is open.  So find $\varepsilon > 0$ such that $\ID_\varepsilon \subset \pi(U_1)$.  Now $\pi^{-1}(\ID_\varepsilon) \times S^1$ is contained in the Hamiltonian displaceable subset $U_1 \times U_2$ and thus $\pi^{-1}(\ID_\varepsilon)$ is stably displaceable.
\end{proof}

\bibliographystyle{alpha}
\bibliography{./sections/References.bib}

\end{document}